\documentclass[10pt,a4paper]{article}

\usepackage[utf8]{inputenc}
\usepackage{TL}
\usepackage{pgfplots}
\pgfplotsset{compat=1.18}

\usepackage{tikz}
\usetikzlibrary{shapes.geometric, arrows, positioning}

\DeclareMathOperator{\Var}{Var}
\newcommand{\brw}{\textsc{brw}\xspace}
\newcommand{\brws}{\textsc{brw}s\xspace}
\newtheorem{assumptions}{Assumptions}

\newcommand{\wM}{\widetilde{M}}
\renewcommand{\kx}{s}
\newcommand{\repscheme}{reproduction scheme\xspace}
\newcommand{\repschemes}{reproduction schemes\xspace}

\author{\textsc{Thomas Lehéricy}\footnote{Universität Zürich}}
\title{Maximum displacement of critical centered branching random walks under minimal assumptions}

\begin{document}

\maketitle

\begin{abstract}
We study the critical centered branching random walk, with offspring and displacement distributions having finite variance, under minimal assumptions on its structure. We show that the probability that the position of the right-most particle is larger than $r$ decays like an explicit constant times $r^{-2}$; this generalizes an earlier result by Lalley and Shao. In addition, we obtain the convergence in distribution of the progeny of the branching random walk conditioned on the position of the right-most particle being large. Our results are applied to multitype branching random walks under minimal assumptions. 
\end{abstract}

\tableofcontents

\section{Introduction}

A \emph{branching random walk} (\brw{}) models a population of individuals that evolves in discrete time and in space. Start with a single individual at generation 0. Afterwards, at every generation, each individual reproduces independently of the others and then dies, giving birth to a random number of offsprings. Each offspring inherits the position of its parent, plus some random displacement. 
The maximum displacement of a \brw, i.e. the maximum position of a particle of the \brw, is a quantity of interest in the study of populations, see e.g. \cite{sawyer1979maximum}. Its behavior is connected to the Fisher-KPP equation \cite{mckean1975application}. The \brw is related to continuous-time \emph{branching Lévy processes}, where the genealogy of the population is given by a Yule process and where each particle evolves according to a Lévy process. A central example of branching Lévy process is the branching Brownian motion. The behavior of \brws and branching Lévy processes are often very similar, although the methods to study them differ. 

A key criterion that determines the behavior of the \brw is the expected number of offsprings $m$ of a particle. 
The \brw is \emph{subcritical} if $m<1$, \emph{critical} if $m=1$, and \emph{supercritical} if $m>1$. In the subcritical and critical case, the \emph{drift}, i.e. the expected mean displacement of the offsprings of a particle, also matters; we say that the \brw is \emph{centered} if this drift is zero. 



A well-known case is when the \brw is supercritical; in this case, the population diverges towards $+\infty$ with positive probability. \cite{aidekon2013convergence} has given a precise asymptotic for the maximum displacement at generation $n$, see also \cite{bramson1978maximal} for the related case of branching Brownian motion. The behavior of the branching Brownian motion near its maximum has also been described, see e.g. \cite{aidekon2013branching}.

When the population is subcritical or critical, the \brw goes extinct almost surely. 
The \emph{maximum displacement} is the maximal position attained by an individual across all generations. Results vary depending on whether the walk is centered or not; on whether the offspring distribution has finite variance or is in the domain of attraction of a $\gamma$-stable distribution with $\gamma<2$; on whether the distribution of displacements has sufficiently high moments; and on whether the \brw is critical or subcritical. 

The quantity of interest to us is the tail of the maximum displacement, i.e. the asymptotic behavior of the quantity $\P(M\geq t)$ as $t\to\infty$ where $M$ is the maximum displacement of the \brw. Before describing the literature, let us list a few other questions that were considered. \cite{kesten1995branching} shows the convergence in distribution of the maximum displacement at generation $n$ of the critical \brw (with sufficiently high moments on its offspring and displacement distributions) conditioned to survive until $\beta n$ for some $\beta>0$. \cite{neuman2021maximal} study the maximum displacement at a given generation of a slightly supercritical \brw with bounded displacements and offspring distributions having a third moment. 
Another interesting case is presented in \cite{fu2025maximal}, which considers a critical \brw in random environment with sufficiently high moments on the offspring distribution and with Gaussian displacements; the authors reveal a significant difference in the behavior of the maximum displacement compared to the case of constant environment. 
Killed branching Lévy processes, i.e. where particles are killed when they reached the negative half-line, are also a subject of interest \cite{zheng2015limiting, ren2025asymptotic, hou2025tailextinction, hou2025asymptotic}. 

Earlier results on the tail of the maximum displacement are found in \cite{sawyer1979maximum}, which establishes the tail behavior of the maximum displacement of critical and subcritical branching Brownian motion where the offspring distribution has finite third moment. This was generalized in \cite{lalley2015maximal} to centered critical \brw with offspring distribution having finite third moment and displacement distribution having finite $4+\varepsilon$ moment. In virtually all articles dealing with \brw, restrictive assumptions on the reproduction scheme (i.e. the joint distribution of the displacements of the offsprings) are made to simplify computations. The goal of our article is to establish some results of \cite{lalley2015maximal} (and some new ones) under minimal hypotheses. 

Heavy-tailed displacements have mostly been considered in the case of Lévy branching processes. 
Offspring distributions having finite variance are found in 
\cite{lalley2016maximalstable}, which works on a critical branching Lévy process with binary branching where the displacement follows a symmetric stable process with index $\alpha$. They obtain an asymptotic of the tail on the maximum displacement in $\P(M \geq t) \sim (2/\alpha)^{1/2} t^{-\alpha/2}$, see also \cite{profeta2021extreme} when the displacement is a critical or subcritical $\alpha$-stable Lévy process with positive jumps, and \cite{profeta2024maximal} for spectrally negative Lévy processes---more on that later. 
This is generalized in \cite{hou2025tail}, where the offspring distribution is in the domain of attraction of a non-Gaussian stable distribution. More precisely, using our notations from Section \ref{sec:preliminaries}, they assume $\P(\chi(\R)\geq t)\sim \kappa t^{-\alpha}$ for some $\kappa>0$ and $\alpha\in (1,2)$. 

Subcritical branching processes exhibit a different behavior. 
\cite{neuman2017maximal} consider a subcritical \brw where the offspring distribution has finite third moment and the distribution of the displacements has an exponential moment, and shows that the maximum displacement has an exponential tail; the exponential moment on the displacements is crucial. Such a result was already established for subcritical branching Brownian motion in \cite{sawyer1979maximum}. Building upon this, 
\cite{hou2025maximal} consider subcritical \brw under the $L \log L$ condition and with displacement having sufficiently high exponential moments, and determines that the tail of the maximum displacement is exponential or subexponential, establishing a sharp transition between the two regimes. These results also hold when the \brw is killed when reaching the negative half-line. 
\cite{profeta2024maximal} considers critical and subcritical branching spectrally negative Lévy processes, where the offspring distribution has a third moment. They show that the maximum displacement has an exponential tail in the subcritical case, and investigate the tail in the critical case when the Lévy process is centered or has positive or negative drift respectively. Finally, \cite{hou2025tailsub} consider critical branching Lévy processes where the offspring distribution belongs to the domain of attraction of a $\gamma$-stable distribution with $\gamma\in (1,2]$ (more precisely $\P(\chi(\R)>t)\sim \kappa t^{-\gamma}$ for some $\kappa\in (0,\infty)$ if $\gamma<2$, or $\E[\chi(\R)^2]<\infty$ if $\gamma=2$), and subcritical branching Lévy processes where the offspring distribution satisfies the $L \log L$ condition, and where the displacement follows an $\alpha$-stable Lévy process with positive jumps. 

In the rest of the article, we work on a critical, centered \brw with offspring distribution having finite variance and displacement having---essentially---a moment of order 4. The closest works are \cite{sawyer1979maximum} for the branching Brownian motion and \cite{lalley2015maximal} for the \brw. 
While the distribution of the maximum displacement of a critical \brw depends on the \emph{reproduction scheme}, i.e. the joint distribution of the displacements of offsprings, and is not universal, the tail is: \cite{lalley2015maximal} established that the probability that the maximum displacement exceeds $r$ behaves like a constant times $r^{-2}$ as $r\to\infty$, where the constant depends on the \repscheme. In order to keep their proof simple, they restricted the \repschemes of the \brw.

This article establishes that the asympotic of the tail proven in \cite[Theorem 1]{lalley2015maximal} holds for \brw in full generality (Theorem \ref{th:main-tail} and Corollary \ref{cor:main-tail}), most importantly without any assumption on the dependencies in the \repscheme. We aim to use the weakest hypotheses possible (Assumptions \ref{assum}); see Section \ref{sec:minimality} for a discussion on their minimality. While our method follows that of \cite{lalley2015maximal}, significant work is required to establish the key estimates, for example the Feynman-Kac representation (our Corollary \ref{cor:Feynman-Kac}, \cite[Corollary 7]{lalley2015maximal}). In addition, we provide in Lemma \ref{lem:continuity-w-broad} a new estimate that supersedes \cite[Lemma 11]{lalley2015maximal}, the end of which proof contains a circular argument. 


Furthermore, we find in Theorem \ref{th:main-condition} an asymptotic on the probability that the maximum displacement \emph{equals} $r$ as $r\to\infty$, when the displacement takes integer values. 
We then prove in Theorem \ref{th:main-volume-smaller} and \ref{th:main-condition-vol} the convergence in distribution of the total progeny of the \brw, and of related quantities, when we condition the maximum displacement to be large. Our method for these two Theorems makes use of a re-interpretation of the problem as a study of the maximum displacement of an almost-critical subcritical \brw. 
As an application, we present how to generalize our results to multitype \brws (Section \ref{sec:extensions}), and use this to obtain information on the geometry of generic critical bipartite Boltzmann planar maps (Section \ref{sec:mobile}). We then comment in Section \ref{sec:snake} on the connection between the \brw in the regime we study and the Brownian snake. 

Section \ref{sec:main-proof} establishes Theorem \ref{th:main-tail}. A brief overview of our method can be found at the beginning of the Section. Section \ref{sec:proof-tail} is devoted to the proof of Theorem \ref{th:main-condition}, and Section \ref{sec:2Volume-general} to the proof of Theorems \ref{th:main-volume-smaller} and \ref{th:main-condition-vol}. Finally, we gathered a useful technical Lemma and some proofs in the Appendix.

\section{Main results}

\subsection{Preliminaries}
\label{sec:preliminaries}

\paragraph{Planar trees.}

We use Neveu's notation \cite{neveu1986arbres}. Define $\UU := \cup_{n\geq 0} \N_*^n$ the infinite Ulam--Harris tree of finite sequences of natural numbers, where $\N_*^0 = \{\emptyset\}$, $\N_* = \{1, 2, \dots\}$ and $\N = \N_* \cup \{0\}$. If $u,v \in \UU$, denote $uv$ their concatenation and $|u|$ the length of $u$. We say that $u$ is an ancestor of $v$ if and only if there exists $w\in\UU$ such that $v = uw$; if $|w|=1$ then $u$ is the parent of $v$, denoted by $p(v)$. For every $u\in \UU$, the ancestral line of $u$ is the sequence $(u_0 = \emptyset, u_1, \dots, u_{|u|} = u)$ where for every $0\leq j \leq |u|$, $u_j$ is the only ancestor of $u$ with $|u_j|=j$. A tree $T \subset \UU$ is a subset of $\UU$ such that $\emptyset \in T$, and for every $u = p(u)k$ with $k\in \N_*$, $p(u) \in T$ and $p(u) j \in T$ for every $0\leq j \leq k$. 

\paragraph{Bienaymé--Galton--Watson trees and processes.} 
Let $k_v$ be i.i.d. $\N$-valued random variables with common distribution $\mu$. The Bienaymé--Galton--Watson tree with offspring distribution $\mu$ is the tree $T$ such that $\emptyset \in T$, and for every $u\in \UU$, $u \in T$ if and only if $p(u)\in T$ and $u = p(u)j$ with $1\leq j \leq k_{p(u)}$. The number of vertices in $T$ at generation $n$, $N_n := \#\{ v \in \N_*^n, v\in T\}$ follows a Bienaymé--Galton--Watson process. If $\sum_{n\geq 0} n \mu(n) = 1$ we say that the Bienaymé--Galton--Watson tree (resp. process) is critical. Assuming $\mu(1) < 1$, the process goes extinct almost surely, i.e. $N_n \to 0$ almost surely as $n\to\infty$. 

\paragraph{Point processes.}
A (real-valued) point process is a random variable with values in the set $\M$ of all counting measures over $\R$, i.e. the set of all measures $\mu$ such that $\mu(A) \in \N \cup \{\infty\}$ for every Borel set $A$, equipped with the sigma-algebra generated by the collection of $\mu \mapsto \mu(A)$ over all Borel sets $A$. 
We let $\sup \mu$ be the supremum of the support of $\mu$, with $\sup\mu = -\infty$ if $\mu = 0$. 
A point process $\chi$ is finite if $\chi(\R)<\infty$ almost surely; in this case there exists a measurable enumeration $(X_i)_{1\leq i \leq \chi(\R)}$ of the atoms of $\chi$, i.e. a family of random variables such that $\chi = \sum_{1\leq i \leq \chi(\R)} \delta_{X_i}$ with $\delta_x$ the Dirac mass at $x$ --- for example, we can enumerate the atoms in increasing order. 

The mean measure of a point process $\chi$ is the measure $M : A \mapsto \E[\chi(A)]$ for every Borel set $A$. Following \cite{baccelli2009stochastic}, define the measure $C$ on $\R\times \M$ by 
\[ C(A \times \Gamma) = \E\left[ \int_A \ind{\chi - \delta_x \in \Gamma} \d \chi(x) \right] . \]
Since $C(\cdot , \Gamma)$ is absolutely continuous with respect to $M$ for every $\Gamma$, we have 
\[ C(A\times \Gamma) = \int_A P_x(\Gamma) \d M(x) . \]
When $M$ is locally finite, the function $(x,\Gamma) \mapsto P_x(\Gamma)$ can be made into a stochastic kernel; we call $P_x$ the reduced Palm distribution of $\chi$ at $x$. We then have the reduced Campbell--Little--Mecke formula: for every positive and measurable, resp. bounded and measurable $f$, 
\begin{equation}\label{eq:Campbell-Little-Mecke}
\E\left[\int f(x, \chi-\delta_x) \d \chi(x) \right] = \int f(x, \phi) \d P_x(\phi) \d M(x) .
\end{equation}
A point process distributed under $P_x$ has itself a mean measure, which we denote by $M_x$. For every positive and measurable, resp. bounded and measurable $f$,  
\begin{equation}\label{eq:Campbell-Little-Mecke-mean}
\E\left[\int \left(\int f(x,y) \d(\chi-\delta_x)(y)\right) \d\chi(x) \right] = \int \left( \int f(x, y) \d M_x(y) \right) \d M(x) .
\end{equation}

\paragraph{Decorated Branching Random Walk.} 
Let $\BB := (\chi, \Lambda, D, (X_i)_{1\leq i \leq \chi(\R)})$, where $\chi$ is a point process, $\Lambda\geq 0$ a random variable with $\Lambda\geq \sup\chi$ a.s., $D\geq 0$, and $(X_i)_{1\leq i \leq \chi(\R)}$ is a measurable enumeration of the atoms of $\chi$.  We define the decorated \brw $(T, (D_v)_{v\in T}, (X_v)_{v\in T}, (\Lambda_v)_{v\in T})$ with \emph{reproduction scheme} $\BB$ as follows. Let $(\chi_u, \Lambda^{(u)}, D_u, (X^{(u)}_i)_{1\leq i \leq \chi_u(\R)})_{u\in\UU}$ be i.i.d. copies of $\BB$. For every $u\in\UU$, let $k_u := \chi_u(\R)$, and let $T$ be the tree constructed from the $(k_u)_u$ (it is a Bienaymé--Galton--Watson tree). Define $X_\emptyset = 0$, $\Lambda_\emptyset = \Lambda^{(\emptyset)}$, and for every $u = p(u) j \in T$, define $X_u = X_{p(u)} + X^{(p(u))}_j$ and $\Lambda_u = X_u + \Lambda^{(u)}$. 
The joint process $(T, (D_u)_{u\in T}, (X_u)_{u\in T}, (\Lambda_u)_{u\in T})$ defines the decorated \brw with \repscheme $\BB$. 
We write $G_n := \sum_{u\in T, |u|=n} \delta_{X_u}$ for the point process of its vertices at generation $n$.

\subsection{Statement of the theorems}
\label{sec:main-statements}

In this article, we make the following assumptions:

\begin{assumptions}\label{assum}
\begin{enumerate}
\item the \brw is critical: $\E[\chi(\R)] = M(\R) = 1$,
\item the \brw is centered: $\E\left[ \int x \d\chi(x)\right] = \int x \d M(x) = 0$,
\item $\E[\chi(\R)^2] = \sigma^2 + 1 \in (1, \infty)$, 
\item $\E\left[\int x^2 \d\chi(x)\right] = \int x^2 \d M(x) = \eta^2 \in (0, \infty)$,
\item $\P(\Lambda>r) = o(r^{-4})$ as $r\to\infty$, 
\item $\E[D] = 1$. 
\end{enumerate}
\end{assumptions}

Note that 3. and 4. imply that the \repscheme is non-degenerate, i.e. that $\P(\chi(\R) = 1) < 1$ and $M(\{0\}) < 1$. Point 5. holds if $\E[\Lambda^4]<\infty$, see Lemma \ref{lem:refined-Markov}. Point 6. is here to simplify notations and could be replaced by ``$D$ is an integrable, non-negative and not a.s. zero random variable'' up to an extra constant in the statements of Theorems \ref{th:main-volume-smaller} and \ref{th:main-condition-vol}.

We are interested in studying the distribution of its maximum displacement $\sup_{v\in T} X_v$, or more generally of its maximum decoration $\sup_{v\in T} \Lambda_v$. More precisely, we establish the following results.

\begin{theorem}\label{th:main-tail}
Under Assumptions \ref{assum}, asymptotically as $r\to\infty$, 
\[ \P\left(\sup_{v\in T} \Lambda_v > r\right) \sim \frac{6\eta^2}{\sigma^2 r^2} . \]
\end{theorem}

Taking $\Lambda = \sup(0, \sup\chi)$, we get the immediate corollary. 

\begin{corollary}\label{cor:main-tail}
Under Assumptions \ref{assum}, asymptotically as $r\to\infty$, 
\[ \P\left(\sup_{v\in T} X_v > r\right) \sim \frac{6\eta^2}{\sigma^2 r^2} . \]
\end{corollary}

This corollary is a generalization of \cite[Theorem 1]{lalley2015maximal}, which is established for the specific \repscheme $\chi = L \delta_X$ with $X$ in $L^{4+\varepsilon}(\Z)$, $\E[X]=0$, $\E[L^3]<\infty$ and $\E[L]=1$ (and which can be extended to $X\in L^4$ and $L\in L^2$ with minimal effort). We extend it to very general \repschemes, with a \brw that is not necessarily in $\Z$, and with minimal moment assumptions. 

In addition, we prove a range of new results. The first one considers the probability that the maximum of the decorations takes a specific value, when the \brw is $\Z$-valued. The maximum span of a random variable is the largest $d$ such that there exists $x$ such that the support of the random variable is contained in $x+d\Z$.

\begin{theorem}\label{th:main-condition}
Under Assumptions \ref{assum}, if $M$ is supported on $\Z$ and has maximal span $1$ and if $\Lambda$ is $\N$-valued, then asymptotically as $r\to\infty$ with $r\in \N$,  
\[ \P\left(\sup_{v\in T} \Lambda_v = r\right) \sim \frac{12 \eta^2}{\sigma^2 r^3} . \]
\end{theorem}

It is likely that such an asymptotic could be established for displacements that have a density with respect to the Lebesgue measure, but we chose not to investigate this. 

The next results establish the asymptotic distribution of the total weight (the sum of $D_v$ over all $v$ in the tree) in the \brw when we condition it on its maximum decoration. This naturally generalizes to multivariate weights: simply let $D_v$ be the scalar product of the multivariate weight against a fixed vector. These Theorems highlight that the behavior of the total weight is dictated by the total progeny of the \brw, while the distribution of the weight, and a fortiori how it depends of the rest of the \repscheme, plays a marginal role.

\begin{theorem}\label{th:main-volume-smaller}
Let $\psi$ be the unique positive solution of
\[ \psi'' = \psi^2 + \psi \]
with $\psi(x)\to\infty$ as $x\to 0$ and $\psi(x)\to 0$ as $x\to\infty$. For every $\alpha>0$, define 
\[\RR(\alpha) = \frac{6 \eta^2}{\sigma^2} \left(\frac{2\alpha}{\sigma^2} \psi\left( \sqrt{\frac{12\alpha}{\sigma^2}}\right) - 1 + \frac{\alpha}{\sigma^2} \right).\]
Then $\RR$ can be extended to $0$ by continuity with a series expansion as $\alpha\searrow 0$
\[ \frac{\sigma^2}{6\eta^2} \RR(\alpha) = \frac{3}{5}\left(\frac{\alpha}{\sigma^2}\right)^2 - \frac{2}{7}\left(\frac{\alpha}{\sigma^2}\right)^3 + \frac{3}{25}\left(\frac{\alpha}{\sigma^2}\right)^4 + \dots  \]
and under Assumptions \ref{assum}, 
if $t(r,\alpha) = \frac{\alpha^2}{2\sigma^2} \left(\frac{6\eta^2}{\sigma^2 r^2}\right)^2$, 
\[ r^2\left( 1 - \E\left[ \e^{-t(r,\alpha)\sum_{v\in T} D_v} \, | \, \sup_v \Lambda_v \leq r \right]\right) \ulim r \infty \RR(\alpha) \]
and
\[ \E\left[ \e^{-t(r,\alpha)\sum_{v\in T} D_v} \, | \, \sup_v \Lambda_v > r \right] \ulim r \infty \frac{2\alpha}{\sigma^2}\psi\left( \sqrt{\frac{12\alpha}{\sigma^2}} \right) . \]
\end{theorem}

We refer to Appendix \ref{sec:appendix-EDO} for the series expansion of $\psi$, from which we derived that of $\RR$.

\begin{theorem}\label{th:main-condition-vol}
Under Assumptions \ref{assum}, 
if $M$ is supported on $\Z$ and has maximal span $1$ and if $\Lambda$ is $\N$-valued, then 
\[ \E\left[ \e^{-t(r,\alpha)\sum_{v\in T} D_v} \, | \, \sup_v \Lambda_v = r \right] \ulim r \infty \frac{2\alpha}{\sigma^2}\psi\left( \sqrt{\frac{12\alpha}{\sigma^2}} \right) . \]
\end{theorem}

\subsection{On the minimality of our hypotheses}
\label{sec:minimality}

Let us explain why assumptions \ref{assum} are essentially minimal. First, the dependency structure between the displacement of the offsprings of any vertex can be arbitrary, which generalizes previous results \cite{lalley2015maximal} significantly. This generalizability was already foreseen by the authors of \cite{lalley2015maximal}, who 
chose to restrict themselves to a technically convenient framework. 

As already observed in \cite{lalley2015maximal} in the case $\chi = N \delta_X$ with $N$ and $X$ independent, there are counter-examples when $X \notin L^4$. In our case, this counterexample covers the case where $t^4\P(\Lambda>t)$ does not converge to $0$ as $t\to\infty$. 

The condition $\E[\int x^2 \d\chi(x)] < \infty$ is also required; \cite{lalley2016maximalstable} considers a case where the displacement is in the domain of attraction of a stable distribution, in particular where this quantity is infinite, and deduce asymptotics that differ from ours. When $\E[\int x \d\chi(x)] \neq 0$ we also expect a different behavior, namely: if it is positive, then the probability in Theorem \ref{th:main-tail} should behave like $r^{-1}$, and if it is negative then the tail should be much lighter. 

Clearly the fact that $\E[\chi(\R)]=1$ is important; if this was $<1$ then the the probability in Theorem \ref{th:main-tail} would decrease exponentially fast (see e.g. \cite{hou2025maximal}), while the supremum is infinite if this is $>1$. The fact that $\E[\chi(\R)^2] <\infty$ is also important, see \cite{hou2025tail} for the case of Lévy branching processes.

The fact that we consider decorated \brws allows for a unified approach to a range of related models. For example, we are able to handle multitype \brws with little extra work, as we illustrate in the Section \ref{sec:extensions}.

\subsection{Multitype branching random walks}
\label{sec:extensions}

A natural generalization are multitype \brws. In multitype \brws, each vertex has a type, and its reproduction scheme depends on its type. Multitype \brws appear naturally in the field of random planar maps, for example as the image of bipartite Boltzmann maps by the Bouttier--Di Francesco--Guitter bijection \cite{miermont2006invariance,stephenson2018local}. 

The scaling limit of multitype Bienaymé--Galton--Watson trees has been investigated 
in the case of finite number of types \cite{miermont2008invarianceGW} and of infinite number of types \cite{raphaelis2017scaling}. In this section, we explain how to extend our results to the case of \brws with finitely many types; we follow most of the notations of \cite{miermont2008invarianceGW}. We do not handle the case of infinitely many type in order to reduce technicalities to a minimum; we do not expect to find obstacles to applying the same idea to \brws with infinitely many types.

Theorem 3 in \cite{miermont2008invarianceGW} establishes the convergence of the multitype \brw towards the Brownian snake. The assumption requires exponential moments on the offspring distribution (in our monotype setting, this is the distribution of $\chi(\R)$) and ---roughly speaking--- moments of order $8+\varepsilon$ on the displacement for some $\varepsilon>0$ (in our monotype setting this would mean $\E[\int x^{8+\varepsilon} \d\chi(x)]<\infty$). The authors conjectured in the remark after \cite[Theorem 4]{miermont2008invarianceGW} that this assumption could be reduced to a moment of order $4+\varepsilon$. We plan to extend in a future paper this result to the case of moments of order 2 on the offspring distribution and of order 4 on the displacement. 

Let us make our setting more precise. 
A multitype planar tree is a tree $T\subset \UU$ together with a type function $\bs:T\to \kX$, where $\kX$ is a set of types (in our case finite). Assume that for every $x\in\kX$ we have a point process $\chi^{(x)}$ on $\R\times\kX$, together with a measurable numbering $(X^{(x)}_i, \kx^{(x)}_i)_{1\leq i \leq \chi^{(x)}(\R\times\kX)}$ of its atoms. For every $y\in\kX$, the point process $\chi^{(x)}(\cdot \times \{y\})$ represents the point process of offsprings that are of type $y$, and for every $1\leq i \leq \chi^{(x)}(\R\times\kX)$, $X^{(x)}_i$ is the displacement the $i$-th child respective to its parent, and $\kx^{(x)}_i$ is its type. 

Given a multitype \repscheme $\BB := (\chi^{(x)}, \Lambda^{(x)}, D^{(x)}, (X^{(x)}_i, \kx^{(x)}_i)_{1\leq i \leq \chi^{(x)}(\R\times\kX)})_{x\in\kX}$, where $D^{(x)}\geq 0$, $\Lambda^{(x)}\geq \sup\chi^{(x)}$ a.s. for every $x\in\kX$ and $(X^{(x)}_i, \kx^{(x)}_i)_{1\leq i \leq \chi^{(x)}(\R\times\kX)}$ is a measurable numbering of $\chi^{(x)}$, 
we define the multitype decorated \brw $(T, \bs, (D_v)_{v\in T}, (X_v)_{v\in T}, (\Lambda_v)_{v\in T})$ with \repscheme $\BB$, where $T\subset \UU$, $\bs:T\to\kX$, $D_v\geq 0$, and $X_v, \Lambda_v \in \R$ for every $v\in T$, and the probabilities $(\P_{x})_{x\in\kX}$ with associated expectations $(\E_x)_{x\in\kX}$, by induction as follows. 
For every $\kx_\emptyset\in\kX$, under $\P_{\kx_\emptyset}$, let $\left((\chi^{(x)}_u, \Lambda^{(x)}_u, (X^{(x)}_{u,i}, \kx^{(x)}_{u,i})_{1\leq i \leq \chi^{(x)}_u(\R\times\kX)})_{x\in\kX}\right)_{u\in \UU}$ be an i.i.d. family with common distribution the same as $\BB$. Then $\emptyset \in T$, $\bs(\emptyset) = \kx_\emptyset$, and $X_\emptyset = 0$, and for every $n\geq 1$ and $u\in\UU$ with $|u|=n$ and $u = vw$ with $v=p(u)$,  
\begin{itemize}
\item $u\in T$ if and only if $1\leq w \leq \chi^{(\bs(v))}_v(\R\times\kX)$,
\item $\bs(u) = \kx^{(\bs(v))}_{v, w}$, 
\item $X_u = X_v + X^{(\bs(v))}_{v,w}$. 
\end{itemize}
In addition, for every $u\in T$, we define
\begin{itemize}
\item $\Lambda_u = X_u + \Lambda^{(\bs(u))}_u$,
\item $D_u = D^{(\bs(u))}_u$. 
\end{itemize}
When $\kX$ is a singleton, we recover the monotype decorated \brw. 
Write $\zeta_x$ for every $x\in\kX$ for the distribution of $(\kx^{(x)}_i)_{1\leq i \leq \chi^{(x)}(\R\times \kX)}$; if we forget the spatial components $(X_v)_{v\in T}$ and $(\Lambda_v)_{v\in T}$, the multitype tree $(T,\bs)$ follows the distribution of a multitype Bienaymé--Galton--Watson tree with offspring distributions $(\zeta_x)_{x\in\kX}$. 

An approach used among others in \cite{miermont2008invarianceGW, raphaelis2017scaling} for multitype Bienaymé--Galton--Watson trees and forests with finitely many, resp. countably infinitely many types, consists in building a ``reduced'' tree (resp. forest) that only contains vertices of a particular type, see \cite[Section 1.3]{raphaelis2017scaling}. 
Provided the \repscheme mixes ``sufficiently well'' between types, the reduced tree is ``close'' to the original tree. 
This approach naturally extends to \brw, with the reduced \brw being monotype and thus falling within the scope of Theorems \ref{th:main-tail} to \ref{th:main-condition-vol}. The question then reduces to determining whether the reduced \brw satisfies Assumptions \ref{assum}. This can be provided in part by \cite{miermont2008invarianceGW, raphaelis2017scaling}; for example, Proposition 2 in \cite{raphaelis2017scaling} ensures that the reduced \brw satisfies 1. and 3. in Assumptions \ref{assum}. We now aim to define the reduction, and to provide conditions under which Assumptions \ref{assum} hold for the reduction.

Assume that the process is non-degenerate in the sense of \cite{miermont2008invarianceGW}, i.e. that there exists at least one $x\in\kX$ such that $\P(\chi^{(x)}(\R\times\kX) = 1)<1$. 
Define for every $x,y\in\kX$
\begin{equation}\label{eq:def-M}
M_{x,y} = \E\left[\chi^{(x)}(\R \times \{y\}) \right]
\end{equation}
the average number of type $y$ offspring of a type $x$ vertex, and assume that $M_{x,y} < \infty$ for every $x,y$. We see them as the $(x,y)$ entry of matrix $M$, whose iterates we write $M^n = (M^n_{x,y})_{x,y\in\kX}$. We suppose that $M$ is irreducible, i.e. for every $x,y$ there exists $n\geq 1$ such that $M^n_{x,y}>0$. By the Perron-Frobenius theorem $M$ has a largest eigenvalue $\rho$ and a left eigenvector $\ba = (a_x)_{x\in\kX}$ and a right eigenvector $\bb = (b_x)_{x\in\kX}$ associated to this eigenvalue, both unique up to multiplication by a constant and both with positive entries. 

\begin{remark}
In many cases we can handle a non-irreducible $M$, in particular when there are subsets of the space of of states that are transient (hence not visited if we start from a recurrent $x$). Naturally we can discard any state that cannot be reached from the initial one, i.e. any $y$ with $M^n_{x,y} = 0$ for every $n$. We can also often handle the presence of multiple communicating classes, i.e. subsets $S$ of states such that for every $y,z\in S$ we have $M^n_{y,z}>0$ for some $n$. For every communicating class we can restrict the matrix $M$ to it. The restriction is irreducible. Then $M$ will have maximal eigenvalue $1$ if and only if for at least one of its communicating classes, the restriction of $M$ on this class has maximal eigenvalue $1$ (such a class is ``critical’’); and if the restriction of $M$ on each of its communicating classes has eigenvalue at most $1$ (if it is $<1$ the class is ``subcritical’’). If $M$ has only one critical communicating class then we can incorporate the other classes into the decorations $D$ and $\Lambda$ of the states of the unique critical communicating class, and define a reduced \brw with state space the critical class. In this case it is necessary that the state through which we reduce belongs to the critical class. 
\end{remark}

Fix $x\in\kX$ and work under $\E_x$, so that $\bs(\emptyset) = x$: we will construct a reduced \brw that ``contains only the vertices of type $x$ of the initial \brw{}''. For every $u,v\in\UU$, write $u \preceq v$ if $u$ is an ancestor of $v$ (i.e. if there exists $w\in\UU$ such that $v = uw$) and $u\prec v$ if $u\preceq v$ and $u\neq v$. 
Following the idea of \cite[Definition 1]{raphaelis2017scaling}, for every $u\in\UU$ define random trees 
\[ \TT_u = \{ v\in T : \forall u \prec w \prec v, \bs(w) \neq x \} \quad , \quad
 \widetilde\TT_u = \{v\in T: \forall u \preceq w \prec v, \bs(w) \neq x\} , \]
subsets
\[ 
\partial\TT_u = \{ v\in \TT_u: \bs(v) = x, v\neq u\} \quad , \quad \partial\widetilde\TT_u = \{v\in\widetilde\TT_u : \bs(v)=x\} 
\]
ordered in lexicographic order, 
\[ \TT_u^\circ = \TT_u \setminus \partial\TT_u \quad , \quad \widetilde\TT_u^\circ = \widetilde\TT_u \setminus \partial\widetilde\TT_u , \]
and point processes 
\[
 \LL_u = \sum_{v\in\partial\TT_u, v\neq\emptyset} \delta_{X_v} \quad , \quad \widetilde\LL_u = \sum_{v\in\partial\widetilde\TT_u} \delta_{X_v} .
\]
Note that if $\bs(u) \neq x$ and $u\in T$ then $\TT_u = \widetilde\TT_u$ and $\partial\TT_u = \partial\widetilde\TT_u$, while if $\bs(u) = x$ and $u\in T$ then $\widetilde{\TT}_u = \partial\widetilde\TT_u = \{u\}$ while $u\notin\partial\TT_u$. 

Define $\BB^r_u = (\LL_u, \sup_{v\in\TT_u^\circ} \Lambda_v, \sum_{v\in\TT_u^\circ} D_v, (X_i)_{i \in \partial\TT_u} )$. 
Clearly $\BB^r_u$ is a deterministic function of $\TT_u$. In addition, we can define a new tree $\T$, a collection $(\B_u)_{u\in\T}$ and a function $t:\T\to T$ by induction with $\emptyset\in\T$, $t(\emptyset) = \emptyset$, and if $u\in\T$ then $\B_u = \BB^r_{t(u)}$, and letting $(v_1, \dots, v_d)$ be the elements of $\partial\TT_{t(u)}$, we let $ui\in \T$ and $t(ui) = v_i$ for every $1\leq i \leq d$. Then the family $(\B_u)_{u\in\T}$ is i.i.d. conditionally on $\T$, and from $\T$ and the family $(\B_u)_{u\in\T}$ we can construct the reduced (monotype) \brw $(\T, (\D_u)_{u\in\T}, (\X_u)_{u\in\T}, (\Lambda^\T_u)_{u\in\T})$, and we can check that 
\[ \sum_{u\in\T} \D_u = \sum_{u\in T} D_u \quad , \quad \sup_{u\in\T} \Lambda^\T_u = \sup_{u\in T} \Lambda_u . \]
This means that Theorems \ref{th:main-tail} to \ref{th:main-condition-vol} will hold for the multitype \brw if we can show that the reduced \brw satisfies Assumptions \ref{assum}.

Recall our assumption that the \repscheme is non-degenerate and that the matrix $M$ from \eqref{eq:def-M} is irreducible and has finite coefficients. For every $y,z\in\kX$ let
\[ \wM_{y,z} = 
\begin{cases}
M_{y,z} &\text{ if } y\neq x\\
0 &\text{ if } y=x 
\end{cases}
\quad , \quad 
N_{y,z} = \E\left[\int t \ind{\kx=z} \d\chi^{(y)}(t,\kx)\right] 
\quad , \quad
O_{y,z} = \E\left[\int t^2 \ind{\kx=z} \d\chi^{(y)}(t,\kx)\right] 
\]
and recall that $\bb$ is the right eigenvector of $M$ associated to its largest eigenvalue. We will check in the proof of Proposition \ref{prop:equiv-assum} that $I-\wM$ is invertible, where $I$ is the identity.

\begin{assumptions}\label{assum-multitype}
\begin{enumerate}
\item the matrix $M$ has maximal eigenvalue $1$,
\item $\left\{\left(M(I-\wM)^{-1} + I\right) N \bb\right\}_x = 0$, 
\item for every $y\in\kX$, $\E[\chi^{(y)}(\R\times\kX)^2] < \infty$, 
\item for every $y\in\kX$, $\E\left[\int t^2 \d\chi^{(y)}(t,\kx) \right] < \infty$, and
\[ \eta^2 = \left\{ \left(M (I-\wM)^{-1} + I\right) \left( O + N(I-\wM)^{-1} N \right) \bb \right\}_x / b_x \]
is non-zero (this is the variance of the displacement of the reduced \brw), 
\item for every $y\in\kX$, $\P(\Lambda^{(y)}>r) = o(r^{-4})$ as $r\to\infty$, 
\item for every $y\in\kX$, $\E[D^{(y)}]<\infty$, and there exists one $y\in\kX$ with $\E[D^{(y)}]>0$. 
\end{enumerate}
\end{assumptions}

The assumptions are the analogues to their counterpart (number by number) in Assumption \ref{assum}. For example, Assumption \ref{assum-multitype}.2 ensures that the reduced \brw is centered.

\begin{proposition}\label{prop:equiv-assum}
The reduced \brw satisfies Assumptions \ref{assum} (replacing 6. by $\E[\D]\in (0,\infty)$), hence the multitype \brw satisfies Theorems \ref{th:main-tail} to \ref{th:main-condition-vol}, if and only if the multitype \brw satisfies Assumptions \ref{assum-multitype}. 
\end{proposition}

We postpone the proof to Appendix \ref{sec:proof-multitype}. The following Lemma provides sufficient, simpler conditions to check that Assumption \ref{assum-multitype}.2 holds and that $\eta^2$ is non-zero. Its proof is also in Appendix \ref{sec:proof-multitype}.

\begin{lemma}\label{lem:simpler-assum}
\begin{enumerate}
\item Assumption \ref{assum-multitype}.2 holds if $N_{y,z} = 0$ for every $y,z\in\kX$.

\item The condition $\eta^2>0$ in assumption \ref{assum-multitype}.4 holds if for at least one pair $(y,z)\in\kX^2$ the mean measure $A \mapsto \E[\chi^{(y)}(A\times\{z\})]$ is non-zero and not supported on a single point.  
\end{enumerate}
\end{lemma}

\subsection{Mobiles of generic critical pointed and rooted Boltzmann planar maps}
\label{sec:mobile}

Our results allow a control of the volume of generic critical rooted and pointed Boltzmann planar maps, conditioned on the distance between their root vertex and marked vertex being large. It also gives us access to the tail of the distribution of this distance. Both are useful to investigate the fine geometric properties of Boltzmann planar maps. We plan to investigate further consequences of our results (and the corollary below) in an upcoming work. 

Let us make our point more precise. 	 
A planar map is a cellular embedding of a connected finite planar graph on the sphere, considered up to orientation-preserving homeomorphism. By cellular, we mean that the image of edges do not intersect except possibly at their endpoints, and that the connected component of the complement of the union of all edges form a collection of disjoint simply connected open sets, which we call the faces of the map. The degree of a face, resp. of a vertex is the number of edges incident to that face, resp. that vertex (if both ``sides'' of an edge are incident to the same face, or both ``endpoints'' are the same vertex, the edge is counted twice). A map is bipartite if all its faces have even degree. A rooted and pointed map $(\bm, e, v)$ is a map $m$ with a distinguished oriented edge $e$, called the root edge, and a distinguished vertex $v$. A bipartite map is said to be positive if is oriented towards its endpoint that is closest to $v$ for the graph distance in the map.

In this section, all the maps we consider are positive rooted and pointed bipartite planar maps. 
By convention, we allow the map $\dagger$ with a single vertex, no edge, and a single face of degree $0$, and we consider it to be rooted and pointed. 

Given nonnegative numbers $(\bq_k)_{k\geq 1}$ that are not identically zero, define for every map $\bm$
\[ w_\bq(\bm) := \prod_{f\in F(\bm)} \bq_{\deg(f)} , \]
where $F(\bm)$ is the set of faces of $\bm$, and with the convention $w_\bq(\dagger) = 1$. Let $Z_\bq = \sum_{\bm} w_\bq(\bm)$, where the sum is over all positive rooted and pointed bipartite planar maps. Proposition 1 in \cite{marckert2007invariance}, rephrased in \cite[Section 3.2.2]{curien2019peeling} (we follow this latter reference) gives a condition for $Z_\bq < \infty$ (we say that the sequence $(\bq_k)_{k\geq 1}$ is admissible), namely that the function
\[ f_\bq(x) = 1 + \sum_{k\geq 1} \bq_\bk {{2k-1}\choose k} x^k \]
is such that $f_\bq(x) = x$ admits a positive solution; $Z_\bq$ is then its smallest positive solution. The weight sequence is critical if $f_\bq'(Z_\bq) = 1$, and generic critical if in addition $f_\bq''(Z_\bq) < \infty$. In this section, we always assume that $\bq$ is generic critical. 

Positive pointed and rooted bipartite maps are in bijection by the Bouttier--Di Francesco--Guitter bijection \cite{bouttier2004planar} with rooted mobiles; when the map is distributed under the Boltzmann distribution, the mobile has the distribution of a two-type (undecorated) \brw. Let us follow \cite[Section 2.3]{marckert2007invariance} to make this more precise, with two caveats: first, our $f_\bq(x)$ equals $1+xf_\bq(x)$ with their $f_\bq$; second, we work with \emph{rooted and pointed} maps, while \cite{bouttier2004planar} works with \emph{pointed} maps and \cite{marckert2007invariance} works with \emph{rooted} maps. Since the bijection is essentially the same in all three settings, we focus on describing the distribution of the mobile, using the formalism of Section \ref{sec:extensions}. 

Let $\kX = \{F,V\}$. Vertices at even generations are of type $V$ while vertices at odd generations are of type $F$. Vertices of type $F$, resp. type $V$, follow a \repscheme $\BB_F$, resp. $\BB_V$ defined as follows. Let 
$\mu_F(k) = \frac{1}{1-Z_\bq^{-1}} {{2k+1} \choose {k+1}} \bq_{k+1} Z_\bq^{k}$ for every $k\geq 0$, and $\mu_V(k) = Z_\bq^{-1} (1-Z_\bq^{-1})^k$ for every $k\geq 0$; and let $N_F \sim \mu_F$, resp. $N_V \sim \mu_V$. 
First, let $D_V\geq 0$ be any integrable random variable, that may depend on $N_V$, and let 
\[ \BB_V = ( N_V \delta_{(0,F)} \ , \ 0 \ , \ D_V \ , \ (0,F)_{1\leq n \leq N_V}) . \]
Conditionally on $N_F$, define $(b_k)_{0\leq k\leq 2N_F+2}$ be a uniform bridge with steps $\pm 1$ with $b_0 = b_{2N_F+2} = 0$ and $b_1 = -1$. An index $0\leq k < 2N_F+2$ is called a downstep if $b_{k+1} = b_k-1$. Write $\DD$ for the (random) set of downsteps of $(b_k)_{0\leq k \leq 2N_F+2}$, excluding $0$. Then let $D_F\geq 0$ be an integrable random variable, that may depend on $N_F$, $\DD$ and $(b_k)_{k\in\DD}$, and define 
\[ \BB_F = \left( \sum_{k\in\DD} \delta_{(b_k,V)} \ , \ \max_{k\in\DD} b_k \ , \ D_F \ , \  (b_k,V)_{k\in\DD} \right) . \]


\begin{proposition}\label{prop:mobile}
Assume that $D_V$ and $D_F$ are not both a.s. zero. Then the two-type \brw 
\[ (T, \bs, (D_v)_{v\in T}, (X_v)_{v\in T}, (\Lambda_v)_{v\in T}) \]
satisfies Assumptions \ref{assum-multitype}. 
\end{proposition}

\begin{proof}
See Appendix \ref{sec:prop-mobile}.
\end{proof}

\begin{corollary}\label{cor:maps}
Let $\bm$ be a positive pointed and rooted planar map distributed under the probability distribution $Z_\bq^{-1} w_\bq(\cdot)$ (this is the Boltzmann distribution), and let $\bd$ be the distance between its root vertex and its pointed vertex and $N_V$, resp. $N_F$ and $N_E$ its number of vertices, resp. faces or edges. Then letting $\sigma^2 = Z_\bq^2 f_\bq''(Z_q)$ and $c_V = 1$, $c_E=Z_\bq$ and $c_F = Z_\bq-1$, 
as $r\to\infty$, 
\begin{align}
	\P(\bd > r) &\sim 2 r^{-2} , \\
	\P(\bd = r) &\sim 4 r^{-3} ,
\end{align}
and under $\P(\, \cdot \, | \bd = r)$, resp. under $\P(\, \cdot \, | \bd>r)$, we have the convergence in distribution
\[ \frac{1}{\sigma^2 r^4} \left(\frac{N_V}{c_V} , \frac{N_F}{c_F}, \frac{N_E}{c_E}\right) \ulimd r \infty (\bn, \bn, \bn) \]
where recalling $\psi$ from Theorem \ref{th:main-volume-smaller}
\[ \E\left[ \e^{-2\alpha^2 \bn} \right] = \frac{2\alpha}{\sigma^2} \psi\left(\sqrt{\frac{12\alpha}{\sigma^2}}\right) . \]
\end{corollary}

This follows from the interpretation of the maximum displacement of the mobile as minus one, plus the distance between the root vertex and pointed vertex of $\bm$. 
To count the number of vertices one takes $D_V = 1$ and $D_F = 0$; to count the number of faces one takes $D_V = 0$ and $D_F = 1$; and recalling Euler's formula, to count the number of edges we takes $D_V = D_F = 1$ (up to an additive constant which does not contribute to the formulas of Corollary \ref{cor:maps} as $r\to\infty$).

\subsection{Connection with the Brownian snake}
\label{sec:snake}

Under stronger hypotheses than ours, it has been shown e.g. in \cite{janson2005convergence,miermont2008invarianceGW,raphaelis2017scaling} and references therein that the \brw converges towards the Brownian snake. Let us make this more precise in a simple case, and explain how this can shed light on our results. 

We follow \cite{janson2005convergence}. In their case, $\chi = \sum_{i=1}^\xi \delta_{U_i}$ and $\Lambda = \sup\chi$, where the $(U_i)_{i\geq 1}$ are i.i.d. with $\E[U_1]=0$, $\Var(U_1) = \eta^2 >0$ and $\P(U_1>t) = o(t^{-4})$, and $\E[\xi]=1$, $\Var(\xi) = \sigma^2 > 0$ and $\E[\e^{a\xi}]<\infty$ for some $\alpha>0$. The condition of exponential moment on the offspring distribution is present in most of the literature, although it has been noted e.g. in \cite{miermont2008invarianceGW} that it could be weakened. The structure of $\chi$ was also weakened e.g. in \cite{miermont2008invarianceGW}. Still, to the best of our knowledge the state of the art still falls short of our Assumptions \ref{assum}, which we conjecture are sufficient to establish the convergence towards the Brownian snake in the sense that follows. 

Condition the \brw to have progeny $n$, i.e. $n$ individuals in total across all generations. Let $f:[0,2n] \to T$ be the depth-first enumeration of the vertices of $T$, as defined in \cite[Section 1.3]{janson2005convergence}, let $V_n(t) = |f(t)|$ be the generation of $f(t)$ (its height in the tree) and $R_n(t) = X_{f(t)}$ its label (the discrete head process), defined for integer values and interpolated linearly for non-integer values. These two processes will be enough for our purposes. Then letting $v_n(t) = n^{-1/2} V_n(2nt)$ and $r_n(t) = n^{-1/4} R_n(2nt)$ for every $t\in [0,1]$, by \cite[Theorem 2 and Corollary 1]{janson2005convergence}, as well as the remark in their Section 1.4 for the normalization, we have that $((\sigma/2) v_n, (\sigma/2)^{1/2} \eta^{-1} r_n)\to (v,r)$ in distribution as $n\to\infty$ in $C([0,1])^2$, where $(v,r)$ is the head of the Brownian snake with lifetime process $v$ a normalized Brownian excursion. 

The quantities of interest for us are then easily defined: the total progeny $\bV$ of the \brw is analogous to the lifetime $\zeta$ of the excursion, while the maximum displacement $M$ is analogous to the maximum value attained by the head of the snake, $r_* = \sup_{[0,\zeta]} r$. As a consequence of the above convergence, we get that for every $x\in(0,\infty)$, 
\begin{equation}\label{eq:convergence-max-towards-snake}
\P(M\geq x n^{1/4} \ | \ \bV = n) \ulim n \infty \P(r_*\geq (\sigma/2)^{1/2} \eta^{-1} x) .
\end{equation}
It is well-known that as $n\to\infty$ (using e.g. Kemperman's formula and a local limit theorem for the Lukasiewicz walk of the Bienaymé--Galton--Watson tree)
\[ \P(\bV = n) \sim \frac{1}{\sqrt{2\pi\sigma^2}} n^{-3/2} . \]
Combining it with \eqref{eq:convergence-max-towards-snake}, we get that as $r\to\infty$,
\begin{align*}
\P(M\geq r) &= \sum_{n\geq 1} \P(\bV = n) \P(M\geq r \ | \ \bV=n) \\
&\sim (2\pi\sigma^2)^{-1/2} \sum_{n\geq 1} n^{-3/2} \P(r_* \geq \sigma^{1/2} \eta^{-1} r n^{-1/4}) \\
&\sim (2\pi\sigma^2)^{-1/2} r^{-2} \sum_{n\geq 1} r^{-4} \left( \frac{r^4}{n} \right)^{3/2} \P\left(r_* \geq \frac{\sigma^{1/2}}{\eta \sqrt{2}} \frac{r}{n^{1/4}} \right) \\
&\sim \frac{1}{r^2 \sqrt{2\pi\sigma^2}} \int_0^\infty y^{-3/2} \P\left(r_* \geq \frac{\sigma^{1/2}}{\eta\sqrt{2}} \ | \ \zeta = y \right)  \d y \\
&\sim \frac{2}{\sigma r^2} \N_0\left( r_* \geq \frac{\sigma^{1/2}}{\eta\sqrt{2}} \right) ,
\end{align*}
where $\N_0$ is the excursion measure of the Brownian snake started from $0$ (see \cite[Section III.5, Section IV.I and Section IV.6]{gall1999spatial}.) Since 
\[ \N_0(r_* \geq y) = \frac{3}{2y^2} \]
we finally get 
\[ \P(M\geq r) \sim \frac{6\eta^2}{\sigma^2 r^2} , \]
recovering our Corollary \ref{cor:main-tail}. Similarly, we can find that
\begin{align*}
\E\left[ \e^{-t\bV r^{-4}} \ | \ M\geq r \right] \ulim r \infty \N_0\left[ \e^{-t\zeta} \ | \ r_*\geq \frac{\sigma^{1/2}}{\eta\sqrt{2}} \right] .
\end{align*}
This gives us an interpretation of the limits in Theorem \ref{th:main-volume-smaller} and \ref{th:main-condition-vol}. 

So why prove our results directly? Most importantly, our hypotheses are weaker than the state-of-the-art for the convergence of discrete snakes towards the Brownian snake. In fact, we aim to use our estimates to prove the convergence towards the Brownian snake for \brw satisfying Assumptions \ref{assum} in an upcoming work.


\section{Proof of the tail estimate}
\label{sec:main-proof}

The goal of this section is to prove Theorem \ref{th:main-tail}. Our proof follows that of \cite[Theorem 1]{lalley2015maximal}, with adaptations. The main idea is to use the Markov property of the \brw to express our object of interest with a martingale: in our case, the function $w(r) = \frac{1}{\P(\sup_{v\in T} \Lambda_v \leq r)}-1 \approx \P(\sup_{v\in T}\Lambda_v > r)$ will be expressed as
\begin{equation}\label{eq:summary-w}
w(r) \approx \E\left[ \kW^{(r)}_n w(S_n) \right] 
\end{equation}
where $(S_n)$ is a random walk started at $r$ and $\kW^{(r)}_n$ is a weighting factor that depends on $w$ and on the trajectory of the random walk. This is established in Proposition \ref{prop:almost-martingale}, after first proving the martingale property in Proposition \ref{prop:convolution-equation} and \ref{prop:convolution-w}. Applying the optional stopping theorem at the first time that $S_n$ goes below a certain value $x$ gives us Corollary \ref{cor:Feynman-Kac}:
\[ \frac{w(r)}{w(x)} \approx \E\left[ \kW^{(r)}_{\tau_x} \right] , \]
from which we deduce a functional equation on the limit of the ratio $w(r)/w(x)$ as $x$ and $r$ go to infinity (Proposition \ref{prop:func-equation-phi}). Solving this equation then allows us to obtain the asymptotics of $w$. 

The difference in our approach happens, for the most part, in the build-up to Corollary \ref{cor:Feynman-Kac} and in the proof of Proposition \ref{prop:func-equation-phi}. For starters, proving \eqref{eq:summary-w} is much harder in our case than in \cite{lalley2015maximal}, because we cannot take advantage of a specific nice structure of the \repscheme. This forces us to truncate the possible values of the branching steps (including those of $\Lambda$); as a consequence, \eqref{eq:summary-w} will only hold up to a remainder that we need to control. The proof of Proposition \ref{prop:convolution-w} requires to devote significant effort to the control of the remainder, while its equivalent (Proposition 5) in \cite{lalley2015maximal} barely warrants a proof.

In addition, a crucial technical ingredient in the proof is a control of the ``continuity'' of $w$. This is done in \cite{lalley2015maximal} by Proposition 8 and Lemma 11. In this article, I use a new estimate on the continuity of $w$ (Lemma \ref{lem:continuity-w-broad}), which contains significant information on the behavior of $w$ (giving, for example, a lower bound that is optimal up to a multiplicative constant). This provides an alternative to a few steps in \cite{lalley2015maximal}. In particular, it provides a justification to the last convergence in the proof of \cite[Lemma 11]{lalley2015maximal}, which does not seem to follow from the estimates in \cite{lalley2015maximal}. 

Finally, to establish Corollary \ref{cor:Feynman-Kac} we need to control the uniform integrability of $\kW^{(r)}_n w(S_n)$ before applying the optional stopping theorem. This is easy in \cite{lalley2015maximal} where $\kW^{(r)}_n \leq 1$ almost surely; in our case it is given by Lemma \ref{lem:discount-UI}.

\subsection{The Markov property of the Branching Random Walk}
\label{sec:Markov-ppty-BRW}

Let $h(r) = \P(\sup_{v\in T} \Lambda_v \leq r)$. The function $h$ is increasing and right-continuous, and under Assumptions \ref{assum} we have $h(0)<1$, $h(r)=0$ for every $r<0$, and $h(r)\to 1$ as $r\to\infty$. We define $w(r) = \frac{1}{h(r)}-1$ for every $r\geq 0$. 

\begin{proposition}\label{prop:convolution-equation}
With the convention that $\ln 0 = -\infty$, for every $r\geq 0$
\begin{equation}
h(r) = \E\left[ \ind{\Lambda\leq r} \exp\left( \int \ln h(r-x) \d \chi(x) \right) \right] .
\end{equation}
\end{proposition}

This is the analogue of \cite[Proposition 5]{lalley2015maximal}. If $\Lambda = \sup(0, \sup\chi)$, then defining the Laplace functional $\LL_\chi(f) := \E[\exp(-\int f \d\chi)]$, then we recognize the equation $h(r) = \LL_\chi(-\ln h(r-\cdot))$.

\begin{proof}
This is a simple application of the Markov property. Let $(\chi, \Lambda, (X_i)_{1\leq i \leq \chi(\R)}) \sim \BB$ be independent of $(T, (X_v)_{v\in T}, (\Lambda_v)_{v\in T})$, then
\begin{align*}
h(r) &= \E\left[ \ind{\Lambda\leq r} \prod_{i=1}^{\chi(\R)} \P\left( \sup_{v\in T} \Lambda_v \leq r-X_i \right) \right] = \E\left[ \ind{\Lambda\leq r} \prod_{i=1}^{\chi(\R)} h(r-X_i) \right]
\end{align*}
from which we recognize the expression of the Proposition. 
\end{proof}

\subsection{Continuity of the tail}
\label{sec:lower-bound-ratio}

An original contribution that allows to circumvent a number of steps of \cite{lalley2015maximal} is the following Lemma.

\begin{lemma}\label{lem:continuity-w-broad}
There exists a neighborhood $I$ of $0$ and two non-increasing continuous functions $F_-, F_+$ on $I$ with $F_-(0) = F_+(0) = 1$ such that for every $y\in I$, 
\begin{align}\label{eq:continuity-ratio-right}
F_-\left( \frac{\sigma}{\eta} y \right) 
&\leq \liminf_{r\to\infty} \frac{w\left( r + \frac{y}{\sqrt{w(r)}}\right)}{w(r)} \\
&\leq \limsup_{r\to\infty} \frac{w\left( r + \frac{y}{\sqrt{w(r)}}\right)}{w(r)} \leq F_+\left( \frac{\sigma}{\eta} y \right) . \label{eq:continuity-ratio-left}
\end{align}
We can take 
\begin{equation*}
F_-(y) = \begin{cases} 
1-y\sqrt{\frac{8}{\pi}} \geq 1-2y &\text{if } y\geq 0 \\
1 						&\text{if } y<0 
\end{cases} \qquad , \qquad 
F_+(y) = \begin{cases}
1 &\text{if } y\geq 0 \\
f^{-1}(-y) &\text{if }y<0 \text{ large enough}
\end{cases}
\end{equation*}
where $f: [1,3] \to [0, \frac{\sqrt{\pi}}{3\sqrt{6}}] , x\mapsto \frac{1-x^{-1}}{\sqrt{x}}\sqrt{\pi/8}$. 

\end{lemma}



\begin{figure}[h]
\begin{center}
\begin{tikzpicture}
\begin{axis}[
    axis lines = left,
    xlabel = \(y\),
]

\addplot [
    domain=-0.5:1, 
    samples=100, 
    color=green,
]
{(1+x/sqrt(6))^(-2)};

\addplot [
    domain=0:0.62666, 
    samples=100, 
    color=red,
]
{1 - x * sqrt(8 / 3.141592654)};
\addplot [
    domain=0.62666:1,
    samples=100, 
    color=red,
]
{0};
\addplot [
    domain=-0.5:0, 
    samples=100, 
    color=red,
]
{1};

\addplot [
    domain=1:3.05, 
    samples=100, 
    color=blue,
]
(-(1-1/x) * sqrt(3.141592654) / sqrt(8*x), x);
\addplot [
    domain=0:1, 
    samples=100, 
    color=blue,
]
{1};

\node[circle,fill,inner sep=2pt,color=blue] at (axis cs:-0.2412,3) {};

\end{axis}
\end{tikzpicture}
\end{center}
\caption{Lower bound on the $\liminf$ (in red) and upper bound on the $\limsup$ (in blue) in Lemma \ref{lem:continuity-w-broad}, when $\eta = \sigma$. Note that the expressions in \eqref{eq:continuity-ratio-right} for $y\geq 0$ and \eqref{eq:continuity-ratio-left} for $y\leq 0$ define a function that is continuous and with continuous derivative at $y=0$. The actual limit from Theorem \ref{th:main-tail} is drawn in green for reference; it diverges at $-\sqrt{6} \approx -2.45$ while the blue one has finite value $3$ and infinite derivative at $-\frac{\sqrt{\pi}}{3\sqrt{6}} \approx -0.2412$.} 
\end{figure}
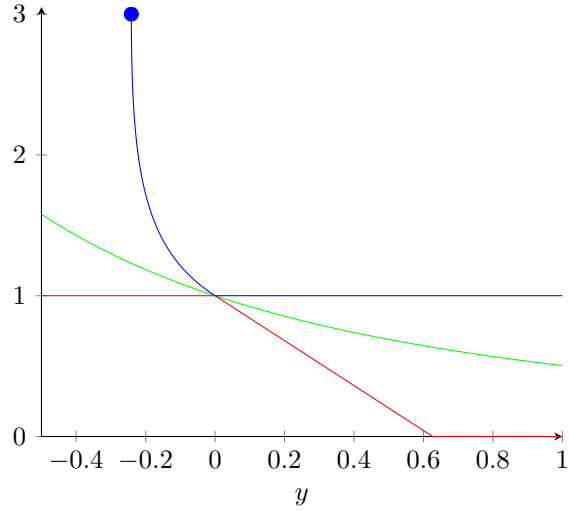

A useful consequence is that because $w(r)\to 0$ as $r\to\infty$ and because $w$ is decreasing, we have for every fixed $K>0$ 
\[ \lim_{r\to\infty} \sup_{|y| \leq K} \left| \frac{w(r+y)}{w(r)} -1\right| = 0 , \]
which is the analogue of \cite[Lemma 13]{lalley2015maximal}.

\begin{proof}[Proof of Lemma \ref{lem:continuity-w-broad}]
We first prove \eqref{eq:continuity-ratio-right}. The result is immediate if $y\leq 0$. 
Rewrite Proposition \ref{prop:convolution-equation}: for every $r,u\geq 0$, using that $\ind{\dots}\leq 1$,
\[ h(r+u) \leq \E\left[ \prod_{x\in G_1} h(r+u-x) \right] . \]
Apply the Markov property $n$ times:
\[ h(r+u) \leq \E\left[ \prod_{x\in G_n} h(r+u-x) \right] . \]
Let us assume that there exists a decreasing function $0\leq f \leq 1$ such that for every $y>0$
\begin{equation}
\liminf_{r\to\infty} \frac{w\left(r + \frac{y}{\sqrt{w(r)}}\right)}{w(r)} \geq f(y) .
\end{equation}
Clearly $f(y) = \ind{y\leq 0}$ works. 
Bound $h(x) \geq h(x) \wedge h(r)$:
\[ h(r+u) \leq \E\left[ \exp\left(\ln h(r) G_n([u,\infty)) + \int_{(-\infty, u]} \ln h(r+u-x) \d G_n(x) \right) \right] . \]
Write $X$ for the expression inside the exponential. 
Since $X<0$ we have $\e^{X} \leq 1 + X + \frac{X^2}{2}$. We then use 
the many-to-one formula: letting $(U_n)_{n\geq 0}$ be a random walk started at $0$ with i.i.d. steps distributed under $M$ the mean measure of $\chi$, 
\[ \E[X] = (\ln h(r)) \P(U_n\geq u) + \E\left[\ind{U_n<u} \ln h(r+u-U_n)\right] . \]
Take $u = y / \sqrt{w(r)} = y \left( \frac{1}{h(r)}-1\right)^{-1/2} \geq y (-\ln h(r))^{-1/2}$ for $y>0$ and $n = \left\lfloor \frac{s}{\eta^2 w(r)} \right\rfloor$ for some $s>0$, so that $\frac{U_n}{\sqrt{n\eta^2}} \ulim r \infty N$ in distribution where $N\sim\NN(0,1)$: 
\begin{align*}
\liminf_{r\to\infty} \frac{-\E[X]}{w(r)} &\geq \liminf_{r\to\infty} \left( \P(U_n\geq u) + \E\left[\ind{U_n<u} \frac{-\ln h(r+u-U_n)}{w(r)} \right] \right) \\
&\geq \P(N\geq y\sqrt{s}) + \liminf_{r\to\infty} \E\left[\ind{U_n\in(-\infty,u)} f\left(y - U_n\sqrt{w(r)}\right) \right] \\
&\geq \P(B_s\geq y) + \E\left[\ind{B_s<y} f\left(y - B_s\right) \right] .
\end{align*} 
where $(B_v)_{v\geq 0}$ is a centered standard Brownian motion. 
On the other hand, 
\[
\E[X^2] \leq \E[G_n(\R)^2] (\ln h(r))^2 = (1+n\sigma^2) (\ln h(r))^2
\]
so that 
\[ \limsup_{r\to\infty} \frac{\E[X^2]}{w(r)} \leq \limsup_{r\to\infty} \frac{s \sigma^2}{\eta^2 w(r)^2} (\ln h(r))^2 = \frac{s \sigma^2}{\eta^2} . \]
We deduce
\begin{equation*}
\liminf_{r\to\infty} \frac{w(r+u)}{w(r)} = \liminf_{r\to\infty} \frac{1-h(r+u)}{w(r)} \geq \E\left[\ind{B_s\geq y} + \ind{B_s<y} f\left(y - B_s\right) \right] - \frac{s \sigma^2}{2\eta^2} .
\end{equation*}
Consider the sequence $f_s^{(m)}$ defined by $f_s^{(0)}(y) = \ind{y\leq 0}$ and 
\[ f_s^{(m+1)}(y) = \E\left[\ind{B_{s}\geq y} + \ind{B_{s}<y}f_s^{(m)}\left(y-B_{s}\right)\right] - \frac{s \sigma^2}{2\eta^2} . \]
We just proved that for every $m$ and $s$, 
\[ \liminf_{r\to\infty} \frac{w(r+u)}{w(r)} \geq f_s^{(m)}(y) . \]
We can check by the Markov property that
\begin{align*}
f_s^{(m)}(y) &= \sum_{k=0}^{m-1} \E\left[\ind{B_0 < y, B_s<y, B_{2s}<y, \dots, B_{ks}<y} \left(\ind{B_{(k+1)s}\geq y} - \frac{s \sigma^2}{2\eta^2} \right) \right] \\
&= \P(\tau_y^{(s)} \leq ms) - \frac{\sigma^2}{2\eta^2} \E\left[ s + \tau_y^{(s)}\ind{\tau_y^{(s)} \leq ms} \right]
\end{align*}
where $\tau_y^{(s)} = \inf \{k\geq 0\ :\ B_{ks}\geq y \}$. Letting $\tau_y = \inf\{x\geq 0 \ : \ B_x\geq y\}$, taking the limit $s\to 0$ with $ms\to x$, we find that 
\[ f(y) = \P(\tau_y \leq x) - \frac{\sigma^2}{2\eta^2} \E[\tau_y \ind{\tau_y \leq x}] \]
is suitable for the Lemma. Using that $\tau_y$ and $y^2 \tau_1$ have the same distribution, and that $\P(\tau_1>x) \leq 2(2\pi x)^{-1/2}$, hence by the layer-cake formula $\E[\tau_1 \ind{\tau_1\leq x}] = \int_0^x \P(\tau_1>x) \leq (2\pi)^{-1} 4\sqrt{x}$, then taking $x = y^2 \eta^2 / \sigma^2$
\begin{align*}
f(y) \geq 1 - \frac{2}{\sqrt{2\pi}}\inf_{x>0} \left(\frac{1}{\sqrt{x}} + \frac{y^2 \sigma^2 \sqrt{x}}{\eta^2} \right) = 1 - y \frac{\sigma}{\eta}\sqrt{\frac{8}{\pi}} .
\end{align*}
Note that using the correct asymptotic would give us $f(y)-1 \sim - y \frac{2\sigma}{\eta\sqrt{6}}$ as $y\to 0$, which is compatible with the $f$ above.

Let us now deduce \eqref{eq:continuity-ratio-left}. Assume that \eqref{eq:continuity-ratio-right} holds for some constant $c$. The result is immediate if $y\leq 0$. For every $x>0$, define $r_x := \inf\{r\geq 0 : w(r) \leq x \}$; by right continuity of $h$ and thus $w$, we have $w(r_x) \leq x$. On the other hand, by the asymptotic continuity of $w$ given after Lemma \ref{lem:continuity-w-broad}, which is a consequence of only \eqref{eq:continuity-ratio-right}, we have that $\limsup_{x\to 0} x^{-1} \lim_{s\nearrow r_x} w(s) = 1$ and $\lim_{x\to 0} x^{-1} w(r_x) = 1$. 
Since $r_x \to \infty$ as $x\to 0$, we have by \eqref{eq:continuity-ratio-left} that for every $y>0$ 
\[ \liminf_{x\to 0} \frac{w\left(r_x + \frac{y}{\sqrt{w(r_x)}}\right)}{w(r_x)} \geq 1 - c\frac{\sigma}{\eta} y . \]
Fix $A > 1$ and consider $r_{Aw(r)}$ as $r\to\infty$: for every $y < (1-A^{-1})\frac{\eta}{c\sigma}$, 
\begin{equation*}
\liminf_{r\to \infty} \frac{w\left(r_{Aw(r)} + \frac{y}{\sqrt{w(r_{Aw(r)})}}\right)}{w(r)} = A \liminf_{r\to \infty} \frac{w\left(r_{Aw(r)} + \frac{y}{\sqrt{w(r_{Aw(r)})}}\right)}{w(r_{Aw(r)})} \geq A \left(1 - c\frac{\sigma}{\eta} y\right) > 1 , 
\end{equation*}
so that for every $r$ large enough, $r > r_{Aw(r)} + \frac{y}{\sqrt{w(r_{Aw(r)})}} \geq r_{Aw(r)} + \frac{y/\sqrt{A}}{\sqrt{w(r)}}$. Then 
\[ r_{Aw(r)} \leq r - \frac{y/\sqrt{A}}{\sqrt{w(r)}} \implies \limsup_{r\to\infty} \frac{ w\left( r - \frac{1}{\sqrt{w(r)}} \frac{y}{\sqrt{A}} \right)}{w(r)} \leq \limsup_{r\to\infty} \frac{ w\left( r_{Aw(r)} \right)}{w(r)} \leq A . \]
Write $f(x) = \frac{1-x^{-1}}{c\sqrt{x}}$, defined for $x>1$. This is bijective from $[1, 3]$ to $[0, \frac{2}{3c\sqrt{3}}]$, and is such that for every $A\in(1,3)$ and every $0\leq z < \frac{\eta}{\sigma} f(A)$, 
\[ \limsup_{r\to\infty} \frac{ w\left( r - \frac{z}{\sqrt{w(r)}} \right)}{w(r)} \leq A , \]
or equivalently, that for every $0<z<\frac{2}{3c\sqrt{3}} \frac{\eta}{\sigma}$, 
\[ \limsup_{r\to\infty} \frac{ w\left( r - \frac{z}{\sqrt{w(r)}} \right)}{w(r)} \leq f^{-1}\left(\frac{\sigma}{\eta} z\right) . \]
\end{proof}

We can already see that Lemma \ref{lem:continuity-w-broad} contains significant information on the tail. For example, we can deduce from it that $\liminf_{r\to\infty} r^2 w(r) \geq \frac{\pi \eta^2}{54 \sigma^2} \approx 0.00969 \frac{6 \eta^2}{\sigma^2}$. Indeed, if this was not the case then we could take $z$ such that \eqref{eq:continuity-ratio-left} would fail at the limit $r\to\infty$ because $r-\frac{z}{\sqrt{w(r)}}$ would take negative values and the limsup would thus be unbounded. For the sake of illustration, we provide here an independent and elementary proof of this fact. 

\begin{lemma}\label{lem:lower-bound-w}
$\liminf_{r\to\infty} r^2 w(r) > 0 .$
\end{lemma}

\begin{proof}
We use the second moment method. By the many-to-one formula, if $(U_n)_{n\geq 0}$ is a random walk with $U_0=0$ and step distribution given by $M$, 
\[ \E[G_{r^2}([r, \infty))] = \P(U_{r^2} \geq r) \ulim r \infty \P( N \geq \eta^{-1} ) > 0 \]
where $N\sim\NN(0,1)$. On the other hand, 
\[ \E[G_{r^2}(\R)^2] = 1 + r^2\sigma^2 . \]
By the Paley--Zygmund inequality,
\[ \P(G_{r^2}([r,\infty)) > 0) \geq \frac{\E[G_{r^2}([r, \infty))]^2}{\E[G_{r^2}(\R)^2]} \]
thus giving $\liminf_{r\to\infty} r^2(1-h(r)) > 0 $. 
\end{proof}

\subsection{The convolution equation}

For every $0\leq R \leq r$ and every $\delta>0$, define 
\begin{equation}\label{eq:event-E-rRd}
E(r,R,\delta) := \left\{\chi(\R)\leq \frac{\delta}{w(r-R)} \right\} \cap \{ \Lambda \leq R \} .
\end{equation} 
Let $(Z, I^{(r,R,\delta)})$ be the couple of random variables with values in $\R\times \{0,1\}$ such that for every positive and measurable $f$, 
\begin{equation}\label{eq:def-ZI}
\E\left[ f(Z, I^{(r,R,\delta)}) \right] = \E\left[ \int f(x,\ind{E(r,R,\delta)}) \d\chi(x) \right] .
\end{equation}
This does indeed define a random variable because $\E[\chi(\R)]=1$. We also define the Markov kernel $((z,i), \Gamma) \mapsto P^{(r,R,\delta)}_{z,i}(\Gamma)$, where $\Gamma$ is a measurable subset of the set of point measures, $i\in\{0,1\}$ and $z\in\R$, such that for every positive and measurable $g$, if $(X_i)_{1\leq i \leq \chi(\R)}$ is a measurable numbering of the atoms of $\chi$, 
\begin{equation}\label{eq:def-Palm}
\E\left[ \sum_{i=1}^{\chi(\R)} g\left(X_i , \ind{E(r,R,\delta)}, \chi - \delta_{X_i} \right) \right] = \E\left[ \int g\left(Z, I^{(r,R,\delta)}, \phi \right) \d P^{(r,R,\delta)}_{Z, I^{(r,R,\delta)}}(\phi) \right] .
\end{equation} 
This Markov kernel defines a ``mean measure function'' $M^{(r,R,\delta)}_{z}$ such that for every Borel subset $A$ of $\R$, 
\begin{equation}\label{eq:def-Palm-mean}
M^{(r,R,\delta)}_{z}(A) = \int \phi(A) \d P^{(r,R,\delta)}_{z,1}(\phi) . 
\end{equation}
Note that we only define it when $i=1$ in $P^{(r,R,\delta)}_{z,i}$, since the case $i=0$ will not be of interest to us.

\begin{proposition}\label{prop:convolution-w}
Under Assumptions \ref{assum} 1. to 4., there exists $\delta^0>0$ and $r_0>0$ such that for every $\delta \in (0,\delta^0)$, $r_\bullet \geq r_0$, $r\geq 2r_\bullet$ and for every $r_\bullet \leq R \leq  r-r_\bullet$, 
\begin{multline}\label{eq:convolution-w}
w(r) = \E\Bigg[ I^{(r,R,\delta)} w(r-Z) \exp\Big( - w(r-Z) - \frac{1}{2} \int w(r-y) \d M^{(r,R,\delta)}_{Z} (y) \\
+ \E\left[ I^{(r,R,\delta)} w(r-Z) \right] \Big) \Bigg] + \mathrm{Remainder}(r,R,\delta) ,
\end{multline}
where
\begin{equation*} 
\frac{1}{90}|\mathrm{Remainder}(r,R,\delta)| \leq \P(\Lambda>R) + \P\left(\chi(\R)>\frac{\delta}{w(r-R)}\right) + (\sigma^2 + 1) w(r-R)^2 \left( w(r-R) + \delta \right) .
\end{equation*}
\end{proposition}

This proposition plays the role of \cite[(11)]{lalley2015maximal}. 

\begin{remark}
Let us anticipate the final result to see what is the order of magnitude of the remainder. We will typically take $R = \frac{y}{\sqrt{w(r)}}$ for some $y>0$, so that $\P(\Lambda>R) = o(w(r)^2)$ by Lemma \ref{lem:refined-Markov}. On the other hand, we will show that for this choice of $R$ then $w(r-R) \leq c w(r)$, hence $\P\left(\chi(\R)>\frac{\delta}{w(r-R)}\right) = o(\delta^{-2} w(r-R))$. We may choose $\delta \to 0$ very slowly as $r\to\infty$ so that this, and the remainder as a whole, is $o(w(r)^2)$. This ensures that if we sum $\approx w(r)^{-1}$ occurences of this remainder, this error term will still be asymptotically negligible in front of the expectation, which will be of order $w(r)$. 
\end{remark}

The rest of this section is devoted to the proof of Proposition \ref{prop:convolution-w}. Since stating right now the conditions that $r_0$ and $\delta^0$ must satisfy would detract readability, I choose instead to state the conditions at the point where they appear naturally in the process of the proof. I will then state ``by taking $r_0$ larger if necessary we may assume that some new assertion holds'' (resp. by taking $\delta^0>0$ smaller if necessary) to mean that there exists a choice of $r_0$, resp. of $\delta^0>0$ such that the new assertion holds, in addition to all the assertions that have already been stated before. 
In everything that follows, we assume that $\delta \in (0,\delta^0)$, $r_\bullet \geq r_0$, $r\geq 2r_\bullet$ and that $r_\bullet \leq R \leq  r-r_\bullet$.

Let $Y_r = \int -\ln h(r-x)\d\chi(x)$, again using the convention $\ln 0 = -\infty$. 
Recall \eqref{eq:event-E-rRd}
\begin{equation*}
E(r,R,\delta) := \left\{\chi(\R)\leq \frac{\delta}{w(r-R)} \right\} \cap \{ \Lambda \leq R \} ,
\end{equation*} 
and write $(*) := \P\left( \chi(\R) > \frac{\delta}{w(r-R)} \right)$ and $(**) := \P(\Lambda > R)$. 
Since $-\ln h$ is decreasing, using that $-\ln h(x) = \ln(1+w(x)) \leq w(x)$ we have on $E(r,R,\delta)$
\[ Y_r \leq (-\ln h(r-R)) \chi(\R) \leq w(r-R) \frac{\delta}{w(r-R)} \leq \delta . \]
By Proposition \ref{prop:convolution-equation}
\begin{align}\label{eq:z1}
h(r) = \E\left[ \ind{\Lambda\leq r} \e^{-Y_r} \right] = \E\left[ \ind{E(r,R,\delta)} \e^{-Y_r} \right] + \E\left[\ind{E(r,R,\delta)^c} \ind{\Lambda\leq r} \e^{-Y_r}\right] .
\end{align}
Let $\omega_\e(x) = \e^{-x} - \left(1-x+\frac{x^2}{2}\right)$. Taking $\delta^0>0$ small enough we can ensure that $|\omega_\e(x)| \leq x^3$ for every $0\leq x \leq \delta^0$. 
Bounding the second term of \eqref{eq:z1} by $\P(E(r,R,\delta)^c) \leq (*) + (**)$ and writing the expansion of $\e^{-Y_r}$ in the first one, we get  
\begin{equation*}
\left| h(r) - 1 + \E[\ind{E(r,R,\delta)} Y_r] - \frac{1}{2}\E\left[ \ind{E(r,R,\delta)} Y_r^2 \right] \right| \leq (*) + (**) + \delta \, \E\left[ \ind{E(r,R,\delta)} Y_r^2  \right] .
\end{equation*}

\paragraph{Approximating $\E[Y_r \ind{E(r,R,\delta)}]$.}
Let $\omega_\ell(x) = \ln(1+x) - x + \frac{x^2}{2}$; for every $x\geq 0$ we have $0\leq \omega_\ell(x) \leq x^3$, hence  
\begin{align*}
\E[Y_r \ind{E(r,R,\delta)}] &= \E\left[\ind{E(r,R,\delta)} \int \left( w(r-x) - \frac{w(r-x)^2}{2} - \omega_\ell(w(r-x)) \right) \d \chi(x) \right]  .
\end{align*}
Since $w$ is decreasing, for every $x\leq R$ we have $w(r-x) \leq w(r-R)$. We deduce 
\begin{multline}\label{eq:approx-Yr1}
\left| \E[Y_r \ind{E(r,R,\delta)}] - \E\left[ \ind{E(r,R,\delta)} \int \left( w(r-x) - \frac{w(r-x)^2}{2} \right) \d\chi(x) \right] \right| \\
\leq w(r-R)^3 \E\left[ \ind{E(r,R,\delta)} \chi(\R) \right] \leq w(r-R)^3 .
\end{multline}
Since $w(x)\to 0$ as $x\to\infty$, taking $r_0$ larger if necessary we may assume that $w(x) \leq 1$ for every $x\geq r_0$. 
Then $0\leq w(r-x)-\frac{w(r-x)^2}{2} \leq w(r-x) $ for every $x\leq R$, hence 
\begin{equation}\label{eq:bound-Yr1}
0\leq \E\left[ \ind{E(r,R,\delta)} \int \left( w(r-x) - \frac{w(r-x)^2}{2} \right) \d\chi(x) \right] \leq \E\left[ \ind{E(r,R,\delta)} \int w(r-x) \d\chi(x) \right] \leq w(r-R) . 
\end{equation}

\paragraph{Approximating $\E[Y_r^2 \ind{E(r,R,\delta)}]$.}

On the other hand, using that $0\leq \ln(1+x) \leq x$ and $|\ln(1+x)-x| \leq x^2$ for $x>0$, 
\begin{align*}
\E[Y_r^2 \ind{E(r,R,\delta)}] =& \E\left[\ind{E(r,R,\delta)} \left(\int \left( w(r-x) + \ln(1+w(r-x)) - w(r-x) \right) \d \chi(x)\right)^2 \right] \\
=& \E\left[\ind{E(r,R,\delta)} \left(\int w(r-x) \d\chi(x) \right)^2\right] \\
&+ \E\Bigg[\ind{E(r,R,\delta)} \Bigg\{ 2\int w(r-x) (\ln(1+w(r-y))-w(r-y)) \d\chi(x)\d\chi(y) \\
&+ \int (\ln(1+w(r-x))-w(r-x)) (\ln(1+w(r-y))-w(r-y)) \d\chi(x)\d\chi(y) \Bigg\} \Bigg]
\end{align*}
so that 
\begin{multline}\label{eq:approx-Yr2}
\left| \E[Y_r^2 \ind{E(r,R,\delta)}] - \E\left[\ind{E(r,R,\delta)} \left(\int w(r-x) \d\chi(x) \right)^2\right] \right| \\
 \leq (\sigma^2 +1) (2w(r-R)^3 + w(r-R)^4) \leq 3 (\sigma^2 +1) w(r-R)^3
\end{multline}
since $r-R \geq r_0$ hence $w(r-R)\leq 1$. 
As before, we have
\begin{equation}\label{eq:bound-Yr2}
\E\left[\ind{E(r,R,\delta)} \left(\int w(r-x) \d\chi(x) \right)^2\right] \leq (\sigma^2 +1) w(r-R)^2 ,
\end{equation}
and from \eqref{eq:approx-Yr2} and \eqref{eq:bound-Yr2} we deduce (taking $r_0$ larger if necessary)
\begin{equation*}
\E\left[\ind{E(r,R,\delta)} Y_r^2 \right] \leq 2(\sigma^2 + 1) w(r-R)^2 .
\end{equation*}

\paragraph{Approximating $w(r)$}

Summing our estimates, letting 
\begin{multline}\label{eq:expansion-1-over-1-plus-w(r)}
\mathrm{Rem}(r,R,\delta) =  \frac{1}{1+w(r)} - 1 + \E\left[ \ind{E(r,R,\delta)} \int \left( w(r-x) - \frac{w(r-x)^2}{2} \right) \d\chi(x) \right] \\
- \frac{1}{2} \E\left[\ind{E(r,R,\delta)} \left(\int w(r-x) \d\chi(x) \right)^2\right] , 
\end{multline}
then 
\begin{equation}\label{eq:bound-rem}
|\mathrm{Rem}(r,R,\delta)| \leq (*) + (**) + (\sigma^2 + 1) w(r-R)^2 \left( 2\delta + 4w(r-R) \right) .
\end{equation}
Write $\mathrm{Bound}(r,R,\delta)$ for the right-hand side of \eqref{eq:bound-rem}. 
Taking $r_0$ larger if needed and taking $\delta^0>0$ smaller if needed, we have  
\begin{equation}\label{eq:bound-rem-2}
\mathrm{Bound}(r,R,\delta) \leq (*) + (**) + w(r-R)^2 .
\end{equation} 
Since 
\[ w(r) = \frac{1}{\frac{1}{1+w(r)}} -1 \]
we apply the fact that for every $|X|\leq \frac 12$
\[ \left|\frac{1}{1-X} - 1 - (X+X^2) \right| \leq 2 |X|^3 \]
to
\begin{multline*}
X := \E\left[ \ind{E(r,R,\delta)} \int \left( w(r-x) - \frac{w(r-x)^2}{2} \right) \d\chi(x) \right] \\
 - \frac{1}{2} \E\left[\ind{E(r,R,\delta)} \left(\int w(r-x) \d\chi(x) \right)^2\right] - \mathrm{Rem}(r,R,\delta) .
\end{multline*}
We first to check, using \eqref{eq:bound-Yr1}, \eqref{eq:bound-Yr2} and \eqref{eq:bound-rem-2}, that 
\[ |X| \leq (*) + (**) + (\sigma^2+1)^{1/3} w(r-R) \] 
hence that it is smaller than $1/2$ by taking $r_0$ larger if necessary. Together with the bound $(a+b+c)^3 \leq 9 (a^3 + b^3 + c^3)$ for every $a,b,c\geq 0$ and the fact that $(*)\leq 1$ and $(**)\leq 1$ we deduce 
\begin{align*}
\left| w(r) - (X + X^2) \right| \leq 18((*)^3 + (**)^3 + (\sigma^2+1) w(r-R)^3) \leq 18 \, \mathrm{Bound}(r,R,\delta) 
\end{align*}
since $(*)^3 + (**)^3 + (\sigma^2+1)w(r-R)^3 \leq (*) + (**) + 4(\sigma^2+1)w(r-R)^3$. 
On the other hand, 
\begin{multline*}
\Bigg| -(X + X^2) + \E\left[ \ind{E(r,R,\delta)} \int \left( w(r-x) - \frac{w(r-x)^2}{2} \right) \d\chi(x) \right] \\
- \frac{1}{2} \E\left[\ind{E(r,R,\delta)} \left(\int w(r-x) \d\chi(x) \right)^2\right] + \E\left[ \ind{E(r,R,\delta)} \int \left( w(r-x) - \frac{w(r-x)^2}{2} \right) \d\chi(x) \right]^2 \Bigg| \\
\leq \mathrm{Bound}(r,R,\delta)(1+w(r-R)+(\sigma^2+1)w(r-R)^2 + \mathrm{Bound}(r,R,\delta)) \\
+ 2 (\sigma^2+1) w(r-R)^3 + (\sigma^2+1)^2 w(r-R)^4 
\leq 3 \, \mathrm{Bound}(r,R,\delta) ,
\end{multline*}
where the last inequality holds by taking $r_0$ larger if needed, and observing that $3(\sigma^2+1)w(r-R)^3 \leq \mathrm{Bound}(r,R,\delta)$. Finally, we bound 
\begin{multline*}
\left| \E\left[ \ind{E(r,R,\delta)} \int \left( w(r-x) - \frac{w(r-x)^2}{2} \right) \d\chi(x) \right]^2 - \E\left[ \ind{E(r,R,\delta)} \int w(r-x) \d\chi(x) \right]^2 \right| \\
\leq 2w(r-R)^3 + w(r-R)^4 \leq 3 w(r-R)^3 \leq \mathrm{Bound}(r,R,\delta)
\end{multline*}
to get 
\begin{multline}\label{eq:approx-w-without-exp-pp}
\Bigg| - w(r) + \E\left[ \ind{E(r,R,\delta)} \int \left( w(r-x) - \frac{w(r-x)^2}{2} \right) \d\chi(x) \right] \\
- \frac{1}{2} \E\left[\ind{E(r,R,\delta)} \left(\int w(r-x) \d\chi(x) \right)^2\right] + \E\left[ \ind{E(r,R,\delta)} \int w(r-x) \d\chi(x) \right]^2 \Bigg| \\
\leq 22 \, \mathrm{Bound}(r,R,\delta) .
\end{multline}

\paragraph{Palm measures.} 

Recall \eqref{eq:def-ZI}, \eqref{eq:def-Palm} and \eqref{eq:def-Palm-mean}. 
We can rewrite the different terms in \eqref{eq:approx-w-without-exp-pp} as follows:
\begin{align*}
\E\left[ \ind{E(r,R,\delta)} \int \left( w(r-x) - \frac{w(r-x)^2}{2} \right) \d\chi(x) \right] = \E\left[ I^{(r,R,\delta)} \left( w(r-Z) - \frac{w(r-Z)^2}{2} \right) \right] 
\end{align*}
and
\begin{align*}
\E\left[ \ind{E(r,R,\delta)} \int w(r-x) \d\chi(x) \right] = \E\left[ I^{(r,R,\delta)} w(r-Z) \right] .
\end{align*}
Let us consider the last term more carefully. Because in the definition of $P^{(r,R,\delta)}_{z,i}$ we work with $\chi - \delta_{X_i}$, we have 
\begin{align*}
&\E\left[\ind{E(r,R,\delta)} \left\{ \left(\int w(r-x) \d\chi(x) \right)^2 - \int w(r-x)^2 \d\chi(x) \right\} \right] \\
&= \E\left[\ind{E(r,R,\delta)} \int w(r-x) \left( \int w(r-y) \d\chi(y) - w(r-x) \right) \d\chi(x) \right] \\ 
&= \E\left[\ind{E(r,R,\delta)} \int w(r-x) \left( \int w(r-y) \d(\chi-\delta_x)(y) \right) \d\chi(x) \right] \\ 
&= \E\left[ I^{(r,R,\delta)} w(r-Z) \int \left(\int w(r-y) \d(\phi-\delta_x)(y) \right) \d P^{(r,R,\delta)}_{Z,I^{(r,R,\delta)}} (\phi) \right] \\ 
&= \E\left[ I^{(r,R,\delta)} w(r-Z) \int w(r-y) \d M^{(r,R,\delta)}_{Z} (y) \right] .
\end{align*}
We can finally rewrite \eqref{eq:approx-w-without-exp-pp}:
\begin{multline}\label{eq:approx-w-without-exp-rw}
\Bigg| - w(r) + \E\left[ I^{(r,R,\delta)} \left( w(r-Z) - w(r-Z)^2 \right) \right] \\
- \frac{1}{2} \E\left[ I^{(r,R,\delta)} w(r-Z) \int w(r-y) \d M^{(r,R,\delta)}_{Z} (y) \right] + \E\left[ I^{(r,R,\delta)} w(r-Z) \right]^2 \Bigg| \\
= \Bigg| - w(r) + \E\Bigg[ I^{(r,R,\delta)} w(r-Z) \Big\{ 1 - w(r-Z) - \frac{1}{2} \int w(r-y) \d M^{(r,R,\delta)}_{Z} (y) \\ 
+ \E\left[ I^{(r,R,\delta)} w(r-Z) \right] \Big\} \Bigg] \Bigg| 
\leq 22 \, \mathrm{Bound}(r,R,\delta) .
\end{multline}

\paragraph{Introduce an exponential.}
Next, write 
\begin{align*}
\xi := - w(r-Z) - \frac{1}{2} \int w(r-y) \d M^{(r,R,\delta)}_{Z} (y) + \E\left[ I^{(r,R,\delta)} w(r-Z) \right] .
\end{align*}
From the definition of $P^{(r,R,\delta)}_{z,i}$, using $g(z,i,\phi) = i \max\left( \ind{\phi(\R)> \frac{\delta}{w(r-R)}-1} , \ind{\sup\phi > R} \right)$, 
\begin{align*}
 \E\left[ \int g\left(Z, I^{(r,R,\delta)}, \phi \right) \d P^{(r,R,\delta)}_{Z, I^{(r,R,\delta)}}(\phi) \right] = \E\left[ \sum_{i=1}^{\chi(\R)} g\left(X_i , \ind{E(r,R,\delta)}, \chi - \delta_{X_i} \right) \right] = 0 ,
\end{align*}
because for every $j\geq 1$ we have 
\[ \ind{j\leq \chi(\R)} \ind{E(r,R,\delta)} \ind{(\chi-\delta_{X_j})(\R) > \frac{\delta}{w(r-R)}-1} = 0 \quad \text{and} \quad  \ind{j\leq \chi(\R)} \ind{E(r,R,\delta)} \ind{\sup(\chi - \delta_{X_j}) > R} = 0 . \] 
Since $g \geq 0$ it means that we can choose $P^{(r,R,\delta)}_{z,i}$ such that for every $z$, 
\[ \int g\left(z, 1, \phi \right) \d P^{(r,R,\delta)}_{z,1}(\phi) = 0 = P_{z,1}^{(r,R,\delta)}\left( \left\{ \phi : \phi(\R) > \frac{\delta}{w(r-R)}-1 \text{ or } \sup\phi > R \right\}\right) . \]
We deduce from the definition of $M^{(r,R,\delta)}_z$ that for every $z$, 
\begin{equation}\label{eq:Palm-controle-M}
M^{(r,R,\delta)}_z(\R) \leq \frac{\delta}{w(r-R)}-1 \quad , \quad M^{(r,R,\delta)}_z((R,\infty)) = 0 .
\end{equation}
This allows us to bound $\xi$: since $Z \leq R$ on $\{ I^{(r,R,\delta)} = 1\}$, 
\begin{align}\label{eq:bound-xi}
\xi \leq w(r-R) \quad \text{and} \quad 
|\xi| \leq w(r-R) + \frac{1}{2} w(r-R) \frac{\delta}{w(r-R)} \leq w(r-R) + \frac{\delta}{2} .
\end{align}
Since $0 \leq e^x - (1+x) \leq x^2$ for every $x\leq 1$, 
\begin{multline*}
\left| \E\left[ I^{(r,R,\delta)} w(r-Z) (1 + \xi) \right] - \E\left[ I^{(r,R,\delta)} w(r-Z) \e^\xi \right] \right| \leq \E\left[ I^{(r,R,\delta)} w(r-Z) \xi^2 \right] \\
\leq \left( w(r-R) + \frac\delta 2\right) \E\Bigg[ I^{(r,R,\delta)} w(r-Z) \left\{ w(r-Z) + \frac{1}{2} \int w(r-y) \d M^{(r,R,\delta)}_{Z} (y) + \E\left[ I^{(r,R,\delta)} w(r-Z) \right] \right\} \Bigg] \\
\leq \left( w(r-R) + \frac\delta 2\right) \left( \frac{\sigma^2}{2} + 2\right) w(r-R)^2 ,
\end{multline*}
where we used that 
\[ \E\left[ I^{(r,R,\delta)} M^{(r,R,\delta)}_Z(\R) \right] = \E\left[ I^{(r,R,\delta)} \int \phi(\R) \d P^{(r,R,\delta)}_{Z,  I^{(r,R,\delta)}}(\phi) \right] = \E\left[ \ind{E(r,R,\delta)} \int (\chi(\R)-1) \d\chi(x) \right] \leq \sigma^2 . \]
Combining the estimates we have established so far gives Proposition \ref{prop:convolution-w}, with the remainder bounded by
\begin{multline*}
|\mathrm{Remainder}(r,R,\delta)| \leq \left(w(r-R)+\frac{\delta}{2}\right)\left(\frac{\sigma^2}{2}+2\right) w(r-R)^2 \\
+ 22 \left((*) + (**) + (\sigma^2 + 1) w(r-R)^2 \left( 2\delta + 4w(r-R) \right) \right) \\
\leq 22 \left((*) + (**) \right) + (\sigma^2 + 1) w(r-R)^2 \left( 45\delta + 90 w(r-R) \right) .
\end{multline*}
Increasing the numerical constants gives the Proposition.

\subsection{The martingale}
\label{sec:the-martingale}

We define a martingale by using Proposition \ref{prop:convolution-w} and the Markov property. This reflects \cite[Proposition 6]{lalley2015maximal}. Note that this martingale is not bounded, unlike in \cite{lalley2015maximal}; to use the optional stopping theorem, we will need to prove in Lemma \ref{lem:discount-UI} that it is uniformly integrable.

We use the convention that $\sum_{n=1}^0$ is zero. Recall $r_0$ and $\delta^0$ from Proposition \ref{prop:convolution-w}. Fix $r_\bullet \geq r_0$ and $r \geq 2 r_\bullet$, and let $(S_n)_{n\geq 0}$, $(R_n)_{n\geq 0}$, $(\delta_n)_{n\geq 0}$ and $(I_n)_{n\geq 1}$ such that $S_0 = r$ and for every $n\geq 0$, writing $\FF_n = \sigma\left( S_k, R_k, \delta_k, I_k : k\leq n\right)$, then conditionally on $\FF_n$
\begin{equation}
(S_{n+1} - S_n , I_{n+1}) \eqdistr (- Z, I^{(S_n,R_n, \delta_n)}) .
\end{equation}
In particular, $(S_n)_{n\geq 0}$ is a random walk started at $S_0 = r$ with i.i.d. steps, such that $S_n-S_{n+1}$ has distribution $M$ for every $n$.

\begin{proposition}\label{prop:almost-martingale}
For every $n\geq 0$, define the ``discount process'' 
\begin{multline}\label{eq:def-Z}
\kW^{(r)}_n := \left(\prod_{k=0}^{n-1} I_{k+1}\right) \exp\Bigg(- \sum_{k=0}^{n-1} \left( w(S_{k+1}) - \E\left[ I_{k+1} w(S_{k+1}) \, | \, \FF_k \right] \right) \\
- \frac 12 \sum_{k=0}^{n-1} \int w(S_k-y) \d M^{(S_k,R_k,\delta_k)}_{S_k-S_{k+1}}(y) \Bigg) . 
\end{multline}
Define the stopping time 
\[ T = \inf\{n\geq 0 : S_n < 2r_\bullet \text{ or } R_n \notin [r_\bullet, S_n - r_\bullet] \text{ or } \delta_n \notin (0,\delta^0) \} . \]
Define the processes 
\begin{equation}
W^{(r)}_n = \kW^{(r)}_n w(S_n) \quad , \quad Y_n := W^{(r)}_n - \sum_{k=0}^{n-1} \kW^{(r)}_k  \mathrm{Remainder}(S_k, R_k, \delta_k) .
\end{equation}
Then $(Y_{n\wedge T})_{n\geq 0}$ is a martingale, i.e. on the event $\{T > n\}$, 
\begin{equation}\label{eq:almost-mg-ppty}
\E\left[ W^{(r)}_{n+1} | \FF_n\right]  -  W^{(r)}_n = \kW^{(r)}_n \mathrm{Remainder}(S_n, R_n, \delta_n) ,
\end{equation}
where the remainder is given by Proposition \ref{prop:convolution-w}. 
\end{proposition}

\begin{proof}
It is a straightforward consequence of the Markov property with Proposition \ref{prop:convolution-w}. 
\end{proof}

Equivalently, we can add some $I_{k+1}$ in the expression: 
\begin{multline*}
\kW^{(r)}_n = \left(\prod_{k=0}^{n-1} I_{k+1}\right) \exp\Bigg(- \sum_{k=0}^{n-1} \left( I_{k+1} w(S_{k+1}) - \E\left[ I_{k+1} w(S_{k+1}) \, | \, \FF_k \right] \right) \\
- \frac 12 \sum_{k=0}^{n-1} I_{k+1} \int w(S_k-y) \d M^{(S_k,R_k,\delta_k)}_{S_k-S_{k+1}}(y) \Bigg) . 
\end{multline*}

\subsection{Control of the discount process}

\begin{lemma}\label{lem:discount-UI}
Up to taking $r_0$ larger, there exists $a_0>0$ such that for every $a \in (0,a_0)$, for every $r_\bullet \geq r_0$ and for every $r\geq 2r_\bullet$, writing $T_a = T \wedge \inf \{ n\geq 0 : R_n > a/\sqrt{w(S_n)} \}$, the stopped process $(\kW^{(r)}_{n\wedge T_a})_{n\geq 0}$ is a uniformly integrable supermartingale. 
\end{lemma}

\begin{proof}
Recall Lemma \ref{lem:lower-bound-w}: writing $C = \liminf_{r\to\infty} r^2 w(r)$ and taking $r_0$ larger if needed we have $r^2 w(r) \geq C/2$ for every $r\geq r_0$. Set $a_0 = \sqrt{C/8}$, then for every $a\in(0,a_0)$ and every $r\geq r_0$ we have $a/\sqrt{w(r)} \leq r/2$, hence for every $r_\bullet \geq r_0$, every $r\geq 2r_\bullet$ and every $R \leq a/\sqrt{w(r)}$ we have $r-R \geq r_\bullet \geq r_0$. In particular, $T_a$ is the same as $T$ with the condition on $R_n$ replaced by $R_n \notin [r_\bullet, a/\sqrt{w(S_n)}]$. 

Just as in the proof of Proposition \ref{prop:convolution-w}, write 
\[ \xi := -w(r-Z) + \E\left[I^{(r,R,\delta)} w(r-Z)\right] - \frac{1}{2} \int w(r-y) \d M_Z^{(r,R,\delta)}(y). \]
We first prove that for every $\delta \in (0,\delta^0)$, every $r\geq r_\bullet$ and every $r_\bullet \leq R \leq a/\sqrt{w(r)}$ we have $\E[I^{(r,R,\delta)} \e^{2\xi}] \leq 1$. Let us show how the lemma follows. Let $\xi_n$ be obtained from $\xi$ by replacing $r$ by $S_n$, $R$ by $R_n$ and $\delta$ by $\delta_n$: on $\{T_a>n\}$ it satisfies $\delta_n \in (0,\delta^0)$, $S_n\geq 2r_\bullet$ and $r_\bullet \leq R_n \leq a/\sqrt{w(r)}$, hence on $\{ T_a > n \}$, 
\begin{equation*}
\E\left[ \left(\kW^{(r)}_{n+1}\right)^2 \, | \, \FF_n\right] 
= \left(\kW^{(r)}_n\right)^2 \E\left[I^{(S_n, R_n, \delta_n)} \e^{2\xi_n} \ | \ \FF_n\right] 
\leq \left(\kW^{(r)}_n\right)^2 .
\end{equation*}
Together with $\kW^{(r)}_0=1$ we conclude that $\E[ (\kW^{(r)}_{n \wedge T_a})^2 ] \leq 1$ for every $n$, which by De La Vallée Poussin's theorem implies that $(\kW^{(r)}_{n \wedge T_a})_{n\geq 0}$ is uniformly integrable. On the other hand, on $\{T_a>n\}$
\[ \E\left[\kW^{(r)}_{n+1} \, | \, \FF_n \right] = \kW^{(r)}_{n} \E\left[I^{(S_n, R_n, \delta_n)} \e^{\xi_n} \, | \, \FF_n \right] . \]
Since $\E[I^{(S_n,R_n,\delta_n)} \e^{\xi_n} \, | \, \FF_n ] \leq (\E[I^{(S_n,R_n,\delta_n)}\e^{2\xi_n} \, | \, \FF_n ])^{1/2}\leq 1$ on $\{T_a>n\}$, the second half of the Lemma follows by induction. 

Let us now check that $\E[I^{(r,R,\delta)} \e^{2\xi}] \leq 1$. Write $I^{(r,R,\delta)}\xi = A + B$, with $A = -I^{(r,R,\delta)} w(r-Z) + \E\left[I^{(r,R,\delta)} w(r-Z)\right] $ and $B = - \frac{1}{2} I^{(r,R,\delta)} \int w(r-y) \d M_Z^{(r,R,\delta)}(y)$. Almost surely $|A| \leq w(r-R)$, and $\E[A]=0$. On the other hand, by \eqref{eq:Palm-controle-M} we have $M_z^{(r,R,\delta)}(\R) \leq \frac{\delta}{w(r-R)}-1$ for every $z, r, R, \delta$, hence $-\frac\delta 2 \leq B\leq 0$. 
In addition, 
\begin{align*}
\E[B] &= -\frac{1}{2}\E\left[ I^{(r,R,\delta)} \int w(r-y) \d M_Z^{(r,R,\delta)}(y) \right] \\
&= - \frac{1}{2} \E\left[ \ind{E(r,R,\delta)} \int \left(\int w(r-y) \d(\chi-\delta_x)(y)\right) \d\chi(x) \right] \\
&= - \frac{1}{2} \E\left[ \ind{E(r,R,\delta)} \int (\chi(\R)-1) w(r-x) \d\chi(x) \right] \\
&\geq -\frac{\sigma^2}{2}w(r-R) .
\end{align*}
Because for every $K\in\R$
\[ \E\left[ \ind{E(r,R,\delta)} (\chi(\R)-1) \chi([K, \infty)) \right] \to \E\left[ (\chi(\R)-1) \chi([K,\infty)) \right] \]
as $r_0\to\infty$ uniformly on every $\delta \in (0,\delta^0)$, $r\geq 2r_0$ and $r_0 \leq R \leq r-r_0$, and because this limit converges to $\sigma^2>0$ as $K\to-\infty$, we can find $c>0$ and $K\in\R$ such that, taking $r_0$ larger if necessary, for every $r_\bullet \geq r_0$, $r\geq 2r_\bullet$ and $r_\bullet \leq R \leq a/\sqrt{w(r)}$, 
\[ \E[B] \leq - c w(r-K) . \]
We then use that $\e^{x} \leq 1 + x + x^2$ for every $|x|\leq 1$ and taking $r_0$ larger if necessary so that $2w(r-R) \leq 1$, to get
\begin{align*}
\E\left[ I^{(r,R,\delta)} \e^{2(A+B)} \right] \leq \E\left[\e^{2B}\right] + \E\left[2A \e^{2B}\right] + \E\left[4A^2 \e^{2B} \right] .
\end{align*}
Since $B\leq 0$ a.s. the last term is bounded by $4\E[A^2] \leq 4 w(r-R)^2$. Taking $\delta^0$ smaller if necessary so that $\frac{1-e^{-\delta}}{\delta} \geq \frac 12$, we use that $\e^{x} \leq 1 + \frac{1-e^{-\delta}}{\delta} x \leq 1 + \frac{x}{2}$ for every $-\delta \leq x \leq 0$ to bound the first term by
\[ \E[\e^{2B}] \leq 1 + \frac{1}{2} \E[2B] \leq 1 - c w(r-K) . \]
For the middle term, write 
\[ \E[A\e^{2B}] = \E[A] + \E[A(\e^{2B}-1)] \leq w(r-R) \E[1-\e^{2B}] \leq 2w(r-R) \E[|B|] \leq \sigma^2 w(r-R)^2 \]
since $\E[|B|] = \E[-B] \leq \frac{\sigma^2}{2}w(r-R)$. We deduce that
\begin{equation}\label{eq:smaller-one}
\E\left[ I^{(r,R,\delta)} \e^{2(A+B)} \right] \leq 1 - c w(r-K) + (2\sigma^2 + 4) w(r-R)^2 .
\end{equation}
By Lemma \ref{lem:continuity-w-broad} and up to taking $a_0$ smaller, for every $a \in (0,a_0)$ we have $\sup_{r\geq 2r_0} \sup_{r_0\leq R \leq \frac{a}{\sqrt{w(r)}}} \frac{w(r-R)^2}{w(r-K)} \to 0$ as $r_0\to\infty$, 
hence up to taking $r_0$ larger if necessary the right-hand side of \eqref{eq:smaller-one} is smaller than $1$. 
This finishes the proof.
\end{proof}

\subsection{The Feynman--Kac representation}

\begin{corollary}\label{cor:Feynman-Kac}
There exists a function $\underline{\delta} : [0,\infty) \to (0,\delta^0)$ decreasing with $\underline{\delta}(y) \to 0$ as $y\to\infty$ and a function $g:(0,1] \to [0,\infty)$ increasing with $g(x)\to 0$ as $x\to 0$ such that the following holds. 

For every $a\in (0,a_0)$, $A>0$, $y>0$, $\RR>0$, and $r \geq x \geq r_0$ where $a_0, r_0$ is given by Lemma \ref{lem:discount-UI}, setting $\delta_n = \underline{\delta}(y)$ for every $n$ and writing 
\begin{equation*}
\begin{cases}
\tau_x = \inf\{ n\geq 0 : S_n \leq x \} \\
\rho = \inf\{n\geq 0 : S_n-R_n < y \text{ or } R_n < \RR\} \\
T' = \inf(\tau_x, T_a, \rho, A/w(x) \}
\end{cases}
\end{equation*}
then 
\begin{equation}\label{eq:FK}
\left| w(r) - \E\left[ W^{(r)}_{T'}\right] \right| \leq  \frac{100 A}{w(x)} \left( \P(\Lambda > \RR) + w(y)^2 g(w(y))\right) .
\end{equation}
\end{corollary}

\begin{proof}
By Proposition \ref{prop:almost-martingale} the process
\[ Y_n := W_n^{(r)} - \sum_{k=0}^{n-1} \kW^{(r)}_k \mathrm{Remainder}(S_k, R_k, \delta_k) , \]
is such that $(Y_{n\wedge T})_{n\geq 0}$ is a $(\FF_n)_{n\geq 0}$-martingale. Apply the stopping theorem with the bounded stopping time $T'$:
\begin{align}\label{eq:sum-for-FK}
w(r) = \E\left[ W^{(r)}_0 \right] = \E\left[ W_{T'}^{(r)} \right] - \E\left[ \sum_{k=0}^{T'-1} \kW^{(r)}_k \mathrm{Remainder}(S_k, R_k, \delta_k) \right] .
\end{align}
By Proposition \ref{prop:convolution-w} 
\[ \frac{1}{90} |\mathrm{Remainder}(r, R, \delta)| \leq \P(\Lambda>R) + \P\left(\chi(\R)>\frac{\delta}{w(r-R)}\right) + (\sigma^2+1) w(r-R)^2(w(r-R)+\delta) , \]
thus on $\{ k < T'\}$ we have
\begin{align*}
\frac{1}{90} |\mathrm{Remainder}(S_k, R_k, \delta_k)| \leq \P(\Lambda>\RR) + \P\left(\chi(\R)>\frac{\underline{\delta}(y)}{w(y)}\right) + (\sigma^2+1) w(y)^2(w(y)+\underline{\delta}(y)) .
\end{align*}
We can then find a function $\underline{\delta}(y)\to 0$ as $y\to\infty$ decreasing and positive, together with a function $g(x) \to 0$ as $x\to 0$ increasing and positive, such that 
\begin{align*}
\P\left(\chi(\R)>\frac{\underline{\delta}(y)}{w(y)}\right) + (\sigma^2+1) w(y)^2(w(y)+ \underline{\delta}(y)) \leq \frac{1}{100} w(y)^2 g(w(y)) .
\end{align*}
Since by Lemma \ref{lem:discount-UI} the process $(\kW^{(r)}_{n\wedge T_a})_{n\geq 0}$ is a supermartingale we have 
\[ \E\left[ \kW^{(r)}_{n\wedge T'} \right] \leq 1 , \]
and given that $T'\leq A/w(x)$ we bound the sum in the right-hand side of \eqref{eq:sum-for-FK} by $A/w(x)$ times the upper bound for the remainder. This finishes the proof. 
\end{proof}

\subsection{The ratio limit}
\label{sec:ratio-limit}

Now that we have established a Feynman--Kac representation for $w$, we follow the method described by \cite[Section 2.3]{lalley2015maximal}, which centers around the object $\phi$ defined as follows. Given a sequence $r_k \to \infty$ as $k\to\infty$, up to extracting a subsequence we can make sense of the limit
\begin{equation}\label{eq:phi}
\phi(y) = \lim_{k\to\infty} \frac{w\left( r_k + \frac{y}{\sqrt{w(r_k)}}\right)}{w(r_k)}
\end{equation}
for every fixed $y\geq 0$, and thus by a diagonal argument, jointly for every rational $y\geq 0$.

The Proposition 8 in \cite{lalley2015maximal} (our Proposition \ref{prop:continuity-phi}), their Lemma 11, and their Lemma 13 (see the end of Section \ref{sec:lower-bound-ratio}) follow easily from our Lemma \ref{lem:continuity-w-broad}. We prove the Feynman--Kac representation for the ratio $\phi$ in Proposition \ref{prop:func-equation-phi} (their Proposition 9), then deduce our Theorem \ref{th:main-tail} (their Theorem 1).

\begin{proposition}\label{prop:continuity-phi}
Any limit $\phi$ in \eqref{eq:phi} extends to a continuous, non-increasing, positive function of $[0,\infty)$. Hence, the convergence \eqref{eq:phi} holds uniformly over every compact of $[0,\infty)$. 
\end{proposition}

\begin{proof}
The fact that $\phi$ is decreasing and its continuity at $0$ follow from Lemma \ref{lem:continuity-w-broad}, while its positivity follows from Lemma \ref{lem:lower-bound-w}.  
We thus only need to check the continuity of $\phi$. 
Fix $y_1, y_2>0$, both rational, and let $r'_k = r_k + \frac{y_1}{\sqrt{w(r_k)}}$. First, 
\[ \lim_{k\to\infty} \frac{w(r'_k)}{w(r_k)} = \phi(y_1) . \]
By Lemma \ref{lem:continuity-w-broad}, 
\[ \liminf_{k\to\infty} \frac{w\left(r'_k + \frac{y_2}{\sqrt{w(r'_k)}}\right)}{w(r'_k)} \geq 1 - \frac{2\sigma}{\eta} y_2 . \]
Since $r'_k \geq r_k$ we have $\frac{y_2}{\sqrt{w(r'_k)}} \geq \frac{y_2}{\sqrt{w(r_k)}}$, thus 
\[ \phi(y_1) \geq \phi(y_1+y_2) \geq \liminf_{k\to\infty} \frac{w\left(r_k + \frac{y_1+y_2}{\sqrt{w(r_k)}}\right)}{w(r_k)} \geq \phi(y_1) \left(1 - \frac{2\sigma}{\eta} y_2\right) \to \phi(y_1) \]
as $y_2\to 0$, proving the continuity of $\phi$. 
\end{proof}


We now prove the following equivalent of \cite[Proposition 9]{lalley2015maximal}. Let $(B_u)_{u\geq 0}$ and $\E^x$ such that under $\E^x$, $(B_u)_{u\geq 0}$ is a standard Brownian motion with $B_0 = x$ a.s., and let $\tau_0^{B} = \inf\{t\geq 0 : B_t\leq 0\}$.  

\begin{proposition}\label{prop:func-equation-phi}
Any limit $\phi$ in \eqref{eq:phi} satisfies
\begin{equation}
\phi(y) = \E^{y/\eta}\left[ \exp\left(-\frac{\sigma^2}{2} \int_0^{\tau_0^B} \phi\left(\eta B_u\right) \d u \right) \right] \quad \text{ for every } y\geq 0 .
\end{equation}
\end{proposition}

\begin{proof}[Proof of Proposition \ref{prop:func-equation-phi}]
Fix $\nu>0$, $A>0$ and $y>0$, and let $r = x + \frac{y}{\sqrt{w(x)}}$ and $R_n = \RR = \frac{a}{\sqrt{w(x)}}$ for every $n$. Thanks to Lemma \ref{lem:continuity-w-broad} we can choose $a \in(0,a_0)$ such that that $w(x) \leq w(x-\RR) \leq (1+\nu) w(x)$ for every large enough $x$. 
Let and $\delta_n = \delta = \underline{\delta}(x-\RR)$ for every $n\geq 0$. Taking $x$ large enough ensures that $T_a > \tau_x$, hence $T' = \inf(\tau_x, A/w(x))$, and that $S_n-R_n \geq x-\RR$ on $\{T'>n\}$. 
We can then apply Corollary \ref{cor:Feynman-Kac}:
\[ \left| w(r) - \E\left[ W^{(r)}_{T'}\right] \right| \leq \frac{100 A}{w(x)}\left(\P(\Lambda>\RR) + w(x-\RR)^2 g(w(x-\RR))\right) \leq A w(x) g_1(x) \]
for some $g_1(x) \to 0$ as $x\to \infty$. 
We now aim to control 
\[ \E\left[ W^{(r)}_{T'}\right] = \E\left[ w(S_{T'}) \kW^{(r)}_{T'} \right] . \]

Recall the expression from $\kW^{(r)}_n$ in \eqref{eq:def-Z}: 
\begin{multline*}
\kW^{(r)}_n = \left(\prod_{k=0}^{n-1} I_{k+1} \right) \exp\Bigg( - \underbrace{\sum_{k=0}^{n-1} \left( I_{k+1} w(S_{k+1}) - \E\left[I^{(S_k, \RR, \delta)} w(S_k - Z) \, | \, \FF_k\right] \right)}_{M_n} \\
-\frac{1}{2} \underbrace{ \sum_{k=0}^{n-1} I_{k+1} \int w(S_{k}-y) \d M_{S_k - S_{k+1}}^{(S_k, \RR, \delta)}(y)}_{Y_n} \Bigg) . 
\end{multline*}
The process $(M_{n\wedge T'})_{n\geq 0}$ is a martingale with increments bounded in absolute value by $w(x-\RR)$, and we have almost surely $w(x-\RR) T' \leq w(x-\RR) \frac{A}{w(x)} \leq A(1+\nu)$. By Doob's maximal inequality, since $\E[M_{T'}^2] \leq (A/w(x)) (1+\nu)^2 w(x)^2 \to 0$ as $x\to 0$, there exists $c_x\to 0$ as $x\to\infty$ such that 
\begin{equation}\label{eq:controle-Z-1}
\P\left( \sup_{0\leq k \leq A/w(x)} |M_{k\wedge T'}| \geq c_x \right) \leq c_x .
\end{equation}
On the other hand, writing 
\[ X_k := I_{k+1} \int w(S_{k}-z) \d M_{S_{k}-S_{k+1}}^{(S_k, \RR, \delta)}(z) \quad , \quad \tilde Y_n = \sum_{k=0}^{n-1} \E\left[ X_k \ | \ \FF_k \right] \]
then $\tilde Y_n$ is a predictable process such that $\tilde M_n := (Y_{n\wedge T'} - \tilde Y_{n\wedge T'})_{n\geq 0}$ is a martingale. By \eqref{eq:Palm-controle-M}, $0\leq X_k \leq \delta$ on $\{T'>k\}$, meaning that $(\tilde M_n)_{n\geq 0}$ is a martingale with increments bounded in absolute value by $\delta$. Hence $\E[\tilde M_n^2] \leq \delta \E[Y_{n\wedge T'}]$. By Doob's maximal inequality, since $\delta\to 0$ as $x\to\infty$ it thus suffices to show that $\E[Y_{n\wedge T'}]$ remains bounded as $x\to\infty$ to conclude that there exists $c'_x\to 0$ as $x\to\infty$ such that 
\begin{equation}\label{eq:controle-Z-2}
\P\left( \sup_{0\leq k \leq A/w(x)} |\tilde M_{k\wedge T'}| \geq c'_x \right) \leq c'_x .
\end{equation}
Let us compute $\E[X_k \, | \, \FF_{k}]$: for every $u$ large enough, 
\begin{align*}
\E\left[ I^{(u,\RR,\delta)} \int w(u-z) \d M_{Z}^{(u, \RR, \delta)}(z) \right] &= \E\left[ \ind{E(u,\RR,\delta)} \int \left(\int w(u-\ell) \d(\chi-\delta_z)(\ell)\right) \d\chi(z) \right] \\
&= \E\left[ \ind{E(u,\RR,\delta)} (\chi(\R)-1)\int w(u-z)\d\chi(z) \right] .
\end{align*}
Using again Lemma \ref{lem:continuity-w-broad}, for every large enough $x$ we have $w(u) \leq w(u-\RR) \leq (1+\nu)w(u)$ for every $u\geq x$. As $u\to\infty$ we have in addition that $\inf_{u\geq x} \P(E(u,\RR,\delta)) = \P(E(x,\RR,\delta)) \to 1$ as $x\to\infty$ --- recall that $\underline{\delta}$ was chosen in the proof of Corollary \ref{cor:Feynman-Kac} in such a way that $\underline{\delta}(y)/w(y) \to \infty$ as $y\to\infty$. Then since $w(u-j)/w(u) \to 1$ for every fixed $j$ as $u\to\infty$ by Lemma \ref{lem:continuity-w-broad}, and by \eqref{eq:Palm-controle-M} and the dominated convergence theorem, we have as $u\to\infty$ 
\begin{align*}
\frac{1}{w(u)} \E\left[ I^{(u,\RR,\delta)} \int w(u-z) \d M_{Z}^{(u, \RR, \delta)}(z) \right] \to \E\left[ \chi(\R)(\chi(\R)-1)\right] = \sigma^2 .
\end{align*}
Writing $d_v = \sup_{u\geq v} \left|\frac{1}{\sigma^2 w(u)} \E\left[ I^{(u,\RR,\delta)} \int w(u-z) \d M_{Z}^{(u, \RR, \delta)}(z) \right] - 1\right|$, we have $d_v \to 0$ as $v\to \infty$. 
The first consequence is that $\E[Y_{T'}] \leq (1+d_x) (A/w(x)) \sigma^2 w(x)$ 
which remains bounded as $x\to\infty$, thus finishing the proof of \eqref{eq:controle-Z-2}. 
In addition, as $x\to\infty$, for every $n$
\begin{equation}\label{eq:controle-Z-3}
\left| \frac{\tilde Y_{n\wedge T'}}{\sigma^2 \sum_{k=0}^{n\wedge T'-1} w(S_k)} - 1 \right| \leq d_x \to 0 .
\end{equation}
Write $\AA$ for the intersection of the complement of the events in \eqref{eq:controle-Z-1} and \eqref{eq:controle-Z-2}; on $\AA$, for $x$ large enough we have for every $n$ 
\[ \left| \kW^{(r)}_{n\wedge T'} - \exp\left(-\frac{\sigma^2}{2} \sum_{k=0}^{n\wedge T'-1} w(S_k) \right) \right| \leq 2(c_x + c'_x + d_x) . \]
On the event $T' = \tau_x$, for every $n< T'$ we have $w(S_n) \leq w(x)$, and since $x-\RR \leq S_{\tau_x} \leq x $, we have $w(x) \leq w(S_{\tau_x}) \leq (1+\nu)w(x)$; this replaces \cite[Lemma 11]{lalley2015maximal}. We conclude that 
\begin{align*}
\left| \frac{w(r)}{w(x)} - \E\left[ \exp\left( - \frac{\sigma^2}{2} \sum_{j=1}^{T'} w(S_{j-1}) \right) \right] \right| \leq \nu + 3(c_x+c'_x+d_x) (1+\nu) + \P(\tau_x>A/w(x)) + A g_1(x) .
\end{align*}
Since $w$ is decreasing, $\left(\frac{S_{\lfloor Nu\rfloor}-x}{\sqrt{\lfloor Nu \rfloor}}\right)_{0\leq u\leq A} \to (\eta B_u)_{0\leq u\leq A}$ as $N\to\infty$ in distribution,
where $(B_u)_{u\geq 0}$ is a standard Brownian motion with $B_0=y/\eta$ a.s., and $w(x) \tau_x \to \tau^B_0$, taking $N = 1/w(x)$ and by the locally uniform convergence to $\phi$ continuous using Proposition \ref{prop:continuity-phi}, taking $x$ along the sequence used to define $\phi$
\[ \sum_{k=0}^{T'-1} w(S_k) = w(x) \sum_{k=0}^{T'-1} \frac{w(S_k)}{w(x)} \to \int_0^{\tau_0^B \wedge A} \phi(\eta (r-x+B_u)) \d u \]
in distribution as $x\to\infty$, hence 
\begin{align*}
\left| \phi(y) - \E^{y/\eta}\left[ \exp\left( - \frac{\sigma^2}{2} \int_0^{\tau_0^B \wedge A} \phi(\eta B_u) \d u \right) \right] \right| \leq \nu + \P^{y/\eta}(\tau^B_0>A)  .
\end{align*}
Since this holds for every $\nu$ and $A$, taking $A\to\infty$ and $\nu\to 0$ we finally get the statement of the Proposition. 
\end{proof}


We can use the rest of the proof of \cite{lalley2015maximal} directly: their Corollary 12 gives the limit form 
\[ \phi(y) = \left( \frac{\sigma y}{\eta\sqrt{6}} +1 \right)^{-2} \]
and their Lemma 13 follows from our observation at the end of Section \ref{sec:lower-bound-ratio}. The proof of their Theorem 1 adapts in a straightforward manner, giving our Theorem \ref{th:main-tail}.

\section{Asymptotic of the probability distribution function of the supremum}
\label{sec:proof-tail}

We prove Theorem \ref{th:main-condition} in this section. The idea is to show that $g(r) := \P(\sup_{v\in T} \Lambda_v = r)$ is sufficiently regular, namely that 
\[ \lim_{\varepsilon\to 0} \limsup_{r\to\infty} \sup_{(1-\varepsilon)r \leq y \leq (1+\varepsilon)r} \left| \frac{g(y)}{g(r)} - 1 \right| = 0 . \]
We then use the fact that 
$ w(r) \sim \sum_{y>r} g(y) $ as $r\to\infty$, and thus that 
\[ \sum_{(1-\varepsilon)r < y \leq (1+\varepsilon)r} g(y) \sim w(r(1-\varepsilon)) - w(r(1+\varepsilon)) \]
to deduce Theorem \ref{th:main-condition}.

\subsection{Probability that a random walk visits a point}
\label{sec:RW}

Let $(U_n)_{n\geq 0}$ be a random walk on $\Z$ with $U_0 = 0$ a.s. and i.i.d. steps with $\E[U_1]=0$ and $\E[U_1^2]=\eta^2<\infty$. The random walk is recurrent, so that writing $\tau_k$ for the $k$-th visit of $0$, we have $\tau_1 < \infty$ a.s. In particular, $ \P(\tau_k \leq n/2) \to \infty$ as $n\to\infty$. Then for every fixed $y\in\Z$, since by the strong Markov property $\P(\exists\, \tau_k \leq u < \tau_{k+1} : U_u = y) = \P(\exists\, 0 \leq u < \tau_{1} : U_u = y)>0$, we have 
\begin{equation}\label{eq:limit-P-being-at-a-point}
\P(\exists \, 0 \leq u \leq n/2 : U_u = y) \ulim n \infty 1  .  
\end{equation}
%
For every $y<0$, write $T_{y} = \inf\{n\geq 0, U_n \leq y\}$. Following \cite[Section 5.1.1]{lawler2010random}, by Theorem 5.1.7 therein, there exists a constant such that for every $n\geq 1$
\[ \P(T_y > n/2) \leq \cste \frac{1-y}{\sqrt{n}} , \]
and by Lemma 5.1.9, for every $m\geq 0$
\[ \P(U_{T_y} \leq y - m) \leq \cste \E\left[ U_1^2 \ind{|U_1| \geq m} \right] \ulim m \infty 0 . \]
For every $\varepsilon>0$, taking $m$ large enough that the right-hand side is smaller than $\varepsilon/2$, and using \eqref{eq:limit-P-being-at-a-point} to choose $n_\varepsilon$ such that for every $n\geq n_\varepsilon$
\[ \inf_{0\leq k \leq m} \P\left(\exists \, 0\leq u \leq \frac{n}{2}: U_u = k\right) \geq 1-\frac{\varepsilon}{2} , \]
we obtain that for every $y<0$ and every $n \geq n_\varepsilon$, by the strong Markov property
\begin{equation}\label{eq:proba-visit-RW}
\P\left(\exists \, 0 \leq u \leq n, U_u = y\right) \geq 1-\varepsilon-\cste \frac{1-y}{\sqrt{n}} . 
\end{equation}
We easily see that the same holds for $y>0$, up to a change in the constants involved.


\subsection{Lower bound on the probability distribution function of the supremum}
\label{sec:LLT}

In this section, we prove Theorem \ref{th:main-condition}. 
Assume that $\chi$ is supported on $\Z$, that $M$ has maximum span $1$, and that $\Lambda$ is $\N$-valued. 
Write $g(r) = \P\left(\sup_{v\in T} \Lambda_v = r\right)$. We obtain, similarly to Proposition \ref{prop:convolution-equation},
\begin{multline}\label{eq:convolution-LLT}
g(r) = \E\Bigg[ \ind{\Lambda = r} \prod_{u=1}^{\chi(\R)} h(r-X_u) \\
+ \ind{\Lambda < r} \sum_{j=1}^{\chi(\R)} \left(\prod_{u=1}^{j-1} h(r-1-X_u)\right) g(r-X_j) \left(\prod_{v=j+1}^{\chi(\R)} h(r-X_v) \right) \Bigg] \ind{r\geq 0} .
\end{multline}
Bound the first term inside the expectation from below by $0$. Since $h$ is increasing, $\e^{-x} \geq 1-x$, $\ln(1+x)\leq x$ hence $-\ln h(x) \leq w(x)$, and recalling $E(r,R,\delta)$ from \eqref{eq:event-E-rRd}, 
\begin{align}
g(r) &\geq \E\left[ \ind{E(r,R,\delta)} \int g(r-x) \left(1 + \int \ln h(r-1-y) \d(\chi-\delta_x)(y) \right) \d\chi(x) \right] \\
&\geq \E\left[ I^{(r,R,\delta)} g(r-Z) \left(1 - \int w(r-1-y) \d M_Z^{(r,R,\delta)}(y) \right) \right] .
\end{align}
By \eqref{eq:Palm-controle-M} we know that a.s. $M^{(r,R,\delta)}_Z(\R)\leq \frac{\delta}{w(r-R)}$ and $M_Z^{(r,R,\delta)}((r,\infty)=0)$. Taking $r>2r_0+1$ and $r_0+1\leq R \leq r/2$, on $\{I^{(r,R,\delta)}\neq 0\}$, for $M_Z^{(r,R,\delta)}$-a.e. $y$ we have $w(r-1-y) \leq w(r-R-1)$. 
This means that 
\begin{equation}\label{eq:bound-integral-M}
\int w(r-1-y) \d M_Z^{(r,R,\delta)}(y) \leq \delta \frac{w(r-R-1)}{w(r-R)} .
\end{equation}
Since $w(x-1)/w(x)\to 1$ as $x\to\infty$, we can assume that the right-hand side is smaller than $2\delta$ for every $x\geq r_0$ (up to taking $r_0$ larger). 
Taking $\delta\leq 1/4$, and using that $1 - x \geq \e^{- (2 \ln 2) x}$ for every $x\in (0,1/2)$, 
we obtain
\begin{align}\label{eq:supermg-g}
g(r) &\geq \E\left[ I^{(r,R,\delta)} g(r-Z) \exp\left(- (4\ln 2) \int w(r-y) \d M_Z^{(r,R,\delta)}(y) \right) \right] .
\end{align}
Let $\nu>0$ and $\varepsilon>0$. For every $y$ with $|y-r|\leq \nu r$, define $T_y := \inf\{n \geq 0 : S_n = y\}$. By \eqref{eq:proba-visit-RW}, there exists $c<\infty$ such that for every $n\geq n_\varepsilon$ 
\[ \P(T_y > n) \leq \varepsilon + c\frac{1+|y-r|}{\sqrt{n}} . \]
Assume henceforth that $n = \frac{c^2 \nu^2 r^2}{\varepsilon^2}$, so that $\P(T_y > n) \leq 2\varepsilon$ uniformly over every $|y-r|\leq \nu r$. 
Fix $R_j = R = r/8$ and $\delta_j = \delta$ for every $j$, and define the event $\AA := \{ \inf_{0\leq k \leq n} S_k \geq 3r/4 \}$. 
Since
\[ \P(I_{k+1} = 0 \, | \, \FF_k) \leq \P\left(\Lambda > R\right) + \P\left(\chi(\R)>\frac{\delta}{w(S_k-R)}\right) , \]
we can find $\delta = \delta(r) \to 0$ as $r\to\infty$ such that as $r\to\infty$, 
\[ \P\left(\AA \cap \left\{\prod_{k=0}^{n-1} I_{k+1} = 0 \right\}\right) \ulim r \infty 0 . \]
By the Markov property together with \eqref{eq:supermg-g}, if $T''(\omega) := \inf\{j : S_j < 3r/4\}$ and 
\begin{align*}
W_j := g(S_j) \left(\prod_{k=0}^{j-1} I_{k+1}\right) \exp\left(- (4\ln 2) \sum_{k=0}^{j-1} \int w(S_k-z) \d M_{S_k-S_{k+1}}^{(S_k, R, \delta)}(z) \right)
\end{align*}
then $(W_{j\wedge T''})_{j\geq 0}$ 
is a positive $(\FF_j)_{j\geq 0}$-supermartingale. By the optional stopping theorem on the bounded stopping time $n \wedge T_y \wedge T''$, and noting that $\{T''>n\} \supset \AA$, 
\begin{align}\label{eq:g-supermg}
g(r) \geq \E\left[ \ind{\AA} g(S_{n\wedge T_y}) \left(\prod_{k=0}^{n\wedge T_y-1} I_{k+1}\right) \exp\left(- (4\ln 2) \sum_{k=0}^{n\wedge T_y-1} \int w(S_k-z) \d M_{S_k-S_{k+1}}^{(S_k, R, \delta)}(z) \right) \right] .
\end{align}
Clearly $g(S_{T_y}) = g(y)$. We aim to bound the other factors in \eqref{eq:g-supermg}. By Doob's maximal inequality, 
\[ \P(\AA^c) = \P\left( \sup_{0\leq j \leq n} -(S_j-r) > r/4\right) \leq \frac{\E[(S_n-r)^2]}{(r/4)^2} \leq 16 \eta^2 \frac{c^2 \nu^2}{\varepsilon^2} \ulim \nu 0 0 . \]
On the other hand, recalling $d_r$ from the proof of Proposition \ref{prop:func-equation-phi}, 
\[ \E\left[ \ind{\AA} I_{k+1} \int w(S_k-z) \d M_{S_k-S_{k+1}}^{(S_k, R, \delta)}(z) \ | \ \FF_k \right] \leq w(r/2) \sigma^2 (1+d_{r/2}) \]
hence, assuming that $r$ is large enough that $d_{r/2}\leq 1$ and $w(r/2)\leq C r^{-2} / 2$ for some $C<\infty$
\[ \E\left[ \ind{\AA} \left(\prod_{k=0}^{j-1} I_{k+1}\right) \sum_{k=0}^{j-1} \int w(S_k-z) \d M_{S_k-S_{k+1}}^{(S_k, R, \delta)}(z) \right] \leq n w(r/2)\sigma^2 (1+d_{r/2}) \leq C \frac{\sigma^2 c^2 \nu^2}{\varepsilon^2} . \]
By Markov’s inequality, writing 
\[ \EE := \left\{ \ind{\AA} \left(\prod_{k=0}^{n-1} I_{k+1}\right) \sum_{k=0}^{n-1} \int w(S_k-y) \d M_{S_k-S_{k+1}}^{(S_k, R, \delta)}(y) \ > \ C \frac{\sigma^2 c^2 \nu^2}{\varepsilon^3} \right\} \]
we have $\P(\EE)\leq \varepsilon$ for every $r$ large enough, and outside of $\EE$, for every $0\leq j \leq n$
\begin{align*}
& \ind{\AA} \left(\prod_{k=0}^{j-1} I_{k+1}\right) \exp\left( - (4\ln 2)  \sum_{k=0}^{j-1} \int w(S_k-y) \d M_{S_k-S_{k+1}}^{(S_k, R, \delta)}(y) \right) \\
&\geq \ind{\AA}\left(\prod_{k=0}^{n-1} I_{k+1}\right) \exp\left( - (4\ln 2)  C \frac{c^2 \nu^2 \sigma^2}{\varepsilon^3} \right) .
\end{align*}
Take $\nu>0$ small enough and $r$ large enough that
\[ \exp\left( - (4\ln 2) C\frac{c^2 \nu^2 \sigma^2}{\varepsilon^3} \right) \geq 1-\varepsilon 
\quad \text{and} \quad 
\P\left(\AA \cap \left\{\prod_{j=1}^{n} I_j = 0 \right\}\right) \leq \varepsilon
\quad \text{and} \quad 
\P(\AA^c) \leq \varepsilon, 
\] 
and we get that on the event $\EE^c \cap \AA \cap \left\{\prod_{j=1}^{n} I_j = 0 \right\} \cap \{T_y\leq n\}$ of probability at least $1-5\varepsilon$, the content of the expectation in \eqref{eq:g-supermg} is larger than $g(y)(1-\varepsilon) $, hence 
\[ g(r) \geq (1-6\varepsilon) g(y) . \]
It follows that for every $\varepsilon>0$ small enough, we can find $\nu>0$ with $\nu\ulim \varepsilon 0 0$ such that for every $y$ with $|y-r|\leq \frac{\nu r}{2}$, we have for $r$ large enough 
\[ (1-6\varepsilon) g(r) \leq g(y) \leq (1+7\varepsilon) g(r) . \]
This means
\[ (1-6\varepsilon) r g(r) \leq \frac{1}{\nu} \sum_{y=(1-\nu/2)r}^{(1+\nu/2) r} g(y) \leq (1+7\varepsilon) r g(r) , \]
and since $\sum_{y>r} g(y) = 1-h(r) = \frac{w(r)}{1+w(r)}$, 
\[ \lim_{r\to\infty} \frac{r^2}{\nu} \sum_{y=(1-\nu/2)r}^{(1+\nu/2) r} g(y) = \lim_{r\to\infty} \frac{r^2}{\nu} (w((1-\nu/2)r) - w((1+\nu/2)r)) = \frac{6 \eta^2}{\sigma^2 \nu}\left( \left(1-\frac{\nu}{2}\right)^{-2} - \left(1+\frac{\nu}{2}\right)^{-2} \right) , \]
with the limit converging to $\frac{12 \eta^2}{\sigma^2}$ as $\nu\to 0$ (thus also as $\varepsilon\to 0$), 
we find that 
\[ \limsup_{r\to\infty} r^3 \left| g(r) - \frac{12 \eta^2}{\sigma^2 r^3} \right| \ulim \varepsilon 0 0 . \]
Since the left-hand side does not depend on $\varepsilon$, this concludes the proof of Theorem \ref{th:main-condition}.

\section{On the volume}
\label{sec:2Volume-general}

This section aims to establish Theorems \ref{th:main-volume-smaller} and \ref{th:main-condition-vol}. It is convenient for us to use Theorem \ref{th:main-tail}, hence this section does not supersede Section \ref{sec:main-proof}. After establishing preliminary estimates on the distribution of the volume without conditioning on the displacement in Section \ref{sec:2Laplace-volume-without-displacement}, followed by preliminary bounds on the quantities of interest in Section \ref{sec:2uniform-control} and \ref{sec:2coupling-RW}, we rephrase the quantities of interest in terms of the maximum displacement of a \emph{subcritical} \brw. This reframing allows us to then use the same approach as for the proof of Theorems \ref{th:main-tail} and \ref{th:main-condition}. The key Lemma establishing the ``continuity’’ of the ratio, analogous to Lemma \ref{lem:continuity-w-broad}, is established in Section \ref{sec:2continuity-tail}. From there the rest of the proof follows like in Section \ref{sec:main-proof}, see Figure \ref{fig:structure-sec-volume}.

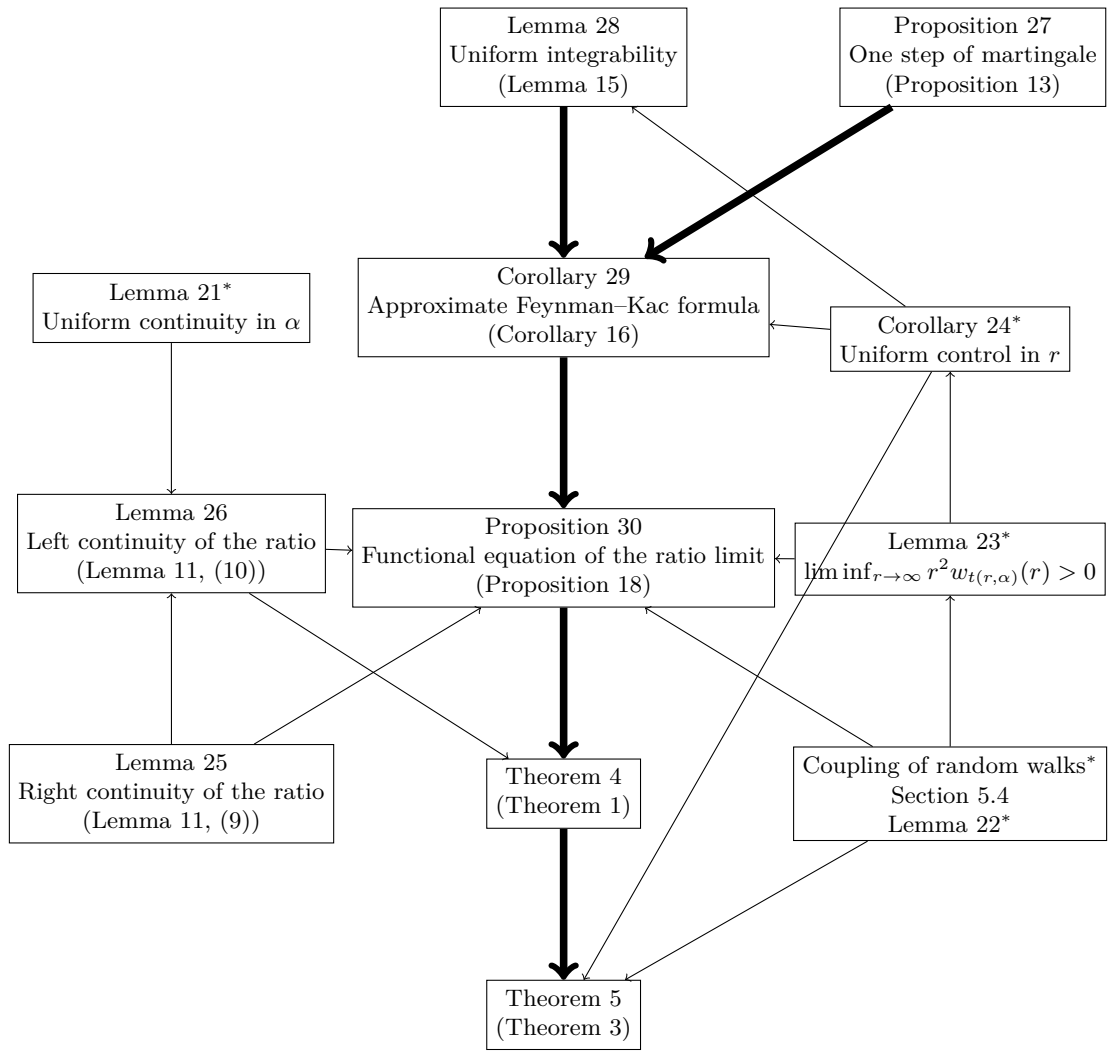
\begin{figure}
\begin{center}
\begin{tikzpicture}[node distance=2cm, auto, scale=1.5, every node/.style={rectangle, draw, align=center, font=\small}]

\node (L38) {Lemma \ref{lem:2discount-UI-vol}\\Uniform integrability\\(Lemma \ref{lem:discount-UI})};
\node (P37) [right=of L38] {Proposition \ref{prop:2convolution-w-vol}\\One step of martingale\\(Proposition \ref{prop:convolution-w})};
\node (C39) [below=of L38] {Corollary \ref{cor:2Feynman-Kac-vol}\\Approximate Feynman--Kac formula\\(Corollary \ref{cor:Feynman-Kac})};
\node (P40) [below=of C39] {Proposition \ref{prop:2func-equation-phi-vol}\\Functional equation of the ratio limit\\(Proposition \ref{prop:func-equation-phi})};
\node (theorem4) [below=of P40] {Theorem \ref{th:main-volume-smaller}\\(Theorem \ref{th:main-tail})};
\node (coupling) [right=of theorem4] {Coupling of random walks$^*$\\Section \ref{sec:2coupling-RW}\\Lemma \ref{lem:coupling-explicit}$^*$};
\node (L35) [left=of theorem4] {Lemma \ref{lem:2continuity-w-broad-volume}\\Right continuity of the ratio\\(Lemma \ref{lem:continuity-w-broad}, \eqref{eq:continuity-ratio-right})};
\node (L36) [above=of L35] {Lemma \ref{lem:2continuity-w-broad-volume-strong}\\Left continuity of the ratio\\(Lemma \ref{lem:continuity-w-broad}, \eqref{eq:continuity-ratio-left})};
\node (L32) [above=of coupling] {Lemma \ref{lem:positivity-c-alpha}$^*$\\$\liminf_{r\to\infty} r^2 w_{t(r,\alpha)}(r)>0$};
\node (L30) [above=of L36] {Lemma \ref{lem:2uniform-continuity-u}$^*$\\Uniform continuity in $\alpha$};
\node (C33) [above=of L32] {Corollary \ref{cor:2uniform-limit-r2wt}$^*$\\Uniform control in $r$};
\node (T5) [below=of theorem4] {Theorem \ref{th:main-condition-vol}\\(Theorem \ref{th:main-condition})};

\draw[->] (C33) -- (L38); 
\draw[->, line width=1mm] (P37) -- (C39);
\draw[->, line width=1mm] (L38) -- (C39); 
\draw[->] (C33) -- (C39); 
\draw[->] (L35) -- (L36);
\draw[->] (L32) -- (C33);
\draw[->] (L30) -- (L36); 
\draw[->] (L32) -- (P40); 
\draw[->] (L35) -- (P40); 
\draw[->] (L36) -- (P40); 
\draw[->, line width=1mm] (C39) -- (P40); 
\draw[->] (L36) -- (theorem4);
\draw[->, line width=1mm] (P40) -- (theorem4); 
\draw[->] (coupling) -- (P40);
\draw[->] (coupling) -- (L32);
\draw[->, line width=1mm] (theorem4) -- (T5); 
\draw[->] (C33) -- (T5);
\draw[->] (coupling) -- (T5);

\end{tikzpicture}
\end{center}
\caption{Dependency structure of the results of Section \ref{sec:2Volume-general}, starting from Section \ref{sec:2continuity-tail}. Results marked with an asterisk are new to this section, the others are adaptations of the result in parenthesis. The bold arrows represent the main line of the proof.}
\label{fig:structure-sec-volume}
\end{figure}

\subsection{Laplace transform of the volume}
\label{sec:2Laplace-volume-without-displacement}

We control the Laplace transform of the weight of $T$ conditionally on $\sup_{v\in T} \Lambda_v$ by considering for every $t\geq 0$
\[h_t(r) = \E\left[ \e^{-t \sum_{v\in T} D_v} \ind{\sup_{v\in T} \Lambda_v \leq r}\right] \quad , \quad h_t(\infty) = \E\left[ \e^{-t \sum_{v\in T} D_v}\right] . \]
The function $h_t$ is increasing and right-continuous, and under Assumptions \ref{assum} we have $h_t(0)<1$, $h_t(r)=0$ for every $r<0$, and $h_t(r)\to h_t(\infty)$ as $r\to\infty$. By the Markov property we have
\begin{equation}\label{eq:2convolution-equation-1}
h_t(r) = \E\left[ \e^{-t D} \prod_{i=1}^{\chi(\R)} h_t(r-X_i) \ind{\Lambda\leq r} \right] .
\end{equation}
When $r=\infty$, and writing $\Phi_t(x) = \E\left[\e^{-tD} x^{\chi(\R)}\right]$, this becomes 
\begin{align}\label{eq:2htinf-point-fixe}
h_t(\infty) = \Phi_t(h_t(\infty)) .
\end{align}
For every $t\geq 0$, the function $\Phi_t$ is convex, $\Phi_t(0)>0$ and $\Phi_t(1)\leq 1$, so $h_t(\infty)$ is the unique solution of the equation $x = \Phi_t(x)$ in $[0,1]$. Let us obtain its asymptotic. 

\begin{lemma}\label{lem:phit-taylor-expansion}
For every $p \in (1,\infty)$ and every random variable $W\in L^p$,
\[ \E\left[W \left(1 - \e^{-tD}\right)\right] = o\left(t^{\frac{p-1}{p}}\right) . \] 
For every $t\geq 0$ define
\[ f(t) := \E\left[\chi(\R) \e^{-tD}\right] -1 . \] 
Then $f(t) = o(t^{1/2})$, 
\begin{equation}\label{eq:phit-taylor-expansion}
\Phi_t(x) = x + \E\left[ \e^{-tD}-1 \right] + (x-1)f(t) + \frac{(x-1)^2}{2} \E\left[ (\chi(\R)^2 - \chi(\R))\e^{-t D} \right] + o((x-1)^2)
\end{equation} 
with a $o((x-1)^2)$ that is uniform over all $t\geq 0$, and as $t\to 0$ with $t\geq 0$
\begin{equation}\label{eq:htinfty-taylor-expansion}
1-h_t(\infty) \sim \sqrt{\frac{2 t}{\sigma^2}} \ .
\end{equation}
\end{lemma}

\begin{remark}
More generally, if $\mu_q(t) = \E[(1-\e^{-tD})^q ]$ then the Lemma holds with $1-h_t(\infty)\sim\sqrt{2\mu_1(t)/\sigma^2}$ as soon as $\mu_2(t) = o(\mu_1(t))$. This is the case if $\E[D]<\infty$, and there are cases where $\E[D]=\infty$ where this does not hold---typically when the tail of $D$ is not regular enough. 
\end{remark}


\begin{proof}
Using Hölder's inequality, letting $q = \frac{p}{p-1}$
\[ \left| \E\left[W \left(1 - \e^{-tD}\right)\right]\right| \leq \E\left[|W|^p\right]^{1/p} \E\left[\left(1 - \e^{-tD}\right)^q\right]^{1/q} . \]
Split the second expectation on $\{ tD>\varepsilon\}$ and $\{tD\leq \varepsilon\}$, and use $|\e^{-x}-1|\leq x$ for $x\geq 0$:
\begin{align*}
\E\left[\left(1 - \e^{-tD}\right)^q\right] &\leq  \E\left[\left(1 - \e^{-tD}\right)^q \ind{tD \leq \varepsilon}\right] +  \E\left[\left(1 - \e^{-tD}\right)^q \ind{tD>\varepsilon} \right] \\
&\leq \varepsilon^{q-1} \E[tD] + \P\left( D > \frac{\varepsilon}{t}\right) \\
&\leq \varepsilon^{q-1} t + \frac{2t}{\varepsilon \psi_D(\varepsilon/t)} 
\end{align*}
with $\psi_D$ given by Lemma \ref{lem:refined-Markov}: we can take $\varepsilon = \varepsilon(t)$ that goes to zero sufficiently slowly that the above expression is $o(t)$, proving our claim. The bound on $f$ follows from it with $p=2$ and $W = \chi(\R)$. Next is the Taylor expansion of $\Phi_t(x)$ near $x=1$ (with $x\leq 1$): for every $\varepsilon>0$
\begin{multline*}
\left| \Phi_t(x) - \E\left[ \e^{-tD}\left( 1 + \chi(\R) \ln x + \frac{\chi(\R)^2}{2} (\ln x)^2 \right) \right] \right| \\
\leq \E\left[ \left| \e^{\chi(\R)\ln x} - 1 - \chi(\R) \ln x - \frac{\chi(\R)^2}{2} (\ln x)^2 \right| \right] .
\end{multline*}
Split it on $\{ \chi(\R)\ln x < -\varepsilon\}$ and $\{\chi(\R)\ln x \geq -\varepsilon\}$, using $|\e^{-u} - 1 - u - u^2/2| \leq \omega(u) u^2$ for $u\leq 0$ with $\omega$ increasing and $\omega(u) \to 0$ as $u\to 0$:
\begin{equation}\label{eq:lem-phit-taylor-exp-1}
\E\left[ \left| \e^{\chi(\R)\ln x} - 1 - \chi(\R) \ln x - \frac{\chi(\R)^2}{2} (\ln x)^2 \right| \ind{\chi(\R)\ln x > -\varepsilon} \right] \leq \omega(\varepsilon) (\ln x)^2 \E\left[ \chi(\R)^2 \right] .
\end{equation}
On the other hand, using first $|\e^{-x}-1| \leq 1$, then Lemma \ref{lem:refined-Markov} to obtain $\psi_\chi$ (increasing to $+\infty$) and $C := \E[\chi(\R)^2 \psi_\chi(\chi(\R))]$,
\begin{align}
& \E\left[ \left| \e^{\chi(\R)\ln x} - 1 - \chi(\R) \ln x - \frac{\chi(\R)^2}{2} (\ln x)^2 \right| \ind{\chi(\R)\ln x \leq -\varepsilon} \right]  \nonumber \\
&\leq \P\left( \chi(\R) \leq \frac{-\varepsilon}{\ln x} \right) - (\ln x) \E\left[ \chi(\R)\ind{\chi(\R)\ln x \leq -\varepsilon} \right] + \frac{(\ln x)^2}{2} \E\left[ \chi(\R)^2 \ind{\chi(\R)\ln x \leq -\varepsilon} \right]  \nonumber \\
&\leq \frac{C}{\varepsilon^2 \psi_\chi(-\varepsilon/\ln x)} (\ln x)^2 + \frac{(\ln x)^2}{\varepsilon \psi_\chi(-\varepsilon/\ln x)} \E\left[ \chi(\R)^2 \psi_\chi(\chi(\R)) \ind{\chi(\R)\ln x \leq -\varepsilon} \right]  \nonumber \\
&\quad + \frac{(\ln x)^2}{2 \psi_\chi(-\varepsilon/\ln x)} \E\left[ \chi(\R)^2 \psi_\chi(\chi(\R)) \ind{\chi(\R)\ln x \leq -\varepsilon} \right] \nonumber \\
&\leq C\left( \varepsilon^{-2} + \varepsilon^{-1} + \frac 1 2 \right) \frac{(\ln x)^2}{\psi_\chi(-\varepsilon / \ln x)} .  \label{eq:lem-phit-taylor-exp-2}
\end{align}
We can find $\varepsilon = \varepsilon(x)$ that goes to $0$ sufficiently slowly as $x\to 1$ such that the right-hand side of \eqref{eq:lem-phit-taylor-exp-1} and \eqref{eq:lem-phit-taylor-exp-2} are both $o((x-1)^2)$, uniformly in $t\geq 0$. To prove \eqref{eq:phit-taylor-expansion} we need to control
\begin{align*}
&\Bigg| \E\left[ \e^{-tD}\left( 1 + \chi(\R) \ln x + \frac{\chi(\R)^2}{2} (\ln x)^2 \right) \right] \\
&\qquad - x - \E\left[ \e^{-tD}-1 \right] - (x-1) \left(\E\left[\chi(\R) \e^{-tD}\right]-1\right) - \frac{(x-1)^2}{2}\E\left[ \e^{-tD}\left( \chi(\R)^2 - \chi(\R)\right)\right] \Bigg| \\
&\leq \E\left[ \chi(\R) \left|\ln x - (x-1) + \frac{(x-1)^2}{2}\right| + \chi(\R)^2 \left|\ln(x)^2 - (x-1)^2\right| \right] = o((x-1)^2)
\end{align*}
uniformly in $t\geq 0$. 

Let us now prove \eqref{eq:htinfty-taylor-expansion} from the previous estimates. 
Since $\Phi_t'(x) = \E[\chi(\R) \e^{-t D} x^{\chi(\R)}]$ we have $0\leq \Phi_t'\leq 1$ on $[0,1]$; in addition $\Phi_t(1) \to 1$ as $t\to 0$, so that $h_t(\infty) \to 1$ as $t\to 0$. Hence using \eqref{eq:phit-taylor-expansion}, writing for simplicity $x := 1-h_t(\infty)$ in the rest of the proof
\begin{align*}
1-x = 1-x + \E\left[\e^{-tD}-1\right] - x f(t) + \frac{\sigma^2 + o(1)}{2} x^2 .
\end{align*}
Using that $\E[\e^{-tD}-1] \sim t$ as $t\to 0$ with $t\geq 0$, the solution must be
\begin{align*}
1-h_t(\infty) = \frac{-f(t) + \sqrt{f(t)^2 + 2 \sigma^2 t + o(t)}}{\sigma^2 + o(1)} \sim \sqrt{\frac{2 t}{\sigma^2}} \ .
\end{align*}
This finishes the proof of Lemma \ref{lem:phit-taylor-expansion}. 
\end{proof}

Finally, we can control the speed of convergence of certain expectations. 
\begin{lemma}\label{lem:speed-cv-expectation-Et}
Let $W \in L^2$, then as $t\to 0$
\begin{equation}
\E[W] - \E\left[ \e^{-tD} h_t(\infty)^{\chi(\R)-1} W \right] = \sqrt{\frac{2t}{\sigma^2}} \ \E[ W(\chi(\R)-1) ] + o(t^{1/2}) .
\end{equation}
\end{lemma}

\begin{proof}
Write
\begin{align*}
&\E[W] - \E\left[ \e^{-tD} h_t(\infty)^{\chi(\R)-1} W \right] - (1-h_t(\infty)) \E[ W(\chi(\R)-1) ]  \\
&= \E\left[ W\left( 1 - \e^{-tD}h_t(\infty)^{\chi(\R)-1} - (1-h_t(\infty))(\chi(\R)-1) \right)\right] \\
&= \E\left[ W\left(1 + (h_t(\infty)-1)(\chi(\R)-1) - h_t(\infty)^{\chi(\R)-1} + h_t(\infty)^{\chi(\R)-1}\left( 1-\e^{-tD} \right) \right)\right] .
\end{align*}
Then the absolute value of the left-hand side is smaller than 
\begin{align*}
\E\left[|W| \left|1 + (h_t(\infty)-1)(\chi(\R)-1) - h_t(\infty)^{\chi(\R)-1}\right|\right] + h_t(\infty)^{-1} \E\left[|W| \left( 1-\e^{-tD} \right)\right] .
\end{align*}
We use Lemma \ref{lem:phit-taylor-expansion} and the fact that $h_t(\infty)\to 1$ to bound the second term by $o(t^{1/2})$ as $t\to 0$. The Cauchy--Schwarz inequality and the same kind of control as in the proof of Lemma \ref{lem:phit-taylor-expansion} can be used to bound the first term by $o(1-h_t(\infty)) = o(t^{1/2})$. We conclude by using the asymptotic of $1-h_t(\infty)$ from Lemma \ref{lem:phit-taylor-expansion}. 
\end{proof}

\subsection{Preliminary bounds on the tail}
\label{sec:2uniform-control}

Let us obtain two easy bounds:
\[ h_t(r) = \E\left[\e^{-t\sum_{v\in T} D_v} \ind{\sup_{v\in T} \Lambda_v \leq r}\right] \leq \E\left[ \ind{\sup_{v\in T} \Lambda_v \leq r}\right] = \P(\sup_{v\in T} \Lambda_v \leq r) = h(r) \]
where $h(r)$ is the same as in Section \ref{sec:Markov-ppty-BRW}. On the other hand, since $1-ab \leq (1-a)+(1-b)$ for $a,b\in [0,1]$, 
\[ 1-h_t(r) \leq \E\left[ 1 - \e^{-t\sum_{v\in T} D_v} \right] + \E\left[1 - \ind{\sup_{v\in T}\Lambda_v \leq r} \right] = 2 - h_t(\infty) - h(r) . \]
To sum up, 
\begin{equation}\label{eq:easy-bounds-htr}
h(r) - (1- h_t(\infty)) \leq h_t(r) \leq h(r) .
\end{equation}

Define $w_t(r) = \frac{h_t(\infty)}{h_t(r)}-1$; this is well-defined for every $r\geq 0$ since $h_t(0)>0$, and it decreases to $0$ as $r\to\infty$. 
Let us compute bounds on $w_t$ from \eqref{eq:easy-bounds-htr}:
\begin{equation}\label{eq:2lower-bound-wt}
w_t(r) = \frac{h_t(\infty)}{h_t(r)}-1 \geq \frac{h_t(\infty)}{h(r)}-1 = h_t(\infty) w(r) - (1-h_t(\infty)) 
\end{equation}
and 
\begin{equation}\label{eq:2upper-bound-wt}
w_t(r) \leq \frac{h_t(\infty)}{h(r) -(1- h_t(\infty))}-1 = \frac{1-h(r)}{h(r)-(1-h_t(\infty))} \leq \frac{w(r)}{h(r) - (1-h_t(\infty))} .
\end{equation}
For every $\alpha\geq 0$ and $r>0$ define
\begin{equation}\label{eq:def-t-r-alpha}
t(r, \alpha) := \frac{\alpha^2}{2\sigma^2} \left( \frac{6\eta^2}{\sigma^2 r^2} \right)^2 \ .
\end{equation}
Since $w(r) \sim \frac{6\eta^2}{\sigma^2 r^2}$ as $r\to\infty$ and $1-h_t(\infty) \sim \sqrt{\frac{2t}{\sigma^2}}$ as $t\to\infty$, for every $\alpha\geq 0$ we find from \eqref{eq:2upper-bound-wt} 
\begin{equation*}
\limsup_{r\to\infty} r^2 w_{t(r,\alpha)}(r) \leq \limsup_{r\to\infty} r^2 w(r) \leq \frac{6\eta^2}{\sigma^2} 
\end{equation*}
and from \eqref{eq:2lower-bound-wt}
\begin{equation}\label{eq:2bound-w-liminf-1-alpha}
\liminf_{r\to\infty} r^2 w_{t(r,\alpha)}(r) \geq -\alpha\frac{6\eta^2}{\sigma^2} + \liminf_{r\to\infty} r^2 w(r) \geq (1-\alpha) \frac{6\eta^2}{\sigma^2} .
\end{equation}
Note that this immediately shows $w_{t(r,\alpha)}(r)\to 0$ as $r\to\infty$ for every $\alpha\geq 0$. 
A stronger version of these bounds is found in Lemma \ref{lem:positivity-c-alpha}. We can strengthen \eqref{eq:2lower-bound-wt} and \eqref{eq:2upper-bound-wt}: 
let $0\leq s \leq t$; by the same argument as around \eqref{eq:easy-bounds-htr}, 
\[ h_s(r) + h_t(\infty) - h_s(\infty) \leq h_t(r) \leq h_s(r) , \]
thus on one hand
\begin{equation}\label{eq:2lower-bound-wt-strong}
w_t(r) = \frac{h_t(\infty)}{h_t(r)} - 1 \geq \frac{h_t(\infty)}{h_s(r)}-1 = \frac{h_t(\infty)}{h_s(\infty)} w_s(r) + \frac{h_t(\infty)}{h_s(\infty)} - 1
\end{equation}
and on the other hand
\begin{multline}\label{eq:2upper-bound-wt-strong}
w_t(r) \leq \frac{h_t(\infty)}{h_s(r) + h_t(\infty) - h_s(\infty)}-1 = \frac{h_s(\infty)-h_s(r)}{h_s(r) + h_t(\infty) - h_s(\infty)} = \frac{1-\frac{h_s(r)}{h_s(\infty)}}{\frac{h_s(r)}{h_s(\infty)} + \frac{h_t(\infty)}{h_s(\infty)} - 1} \\
= \frac{w_s(r)}{-w_s(r) + \frac{h_t(\infty)}{h_s(\infty)}(1+w_s(r))} = \frac{w_s(r)}{\frac{h_t(\infty)}{h_s(\infty)} + w_s(r)\left(\frac{h_t(\infty)}{h_s(\infty)}-1\right)} . 
\end{multline}

\begin{remark}

A significant complexity is introduced by the need to control the asymptotic as $r\to\infty$ for every fixed $t$, as well as for $t$ and $r\to\infty$ jointly. Here is what we expect: as $t\to 0$ and $r\to\infty$, if $t = o(r^{-4})$ then $w_t(r) \sim w(r)$; on the other hand if $r = o(t^{-1/4})$ then $w_t(r) = o(w(r)) = o(r^{-2})$. Only in the regime where $t r^4$ is of order $1$, for example when $t=t(r,\alpha)$, will we see a limit in $(0,6\eta^2/\sigma^2)$. 
 
\end{remark}

A control of the regularity of $w_{t(r,\alpha)}(r)$ in $\alpha$ will be useful. 

\begin{lemma}\label{lem:2uniform-continuity-u}
For every $\alpha>0$ and every $\varepsilon>0$, there exists $r_\varepsilon$ and $f_\varepsilon(r) \to 6 \eta^2 \varepsilon / \sigma^2$ as $r\to\infty$ such that for every $r\geq r_\varepsilon$, 
\[ \sup_{\gamma, \beta \in [0,\alpha]: |\beta-\gamma|\leq \varepsilon} r^2 \left| w_{t(r,\beta)}(r) - w_{t(r,\gamma)}(r) \right| \leq f_\varepsilon(r) . \]
\end{lemma}

\begin{proof}
Since $1-h_t(\infty) \sim_{t\to 0} \sqrt{\frac{2t}{\sigma^2}}$, which implies $1-h_{t(r,\alpha)}(\infty) \sim_{r\to\infty} \alpha \frac{6\eta^2}{\sigma^2 r^2}$, and since $h_{t(r,\alpha)}(\infty)$ is decreasing in $\alpha$, we have that for every $\varepsilon, \varepsilon'>0$ and $\alpha>0$, for every $0\leq \beta \leq \alpha$ 
\[ \limsup_{r\to\infty} r^2\left(1 - \frac{h_{t(r,\beta+\varepsilon+\varepsilon')}(\infty)}{h_{t(r,\beta)}(\infty)} \right) \leq \frac{6\eta^2}{\sigma^2} (\varepsilon+\varepsilon') . \]
Taking the supremum over the finitely many $\beta \in [0,\alpha] \cap \varepsilon' \Z$ and by monotonicity of $\beta \mapsto h_{t(r,\beta)}(\infty)$, then taking $\varepsilon'\to 0$, we find $g_\varepsilon(r) \to 6\eta^2/\sigma^2$ as $r\to\infty$ such that 
\begin{equation*}
\sup_{0\leq \beta \leq \alpha} r^2\left(1 - \frac{h_{t(r,\beta+\varepsilon)}(\infty)}{h_{t(r,\beta)}(\infty)} \right) \leq \varepsilon g_\varepsilon(r) . 
\end{equation*}

Fix $\alpha>0$. 
By \eqref{eq:2upper-bound-wt}, there exists $f_1(x)$ (that depends on $\alpha$) that converge to $6\eta^2/\sigma^2$ as $x\to\infty$ with $f_1(x)\geq 6\eta^2/\sigma^2$ such that $r^2 w_{t(r,\beta)}(r) \leq \frac{6 \eta^2}{\sigma^2} f_1(r) $ for every $\beta \in [0,\alpha]$. 
By \eqref{eq:2lower-bound-wt-strong}, for every $\beta \in [0,\alpha]$
\begin{align*}
r^2 w_{t(r,\beta+\varepsilon)}(r) &\geq \left( 1 - \varepsilon g_\varepsilon(r)\right) r^2 w_{t(r,\beta)}(r) - \varepsilon g_\varepsilon(r) \\
&\geq r^2 w_{t(r,\beta)}(r) - \varepsilon g_\varepsilon(r) \left(1+ \frac{f_1(r)}{r^2}\right) .
\end{align*}
On the other hand, by \eqref{eq:2upper-bound-wt-strong}, taking $\varepsilon>0$ small enough, then $r$ large enough that $r^{-2} \varepsilon g_\varepsilon(r) \left( 1 + \frac{f_1(r)}{r^2} \right) \leq 1/2$, using that $|(1-x)^{-1} - 1| \leq 2x$ for $|x|\leq 1/2$, 
\begin{align*}
w_{t(r,\beta+\varepsilon)}(r) &\leq w_{t(r,\beta)}(r) \left( 1 - r^{-2} \varepsilon g_\varepsilon(r) \left( 1 + \frac{f_1(r)}{r^2} \right) \right)^{-1} \\
&\leq w_{t(r,\beta)}(r) + 2 w_{t(r,\beta)} r^{-2} \varepsilon g_\varepsilon(r) \left( 1 + \frac{f_1(r)}{r^2} \right) .
\end{align*}
The Lemma follows by taking 
\[ f_\varepsilon(r) = \max \left( 2f_1(r) r^{-2} \varepsilon g_\varepsilon(r) \left( 1 + \frac{f_1(r)}{r^2} \right) , \varepsilon g_\varepsilon(r) \left(1+ \frac{f_1(r)}{r^2}\right) \right) , \]
and we can check that indeed $f_\varepsilon(r) \to 6\eta^2 \varepsilon/\sigma^2$ as $r\to\infty$. 
\end{proof}

\subsection{Change of measure to subcritical}
\label{sec:2change-of-measure}

Recall the definition of a \brw from Section \ref{sec:preliminaries}. 
Define a family $(\E^{(t)})_{t\geq 0}$ of expectations on the same probability space as $\E$ such that $(k_u, \chi_u, D_u, \Lambda^{(u)}, (X^{(u)}_i)_{1\leq i \leq \chi_u(\R)})_{u\in \mathcal U}$ are i.i.d. with the distribution of $(k, \chi, D, \Lambda, (X_i)_{1\leq i \leq \chi(\R)})$ such that for every nonnegative and measurable $f$
\begin{equation}
\E^{(t)}\left[ f\left( k, \chi, D, \Lambda, (X_i)_{1\leq i \leq \chi(\R)} \right) \right] \\
= \E\left[ \e^{-tD} h_t(\infty)^{\chi(\R)-1} f\left( k, \chi, D, \Lambda, (X_i)_{1\leq i \leq \chi(\R)} \right) \right] .
\end{equation}
We can then check that for every nonnegative and measurable $f$
\begin{equation}
\E^{(t)}\left[ f\left( T, (X_v)_{v\in T}, (D_v)_{v\in T}, (\Lambda_v)_{v\in T} \right) \right] 
= \frac{\E\left[ \e^{-t\sum_{v\in T} D_v} f\left( T, (X_v)_{v\in T}, (D_v)_{v\in T}, (\Lambda_v)_{v\in T} \right) \right]}{\E\left[ \e^{-t\sum_{v\in T} D_v} \right]} \ ,
\end{equation}
noting that $h_t(\infty) = \E[ \e^{-t\sum_{v\in T} D_v} ]$. 
Differentiating $\Phi_t$ gives 
\begin{equation}\label{eq:2expectation-chi-vol}
\E^{(t)}\left[ \chi(\R) \right] = \Phi_t'(h_t(\infty)) \leq 1 
\end{equation}
with a strict inequality when $t>0$. 
In other words, the \brw becomes subcritical under $\E^{(t)}$. We can be more precise: by Lemma \ref{lem:speed-cv-expectation-Et} 
\begin{equation}\label{eq:asymptotic-expectation-chi-vol}
1- \E^{(t)}[ \chi(\R) ] \sim \sqrt{2t\sigma^2} .
\end{equation}
Because $\e^{-tD} h_t(\infty)^{\chi(\R)-1} \leq h_t(\infty)^{-1}$ a.s. we have by the dominated convergence theorem $\E^{(t)}[W] \to \E[W]$ as $t\to 0$ as soon as $W$ is integrable, in particular 
\begin{equation}\label{eq:2variance-chi-vol}
\E^{(t)}\left[ \chi(\R)(\chi(\R)-1) \right] \ulim t 0 \sigma^2 ,
\end{equation}
and there exists $t_0>0$ such that for every $t\in[0,t_0]$ and every $W\geq 0$ integrable 
\begin{equation}\label{eq:bound-Et-by-2E}
\E^{(t)}[W] \leq 2\E[W] .
\end{equation}
We can rewrite \eqref{eq:2convolution-equation-1} as an equation on 
\[ \tilde h_t(r) := \P^{(t)}\left( \sup_n \sup G_n \leq r \right) = \frac{h_t(r)}{h_t(\infty)} = \frac{1}{1+w_t(r)} \ , \]
namely
\begin{equation}\label{eq:2convolution-equation-vol}
\tilde{h}_t(r) = \E^{(t)}\left[ \prod_{i=1}^{\chi(\R)} \tilde h_t(r-X_i) \ind{\Lambda\leq r} \right] = \E^{(t)}\left[ \ind{\Lambda\leq r} \exp\left( \int \ln\tilde h_t(r-x) \d\chi(x) \right) \right] . 
\end{equation}
This interpretation allows us to draw an analogue with our work in Section \ref{sec:main-proof}, except that now the \brw is subcritical.

\subsection{On coupling the random walks}
\label{sec:2coupling-RW}

For every $t\geq 0$, there exists a random variable $Z_t$ such that for every positive and measurable $f$
\begin{equation}\label{eq:2def-Zt}
\E^{(t)}[\chi(\R)] \cdot \E[f(Z_t)] = \E^{(t)}\left[ \int f(x) \d\chi(x) \right] .
\end{equation}
If $t=0$ then $Z_0 =Z$ in distribution, where $Z$ is from Section \ref{sec:main-proof}. 
Since $\E^{(t)}[\chi(\R)] \to 1$, there exists $t_0>0$ such that 
for every positive and measurable $f$ and every $t\in [0,t_0]$
\[ \E[f(Z_t)] \leq 2 \E[f(Z_0)] . \]
In fact, if $\E[f(Z_0)]<\infty$ then $\E[f(Z_t)]\to \E[f(Z_0)]$ as $t\to 0$ by dominated convergence; in particular, $\E[Z_t^2] \to \eta^2$ and $\E[Z_t] \to 0$ as $t\to 0$.

Let us construct these random variables on the same probability space in a way that $\P(Z_t = Z_0)$ converges to $1$ as $t\to 0$. 
For every $t\geq 0$ let 
\[ V_t := \frac{\e^{-tD} h_t(\infty)^{\chi(\R)-1}}{\E^{(t)}[\chi(\R)]} . \]
Clearly $\E[V_t \chi(\R)]=1$. 
Let $U$ be uniform on $[0,1]$ and independent from the \repscheme, in particular of $(D,\chi)$, 
and for every $t\geq 0$ let $J_t$ be a random variable such that for every nonnegative and measurable $f$
\[ \E[f(Z, J_t)] = \E\left[ \int f\left( x, \ind{U < V_t}\right) \d\chi(x) \right] . \]
Finally, let $\tilde{Z_t}$ independent from $(Z, J_t)$ such that for every nonnegative and measurable $f$
\[ \E[f(\tilde Z_t)] = \E\left[ \frac{(V_t-1) \ind{V_t\geq 1}}{\E[(V_t-1) \ind{V_t\geq 1} \chi(\R)]} \int f(x) \d\chi(x) \right] . \]
Finally, define
\begin{equation}\label{eq:def-Zt-coupling}
Z_t := Z \ind{J_t=1} + \tilde Z_t \ind{J_t=0}
\end{equation} 

\begin{lemma}\label{lem:coupling-explicit}
The random variables $(Z_t)_{t\geq 0}$ from \eqref{eq:def-Zt-coupling} satisfy \eqref{eq:2def-Zt}, and in addition
\begin{equation}
\limsup_{t\to 0} \ (2\sigma^2 t)^{-1/2} \P(Z_t \neq Z) \leq 1 .
\end{equation}
and
\begin{equation}\label{eq:coupling-first-step}
\E[(Z_t - Z_0)^2] \ulim t 0 0 .
\end{equation}
\end{lemma}

\begin{proof}
Clearly, outside of the event $J_t=0$ we have $Z_t = Z$. This event has probability
\begin{equation}\label{eq:PJt0}
\P(J_t=0) = \E[\chi(\R) \ind{U>V_t}] = \E[\chi(\R) (1-V_t)\ind{V_t<1}] = \E[\chi(\R)\ind{V_t<1}] - \E[\chi(\R) V_t \ind{V_t<1}] .
\end{equation}
Using that $\E[\chi(\R)] = \E[V_t\chi(\R)] = 1$, we can rewrite this as 
\[ \P(J_t=0) = \E[\chi(\R) V_t \ind{V_t\geq 1}] - \E[\chi(R) \ind{V_t\geq 1}] = \E\left[ (V_t-1) \ind{V_t\geq 1} \chi(\R) \right] . \]
Then
\begin{align*}
\E[f(Z_t)] &= \E[\ind{J_t=1} f(Z)] + \E[\ind{J_t=0} f(\tilde Z_t)] \\
&= \E\left[ \ind{U<V_t} \int f(x)\d\chi(x) \right] + \P(J_t=0) \E\left[ f(\tilde Z_t) \right] \\
&= \E\left[ \min(1, V_t) \int f(x)\d\chi(x) \right] + \E\left[ (V_t-1) \ind{V_t\geq 1} \int f(x)\d\chi(x) \right] \\
&= \E\left[V_t \int f(x) \d\chi(x) \right] ,
\end{align*}
thus proving \eqref{eq:2def-Zt}. On the other hand, by Lemma \ref{lem:speed-cv-expectation-Et} and from $\E^{(t)}[\chi(\R)]\to 1$ as $t\to 0$
\[ \E[\chi(\R)\ind{V_t<1}] -  \E^{(t)}[\chi(\R)] \E[\chi(\R) V_t \ind{V_t<1}] = \sqrt{\frac{2t}{\sigma^2}} \E[(\chi(\R)-1)\chi(\R) \ind{V_t<1}] + o(t^{1/2}) . \]
Plugging this into \eqref{eq:PJt0}, 
\begin{equation*}
\P(Z_t \neq Z) \leq \P(J_t =0) \leq \E[\chi(\R)\ind{V_t<1}] - \E^{(t)}[\chi(\R)]\E[\chi(\R) V_t \ind{V_t<1}]
\end{equation*}
is such that 
\[ \limsup_{t\to 0} \frac{1}{\sqrt{2t \sigma^2}} \P(Z_t \neq Z) \leq \limsup_{t\to 0}  \frac{1}{\sigma^2} \E[(\chi(\R)-1)\chi(\R) \ind{V_t<1}] \leq 1 . \]
Finally, from $(a-b)^2 \leq 2a^2 + 2b^2$
\begin{align*}
\E[ (Z_t-Z)^2] &= \E\left[ \ind{J_t=0} (\tilde Z_t-Z)^2 \right] \\
&\leq 2\P(Z_t=0) \E\left[ (\tilde Z_t)^2 \right] + 2\E[\ind{J_t=0} Z^2] \\
&= 2\E\left[ (V_t-1) \ind{V_t\geq 1} \int x^2 \d\chi(x) \right] + 2\E[\ind{J_t=0} Z^2]
\end{align*}
and both terms converge to $0$ as $t\to 0$ by dominated convergence. 
\end{proof}

Let $(Z_n^{t})_{n\geq 1, t\geq 0}$ be i.i.d. copies of $(Z_t)_{t\geq 0}$, and define for every $t\geq 0$ and every $n\geq 0$
\begin{equation}\label{eq:def-Unt}
U_n^t = Z_1^t + \dots + Z_n^t
\end{equation}
where $U_0^t = 0$. 
The many-to-one formula takes the following form: for every measurable and bounded $g$, for every $n$
\begin{equation}\label{eq:many-to-one-chiR}
\E^{(t)}\left[\int g \d G_n\right] = \E^{(t)}[\chi(\R)]^n \cdot \E[g(U_n^t)] . 
\end{equation} 
A first consequence of the many-to-one formula is the following.

\begin{lemma}\label{lem:positivity-c-alpha}
For every $\alpha\geq 0$, 
\[ 0 < c(\alpha) := \liminf_{r\to\infty} r^2 \inf_{t\in[0,t(r,\alpha)]} w_t(r) \leq \limsup_{r\to\infty} r^2 \sup_{t\in[0,t(r,\alpha)]} w_t(r) \leq \frac{6\eta^2}{\sigma^2} . \]
\end{lemma}

\begin{proof}
The upper bound follows from \eqref{eq:2upper-bound-wt}. 
We prove the lower bound by the second moment method. By the many-to-one formula \eqref{eq:many-to-one-chiR}, 
\[ \E^{(t)}[G_n([r, \infty))] = \E^{(t)}[\chi(\R)]^n \P(U^t_n \geq r) . \]
The second moment is
\begin{align*}
\E^{(t)}[G_n(\R)^2] 
&= \E^{(t)}\left[ \left(\sum_{x\in G_{n-1}} \chi_x(\R)\right)^2\right] \\
&= \E^{(t)}[G_{n-1}(\R)] \E^{(t)}[\chi(\R)^2] + \E^{(t)}[G_{n-1}(\R)(G_{n-1}(\R)-1)] \E^{(t)}[\chi(\R)]^2 \\
&= \E^{(t)}[G_{n-1}(\R)^2] \E^{(t)}[\chi(\R)]^2 + \E^{(t)}[G_{n-1}(\R)] \Var^{(t)}(\chi(\R))
\end{align*}
so that
\begin{align}\label{eq:second-moment-Gr2}
\E^{(t)}[G_n(\R)^2] &\leq \E^{(t)}[\chi(\R)]^{2n} + n\Var^{(t)}(\chi(\R)) \E^{(t)}[\chi(\R)]^{n-1} \\
&\nonumber \leq \E^{(t)}[\chi(\R)]^{n-1}\left(1+n\Var^{(t)}(\chi(\R))\right) \\
&\nonumber \leq 2n\E^{(t)}[\chi(\R)]^{n}\sigma^2 
\end{align}
as soon as $t, n$ and $r$ are large enough. Taking $n = r^2/\eta^2$, by the Paley--Zygmund inequality we get
\begin{equation*}
\frac{r^2 w_t(r)}{1+w_t(r)} = r^2 \P^{(t)}(G_{r^2}([r,\infty))>0) \geq \E^{(t)}[\chi(\R)]^{r^2} \frac{\eta^2 \P(U_{r^2}^t \geq r)^2}{2\sigma^2} .
\end{equation*}
Using \eqref{eq:coupling-first-step} we can check that if we take $t_0>0$ small enough then $\liminf_{r\to\infty} \inf_{t\in [0,t_0]} \P(U_{r^2/\eta^2}^t \geq r) > \P(N\geq 2)$, where $N\sim \NN(0,1)$.  
Using that $\E^{(t)}[\chi(\R)] = 1 - \sqrt{2t\sigma^2} + o(t) \geq \exp( - 2\sqrt{t\sigma^2})$ for every $t\in [0,t_0]$ (taking $t_0$ smaller if needed), we finally find 
\[ \liminf_{r\to\infty} r^2 \inf_{t\in [0, t(r,\alpha)]} w_t(r) \geq \frac{\P(N\geq 2)^2}{2}\frac{\eta^2}{\sigma^2} \exp\left(-\alpha\frac{6\sqrt{2}}{\sigma^2} \right) . \]
\end{proof}

\begin{corollary}\label{cor:2uniform-limit-r2wt}
For every $\alpha>0$ and $b \in (0,1)$, 
\[ c(\alpha) \leq \liminf_{r\to\infty} \inf_{br \leq x \leq r} x^2 w_{t(r,\alpha)}(x) \leq \limsup_{r\to\infty} \sup_{br \leq x \leq r} x^2 w_{t(r,\alpha)}(x) \leq \frac{6\eta^2}{\sigma^2} . \]
More generally, 
\[ c(\alpha) \leq \liminf_{r\to\infty} \inf_{s\in [0,t(r,\alpha)]} \inf_{br \leq x \leq r} x^2 w_{s}(x) \leq \limsup_{r\to\infty} \sup_{s\in [0,t(r,\alpha)]}  \sup_{br \leq x \leq r} x^2 w_{s}(x) \leq \frac{6\eta^2}{\sigma^2} . \]
\end{corollary}

\begin{proof}
For the second part, we use the correspondence $t(r,\alpha) = t(br, b^2 \alpha)$. Fix $\varepsilon>0$ and assume that $r_\bullet$ is such that for every $r\geq r_\bullet$
\[ (1-\varepsilon)c(\alpha) \leq \inf_{0\leq t \leq t(r,\alpha)} r^2 w_t(r) \leq \sup_{0\leq t \leq t(r,\alpha)} r^2 w_t(r) \leq (1+\varepsilon) \frac{6\eta^2}{\sigma^2} . \] 
Then for every $r\geq r_\bullet/b$ and every $q \in [b,1]$, because $q^2 \alpha \leq \alpha$, 
\[ (1-\varepsilon)c(\alpha) \leq (qr)^2 w_{t(qr, q^2 \alpha)}(qr) = (qr)^2 w_{t(r, \alpha)}(qr) \leq (1+\varepsilon) \frac{6\eta^2}{\sigma^2} . \]
This concludes the proof of the first equation. The second follows from the first and \eqref{eq:2upper-bound-wt-strong} (for the lower bound) and \eqref{eq:2upper-bound-wt} (for the upper bound). 
\end{proof}

\subsection{Continuity of the tail}
\label{sec:2continuity-tail}

In this section, we gathered technical estimates which play the role of Lemma \ref{lem:continuity-w-broad}. Lemma \ref{lem:2continuity-w-broad-volume} is its straightforward adaptation; we will however need the stronger version found in Lemma \ref{lem:2continuity-w-broad-volume-strong}.

\begin{lemma}\label{lem:2continuity-w-broad-volume}
For every $\alpha\geq 0$, 
define 
\[ D(\alpha) = \frac{4}{\eta\sqrt{\pi}} \sqrt{\sigma^2 + \frac{\alpha}{c(\alpha)}} \ . \]
For every $y\geq 0$
\begin{equation}\label{eq:2continuity-ratio-vol-right}
\liminf_{r\to\infty} \frac{w_{t(r, \alpha)}\left(r + \frac{y}{\sqrt{w_{t(r, \alpha)}(r)}}\right)}{w_{t(r, \alpha)}(r)} \geq 1 - y D(\alpha) .
\end{equation}
\end{lemma}


\begin{proof}
In this proof, for conciseness we write $t = t(r,\alpha)$. We follow the proof of Lemma \ref{lem:continuity-w-broad}, with adaptations. 
Let $r_y = \frac{y}{\sqrt{w_t(r)}}$. 
Assume that $r$ is large enough that $u_t(r) \geq c(\alpha)/2$.
By \eqref{eq:2convolution-equation-vol}, using that $\tilde h_t(x) \leq \tilde h_t(r)$ when $x\leq r$, 
\[ \tilde h_t(r+r_y) \leq \E^{(t)}\left[ \exp\left( \ln \tilde h_t(r) G_n([r_y,\infty)) + \int_{(-\infty,r_y]} \ln\tilde h_t(r+r_y-x) \d G_n(x) \right) \right] . \]
Write $X$ for the expression inside the exponential. As in Lemma \ref{lem:continuity-w-broad}, we bound $\E[\e^{X}] \leq \E[1 + X + \frac{X^2}{2}]$, and do so by bounding the first and second moments of $X$ under $\E^{(t)}$: by the many-to-one formula \eqref{eq:many-to-one-chiR}
\begin{align}
\nonumber & \E^{(t)}[\chi(\R)]^{-n} \frac{-\E^{(t)}[X]}{w_t(r)} \\
\nonumber &= \E\left[ \frac{\ln(1+w_t(r))}{w_t(r)} \ind{U_n^t \geq r_y} + \frac{\ln(1+w_t(r+r_y-U_n^t))}{w_t(r)}\ind{U_n^t < r_y} \right] \\
\nonumber &\geq \frac{\ln(1+w_r(t))}{w_t(r)}\P(U_n^t \geq r_y) + \E\left[ \frac{\ln(1+w_t(r+r_y-U_n^t))}{w_t(r)} \ind{U_n^t<r_y , |U_n^t-U_n^0|\leq \delta\sqrt{n}} \right]\\
\nonumber &\geq \frac{\ln(1+w_r(t))}{w_t(r)}\P(U_n^0 \geq r_y-\delta\sqrt{n}) \\
\nonumber &\qquad - \underset{(a)}{\underbrace{\P\left(|U_n^t - U_n^0|>\delta\sqrt{n}\right)}} - \underset{(b)}{\underbrace{\P\left( U_n^0 \geq r_y-\delta\sqrt{n}, U_n^t < r_y, |U_n^t-U_n^0|\leq \delta\sqrt{n} \right)}} \\
\nonumber &\qquad + \E\left[ \frac{\ln(1+w_t(r+r_y-U_n^0 + \delta\sqrt{n}))}{w_t(r)} \ind{U_n^0 < r_y+\delta\sqrt{n}} \right] \\
&\qquad - \underset{(c)}{\underbrace{\P\left( |U_n^t - U_n^0| > \delta\sqrt{n} \right)}} - \underset{(d)}{\underbrace{\P\left( |U_n^t-U_n^0|\leq \delta\sqrt{n}, U_n^t\geq r_y, U_n^0<r_y+\delta\sqrt{n} \right)}} \label{eq:discounted-expectation-X}
\end{align}
where we used that $w_t$ is a decreasing function and $0< \ln(1+x)/x \leq 1$ for $x>0$. Terms $(b)$ and $(d)$ are clearly smaller than
\[ \P(U_n^0 \in [r_y - \delta\sqrt{n}, r_y+\delta\sqrt{n}) ) . \]
The central limit theorem gives us a sequence $\omega_1(n) \to 0$ as $n\to\infty$ and a constant $c$ such that this probability is smaller than $c\delta + \omega_1(n)$. On the other hand, by Bienaymé--Chebyshev's inequality
\[ \P(|U_n^t-U_n^0|>\delta\sqrt{n}) \leq \frac{\E[(U_1^t - U_1^0)^2]}{\delta^2} , \]
which goes to zero as $t\to 0$ by \eqref{eq:coupling-first-step}. 

Let $(B_u)_{u\geq 0}$ be a centered standard Brownian motion. Fix some $s>0$, and from now on set $n = n(r) = \lfloor \eta^{-2} s w_t(r)^{-1}\rfloor$. The sequence $U_n^0 (\eta^2 n)^{-1/2}$ converges to $B_1$ in distribution as $r\to\infty$, so that $U_n^0/r_y$ converges to $y^{-1} B_s$. 
Assume that \eqref{eq:2continuity-ratio-vol-right} holds with the right-hand side being replaced by some bounded, non-increasing function $f(y)$. Then 
\begin{multline*}
\liminf_{r\to\infty} \E\left[ \frac{\ln\left(1+w_t\left(r+r_y+\delta\sqrt{n} - U_n^0\right)\right)}{w_t(r)} \ind{U_n^0 < r_y\left(1 + \frac{\delta \sqrt{s}}{\eta}\right)} \right] \\
\geq \E\left[ f\left( y\left[ 1 + \frac{\delta\sqrt{s}}{\eta} \right] - B_s \right) \ind{y\left[ 1 + \frac{\delta\sqrt{s}}{\eta} \right] - B_s > 0} \right] .
\end{multline*}
We can check this by taking a coupling of the sequence $U_n^0 (\eta^2 n)^{-1/2}$ that converges a.s. towards $B_1$ by the Skorokhod representation theorem, then using the monotonicity of $f$ with Fatou's Lemma. Clearly $f(y) = \ind{y\leq 0}$ works. 

Since $w_t(r)\to 0$ as $r\to\infty$, taking the liminf of \eqref{eq:discounted-expectation-X} as $r\to\infty$, 
\begin{multline*}
\liminf_{r\to\infty} \E^{(t)}[\chi(\R)]^{-n} \frac{-\E^{(t)}[X]}{w_t(r)} \\
\geq \P\left(B_s \geq y-\frac{\delta\sqrt{s}}{\eta} \right) + \E\left[ f\left( y\left[ 1 + \frac{\delta\sqrt{s}}{\eta} \right] - B_s \right) \ind{y\left[ 1 + \frac{\delta\sqrt{s}}{\eta} \right] > B_s} \right] - 2c\delta .
\end{multline*}
Take $\delta\to 0$:
\begin{equation}\label{eq:liminf-EX}
\liminf_{r\to\infty} \E^{(t)}[\chi(\R)]^{-n} \frac{-\E^{(t)}[X]}{w_t(r)}
\geq \P(B_s\geq y) + \E\left[\ind{B_s < y} f\left(y-B_s\right)\right] .
\end{equation}
On the other hand, by \eqref{eq:second-moment-Gr2} 
and taking $t_0>0$ small enough that $\E^{(t)}[\chi(\R)]^{-1} \Var^{(t)}(\chi(\R))\leq 2\sigma^2$ for every $t\leq t_0$, 
we get uniformly in $t\in[0,t_0]$, 
\begin{align*}
\E^{(t)}[X^2] \leq (\ln \tilde h_t(r))^2 \E^{(t)}[G_n(\R)^2] \leq (\ln \tilde h_t(r))^2 \E^{(t)}[\chi(\R)]^n (1+2n\sigma^2) ,
\end{align*}
Since $n \ln \tilde h_t(r) \ulim r \infty s \eta^{-2}$ as $r\to\infty$ from the definition of $n$, 
\begin{equation}\label{eq:liminf-EX2}
\limsup_{r\to\infty} \E^{(t)}[\chi(\R)]^{-n} \frac{\E^{(t)}[X^2]}{w_t(r)} \leq \frac{2 s \sigma^2}{\eta^2} .
\end{equation}
Combining \eqref{eq:liminf-EX} and \eqref{eq:liminf-EX2}, 
\begin{equation}\label{eq:2why-t-must-vary}
\liminf_{r\to\infty} \frac{w_t(r+r_y)}{w_t(r)} \E^{(t)}[\chi(\R)]^{-n} 
\geq \E\left[ \ind{B_s\geq y} + \ind{B_s<y} f\left(y-B_s\right)\right] - s \frac{\sigma^2}{\eta^2} . 
\end{equation}
Recall from \eqref{eq:asymptotic-expectation-chi-vol} that $1-\E^{(t)}[\chi(\R)] \sim \sqrt{2t\sigma^2}$ as $t\to 0$. Since $\liminf_{r\to\infty} r^2 w_t(r) \geq c(\alpha) \frac{6\eta^2}{\sigma^2} > 0$ by Lemma \ref{lem:positivity-c-alpha} and since $t = t(r,\alpha)$ we get 
\begin{align*}
\liminf_{r\to\infty} \frac{w_t(r+r_y)}{w_t(r)} &\geq \exp\left(- \frac{s}{w_t(r) \eta^2}\frac{\alpha}{r^2} \right) \left\{ \E\left[ \ind{B_s\geq 1} + \ind{B_s<1} f\left(y-yB_s\right)\right] - s \frac{\sigma^2}{\eta^2} \right \} \\
&\geq \E\left[ \ind{B_s\geq 1} + \ind{B_s<1} f\left(y-yB_s\right)\right] - \frac{s}{\eta^2}\left(\sigma^2 + \frac{\alpha}{c(\alpha)}\right) .
\end{align*}
The rest of the proof follows that of Lemma \ref{lem:continuity-w-broad}:
\begin{align*}
\liminf_{r\to\infty} \frac{w_t(r+r_y)}{w_t(r)} &\geq \P(\tau_y \leq x) - \left( \frac{\sigma^2}{\eta^2} + \frac{\alpha}{\eta^2 c(\alpha)} \right) \E\left[\tau_y \ind{\tau_y\leq x}\right] \\
&\geq 1 - \frac{2}{\sqrt{2\pi}} \inf_{x>0} \left( \frac{1}{\sqrt{x}} + y^2 \sqrt{x} \frac{2}{\eta^2}\left(\sigma^2 + \frac{\alpha}{c(\alpha)}\right) \right) \\
&= 1 - \frac{4y}{\eta\sqrt{\pi}} \sqrt{\sigma^2 + \frac{\alpha}{c(\alpha)}} \ .
\end{align*}
\end{proof}

\begin{lemma}\label{lem:2continuity-w-broad-volume-strong}
For every $\alpha>0$ and $y\geq 0$, recalling $D(\alpha)$ from Lemma \ref{lem:2continuity-w-broad-volume}, 
\begin{equation}\label{eq:2continuity-ratio-vol-right-strong}
\liminf_{r\to\infty} \inf_{0\leq t \leq t(r,\alpha)} \frac{w_t\left(r + \frac{y}{\sqrt{w_t(r)}}\right)}{w_t(r)} \geq
\begin{cases}
1 - y D(\alpha)  &\text{ if } y\geq 0, \\
1 &\text{ if } y\leq 0 ,
\end{cases}
\end{equation}
and for every $y > - \frac{2}{\e D(\alpha)}$, 
\begin{equation}\label{eq:2continuity-ratio-vol-left-strong}
\limsup_{r\to\infty} \sup_{0\leq t \leq t(r,\alpha)} \frac{w_t\left(r + \frac{y}{\sqrt{w_t(r)}}\right)}{w_t(r)} \leq 
\begin{cases}
f_\alpha^{-1}(-y)  &\text{ if } y\leq 0 , \\
1 &\text{ if } y\geq 0 ,
\end{cases}
\end{equation}
where $f_\alpha^{-1}$ is the inverse of the function $f_\alpha : [1,3] \to [0, \frac{2}{D(\alpha) 3\sqrt{3}} ] , x \mapsto \frac{1-x^{-1}}{D(\alpha)\sqrt{x}} $. 
\end{lemma}

A consequence of \eqref{eq:2continuity-ratio-vol-right-strong} is that for every fixed $z\in\R$, 
\begin{equation}\label{eq:2continuity-ratio-vol-constant-offset}
\frac{w_{t(r,\alpha)}(r+z)}{w_{t(r,\alpha)}(r)} \ulim r \infty 1 .
\end{equation}
It follows directly when $z\geq 0$. For $z<0$, taking $\beta>\alpha$ we find that $t(r+z,\beta) \geq t(r,\alpha)$ for every large enough $r$, and using \eqref{eq:2continuity-ratio-vol-right-strong} with $r+z$ instead of $r$ and $\beta$ instead of $\alpha$ gives the desired bound.

\begin{proof}
Fix $\alpha>0$, 
$\varepsilon>0$ and $r_\varepsilon$ such that for every $r\geq r_\varepsilon$, by Lemma \ref{lem:2uniform-continuity-u}, for every $\beta,\gamma \in [0,\alpha]$ with $|\beta-\gamma|\leq \varepsilon$ we have $\left| \frac{w_{t(r,\beta)}(r)}{w_{t(r,\gamma)}(r)} - 1 \right| \leq \varepsilon$. Writing $s = t(r,\gamma)$ and $t=t(r,\beta)$, 
\begin{align*}
\frac{w_t\left(r+\frac{y}{\sqrt{w_t(r)}}\right)}{w_t(r)} 
&\geq (1-\varepsilon)^2 \frac{w_s\left( r + \frac{y}{\sqrt{w_t(r)}}\right)}{w_s(r)} 
\geq (1-\varepsilon)^2 \frac{w_s\left( r + \frac{\frac{y}{1-\varepsilon}}{\sqrt{w_s(r)}}\right)}{w_s(r)}  .
\end{align*}
Taking $r_\varepsilon$ larger if necessary, by Lemma \ref{lem:2continuity-w-broad-volume} we can assume that for every $\gamma \in \varepsilon\Z \cap [0,\alpha]$
\[ \frac{w_{t(r,\gamma)}\left( r + \frac{\frac{y}{1-\varepsilon}}{\sqrt{w_{t(r,\gamma)}(r)}}\right)}{w_{t(r,\gamma)}(r)} \geq (1-\varepsilon) \left(1 - D(\gamma) \frac{y}{1-\varepsilon} \right) \geq 1-\varepsilon - D(\alpha) y .
\]
Here we used that $c$ is a nonincreasing function. It follows that for every $r\geq r_\varepsilon$
\[ \inf_{0\leq t \leq t(r,\alpha)} \frac{w_t\left(r + \frac{y}{\sqrt{w_t(r)}}\right)}{w_t(r)}  \geq  (1-\varepsilon)^2 (1 - \varepsilon - D(\alpha) y) .\]
Taking the liminf as $r\to\infty$, then $\varepsilon\to 0$ gives \eqref{eq:2continuity-ratio-vol-right-strong}.

We first prove \eqref{eq:2continuity-ratio-vol-left-strong} without the supremum; \eqref{eq:2continuity-ratio-vol-left-strong} will follow by the same reasoning as that we used above to deduce \eqref{eq:2continuity-ratio-vol-right-strong} from \eqref{eq:2continuity-ratio-vol-right}. 
Consider as in the proof of Lemma \ref{lem:continuity-w-broad} the quantity $r_{x} = \inf\{ r \geq 0, w_{t(r,\alpha)}(r) \leq x\}$. Let $B\geq 1$ and $t=t(r,\alpha)$; 
by \eqref{eq:2continuity-ratio-vol-right-strong}, 
\[ \liminf_{r\to\infty} \inf_{A \in[1,B]} \frac{w_t\left(r_{Aw_t(r)} + \frac{y}{\sqrt{w_t(r_{Aw_t(r)})}}\right)}{w_t(r_{Aw(r)})} \geq 1 - y D(\alpha) . \]
For every $\varepsilon>0$ there exists $r_\varepsilon$ such that for every $A\in [1,B]$ and every $r\geq r_\varepsilon$, the ratio inside the $\inf$ is larger than $(1-\varepsilon) (1-y D(\alpha))$. Since $w_t(r_{Aw_t(r)}-1)>Aw_t(r)$, using \eqref{eq:2continuity-ratio-vol-constant-offset} with $z=-1$ and taking $r_\varepsilon$ larger if necessary, we have $w_t(r_{Aw_t(r)}) \geq (1-\varepsilon) Aw_t(r)$ for every $r\geq r_\varepsilon$, thus 
\[ \frac{w_t\left(r_{Aw_t(r)} + \frac{y}{\sqrt{w_t(r_{Aw_t(r)})}}\right)}{w_t(r)} \geq A(1-\varepsilon) \frac{w_t\left(r_{Aw_t(r)} + \frac{y}{\sqrt{w_t(r_{Aw_t(r)})}}\right)}{w_t(r_{Aw(r)})} \geq A (1-y D(\alpha)) (1-\varepsilon)^2 , \]
so that for every $y\leq D(\alpha)^{-1}(1 - (A(1-\varepsilon)^2)^{-1})$, the right-hand side is larger than $1$, hence since $w_t(r_{Aw_t(r)}) \leq A w_t(r)$
\[ r > r_{Aw_t(r)} + \frac{y}{\sqrt{w_t(r_{Aw_t(r)})}} \geq r_{Aw_t(r)} + \frac{y/\sqrt{A}}{\sqrt{w_t(r)}} \implies r_{Aw_r(t)} \leq r - \frac{y/\sqrt{A}}{\sqrt{w_t(r)}} \ , \] 
whence 
\[ \frac{w_t\left( r - \frac{y/\sqrt{A}}{\sqrt{w_t(r)}} \right)}{w_t(r)} \leq \frac{w_t(r_{Aw_t(r)})}{w_t(r)} \leq A . \]
We conclude with the same reasoning as in the proof of Lemma \ref{lem:continuity-w-broad}. 

Finally, we can use the same reasoning as in the proof of Corollary \ref{cor:2uniform-limit-r2wt} to prove the final statement of the Lemma. 
\end{proof}

\subsection{The Feynman--Kac representation}
\label{sec:Feynman-Kac-volume}

Recall $E(r,R,\delta)$ from \eqref{eq:event-E-rRd}.  
Define the distribution of $(Z, I^{(r,R,\delta)})$ under $\E^{(t)}$ to be such that for every positive and measurable $f$, 
\begin{equation}\label{eq:2def-ZI-vol}
\E^{(t)}[\chi(\R)] \cdot \E^{(t)}\left[ f(Z, I^{(r,R,\delta)}) \right] = \E^{(t)}\left[ \int f(x,\ind{E(r,R,\delta)}) \d\chi(x) \right] .
\end{equation}
Note that $\E^{(t)}[f(Z)] = \E[f(Z_t)]$ for every $t$. Define the Markov kernel $((z,i), \Gamma) \mapsto P^{(r,R,\delta)}_{z,i}(\Gamma)$ (that implicitely depend on $t$), where $\Gamma$ is a measurable subset of the set of point measures, $i\in\{0,1\}$ and $z\in\R$, such that for every positive and measurable $g$, if $(X_i)_{1\leq i \leq \chi(\R)}$ is a measurable numbering of the atoms of $\chi$, 
\begin{equation}\label{eq:2def-Palm-vol}
\E^{(t)}[\chi(\R)] \cdot  \E^{(t)}\left[ \int g\left(Z, I^{(r,R,\delta)}, \phi \right) \d P^{(r,R,\delta)}_{Z, I^{(r,R,\delta)}}(\phi) \right] = \E^{(t)}\left[ \sum_{i=1}^{\chi(\R)} g\left(X_i , \ind{E(r,R,\delta)}, \chi - \delta_{X_i} \right) \right] .
\end{equation} 
This Markov kernel defines a ``mean measure function'' $M^{(r,R,\delta)}_{z}$ (that implicitely depend on $t$) such that for every measurable subset $A$ of $\R$, 
\begin{equation}\label{eq:2def-Palm-mean-vol}
M^{(r,R,\delta)}_{z}(A) = \int \phi(A) \d P^{(r,R,\delta)}_{z,1}(\phi) . 
\end{equation}

Proposition \ref{prop:convolution-w} becomes 

\begin{proposition}\label{prop:2convolution-w-vol}
Under Assumptions \ref{assum} 1. to 4., there exists $\delta^0, t_0, r_0>0$ and $C>0$ such that for every $\delta \in (0,\delta^0)$, $t\in [0,t_0)$, $r_\bullet \geq r_0$, $r\geq 2r_\bullet$ and every $r_\bullet \leq R \leq  r-r_\bullet$, 
\begin{multline}\label{eq:2convolution-w-vol}
w_t(r) = \E^{(t)}\Bigg[ I^{(r,R,\delta)} w_t(r-Z) \exp\Big( \ln\E^{(t)}[\chi(\R)] - w_t(r-Z) - \frac{1}{2} \int w_t(r-y) \d M^{(r,R,\delta)}_{Z} (y) \\
+ \E^{(t)}\left[ I^{(r,R,\delta)} w_t(r-Z) \right] \Big) \Bigg] + \mathrm{Remainder}_t(r,R,\delta) ,
\end{multline}
where 
\begin{equation*} 
C_\mathrm{rem} |\mathrm{Remainder}_t(r,R,\delta)| \leq \P(\Lambda>R) + \P\left(\chi(\R)>\frac{\delta}{w_t(r-R)}\right) + (\sigma^2 + 1) w_t(r-R)^2 \left( w_t(r-R) + \delta \right) .
\end{equation*}
\end{proposition}


The proof of this Proposition is identical to that of Proposition \ref{prop:convolution-w} with $\E$ replaced by $\E^{(t)}$, $h$ replaced by $\tilde h_t$ and $w$ replaced by $w_t$, since most of the estimates therein are deterministic and the rest have been controlled in Section \ref{sec:2Laplace-volume-without-displacement}. Note the additional $\ln\E^{(t)}[\chi(\R)]$ in the exponential which comes from \eqref{eq:2def-ZI-vol}, \eqref{eq:2def-Palm-vol} and \eqref{eq:2def-Palm-mean-vol}. A useful observation to bound the remainder with expressions in $\P$ and not in $\P^{(t)}$ is \eqref{eq:bound-Et-by-2E}. 
More precisely, we use the following substitutions.

\begin{center}
\begin{tabular}{|r|l|}
\hline
Proposition \ref{prop:convolution-w} & Proposition \ref{prop:2convolution-w-vol} \\
\hline
\eqref{eq:event-E-rRd} & \eqref{eq:event-E-rRd} \\ \hline
\eqref{eq:def-ZI} & \eqref{eq:2def-ZI-vol} \\ \hline
\eqref{eq:def-Palm} & \eqref{eq:2def-Palm-vol} \\ \hline
\eqref{eq:def-Palm-mean} & \eqref{eq:2def-Palm-mean-vol} \\ \hline
Proposition \ref{prop:convolution-equation} resp. \eqref{eq:z1} & \eqref{eq:2convolution-equation-vol} \\ \hline
$w(x) \ulim x \infty 0$ & $\sup_{t \in [0,t_0]} w_{t}(r) \ulim r \infty 0$ for some $t_0>0$, from \eqref{eq:2upper-bound-wt} \\
& (It is important to keep $r_0$ independent of $t$.) \\ \hline
$\E[\chi(\R)^2] = \sigma^2+1$ & $\E^{(t)}[\chi(\R)^2] \leq 2(\sigma^2+1)$ for every $t\in [0,t_0]$, see \eqref{eq:2variance-chi-vol} \\
& (This will change the constants in Bound$_t$ and in Remainder$_t$.)\\
\hline
\end{tabular}
\end{center}

Fix $r_\bullet \geq r_0$ and $r\geq 2r_\bullet$. 
Define $S_n$, $R_n$, $\delta_n$ and $\FF_n$ as in Section \ref{sec:the-martingale}, and 
\[ T = \inf\{n\geq 0: S_n < 2r_\bullet \text{ or } R_n\notin [r_\bullet, S_n-r_\bullet] \text{ or } \delta_n \notin (0,\delta^0)\}. \]
Proposition \ref{prop:almost-martingale} also holds, with an extra $n\ln\E^{(t)}[\chi(\R)]$ in the exponential in $\kW^{(r)}_n$
\begin{multline}\label{eq:2def-Z-vol}
\kW^{(r)}_n = \left(\prod_{k=0}^{n-1} I_{k+1}\right) \exp\Bigg( n \ln\E^{(t)}[\chi(\R)] - \sum_{k=0}^{n-1} \left(w_t(S_{k+1}) - \E^{(t)}\left[I^{(S_{k}, R_{k}, \delta_{k})} w_t(S_{k}-Z) \, | \, \FF_{k} \right] \right) \\
-\frac{1}{2} \sum_{k=0}^{n-1} \int w_t(S_{k}-y) \d M_{S_k-S_{k+1}}^{(S_k, R_k, \delta_k)}(y) \Bigg) 
\end{multline}
and replacing $\E$ by $\E^{(t)}$, $w$ by $w_t$ and $\mathrm{Remainder}$ by $\mathrm{Remainder}_t$. 
Next comes the analogue of Lemme \ref{lem:discount-UI}:

\begin{lemma}\label{lem:2discount-UI-vol}
For every $\alpha>0$ and $0<C_-<1<C_+<\infty$, there exists $a_\alpha$ and $r_\alpha$ such that for every $a\in (0,a_\alpha)$, every $r_\bullet \geq r_\alpha$, $r\geq 2 r_\bullet$ and every $t\in [0,t(r,\alpha)]$, letting $R_n = \RR := a / \sqrt{w_t(C_- r)}$ for every $n$ and $T_a = T \wedge \inf \{ n\geq 0 : S_n \notin [C_-r, C_+r] \}$, 
the stopped process $(\kW^{(r)}_{n\wedge T_a})_{n\geq 0}$ is a uniformly integrable supermartingale under $\E^{(t)}$. 
\end{lemma}


\begin{proof}
The majority of the proof is a straightforward adaptation of that of Lemma \ref{lem:discount-UI} with the modifications described before Lemma \ref{lem:2discount-UI-vol}. 
Instead of Lemma \ref{lem:lower-bound-w} we use Corollary \ref{cor:2uniform-limit-r2wt} with $(2C_+)^2\alpha$ instead of $\alpha$, $2C_+ r$ instead of $r$, and $b=C_- / C_+ / 4$. Recalling $t(r,\alpha) = t(2C_+ r, (2C_+)^2 \alpha)$, we find $r_\alpha\geq r_0$ such that that for every $r\geq 2r_\alpha$, 
\[ \frac{1}{2}c((2C_+)^2\alpha) \leq \inf_{t\in [0,t(r,\alpha)]} \inf_{C_-r/2 \leq y \leq 2C_+r} y^2 w_{t}(y) \leq \sup_{t\in [0,t(r,\alpha)]} \sup_{C_-r/2 \leq y \leq 2C_+r} y^2 w_{t)}(y) \leq \frac{12 \eta^2}{\sigma^2} . \]
Take $a_\alpha = \sqrt{c((2C_+)^2\alpha)/8}$, then for every $a \in (0,a_\alpha)$ and for every $C_- r \leq x \leq C_+ r$ we have $\RR \leq \frac{x}{2}$ hence $C_- r/2 \leq x/2 \leq x-\RR \leq C_+ r$, thus 
\begin{equation}\label{eq:wt2-over-wt-goes-to-zero}
\sup_{t\in [0,t(r,\alpha)]} \sup_{r\geq 2r_\alpha} \sup_{a\in (0,a_\alpha)} \sup_{x\in [C_-r, C_+r]} \frac{w_{t}(x-\RR)^2}{w_{t}(x-K)} 
\ulim {r_\alpha} \infty 0 .
\end{equation}
Call $(\delta, r_\bullet, r, x)$ ``admissible’’ if $\delta\in (0,\delta^0)$, $r_\bullet \geq r_\alpha$, $r\geq 2r_\bullet$ and $x \in [C_-r, C_+r] \cap [2r_\bullet, \infty)$. 
We point out that 
\begin{equation}\label{eq:Ta-is-admissible}
T_a = \inf \{ n\geq 0 \ : \ (\delta_n, r_\bullet, r, S_n) \text{ is not admissible} \} .
\end{equation}
A key step of the proof is to show that for every admissible $(\delta, r_\bullet, r, x)$, letting
\[ A = - I^{(x,\RR,\delta)}w_t(x-Z) + \E^{(t)}\left[I^{(x,\RR,\delta)} w_t(x-Z)\right] \quad , \quad B = - \frac{I^{(x,\RR,\delta)}} 2 \int w_t(x-y) \d M_Z^{(x,\RR,\delta)}(y) \]
then $\E^{(t)}[I^{(x, \RR, \delta)} \e^{2(A+B)}] \leq 1$; let us prove it. 
Obviously $|A|\leq w_t(x-\RR)$ from \eqref{eq:2def-ZI-vol}. 
By the analogue of \eqref{eq:Palm-controle-M} (using \eqref{eq:2def-Palm-vol} and \eqref{eq:2def-Palm-mean-vol}) applied to the new $M_z^{(r,\RR,\delta)}$ we have $M_z^{(x,\RR,\delta)}(\R) \leq \frac{\delta}{w_t(x-\RR)}-1$ on $E(r,\RR,\delta)$ for every $z, x, \RR, \delta$, hence $-\frac\delta 2 \leq B\leq 0$. Then
\begin{equation*}
\E^{(t)}[B] \geq -\frac{w_t(x-\RR)}{2} \E^{(t)}\left[\ind{E(x,\RR,\delta)} (\chi(\R)-1)\chi(\R)) \right] \geq -\sigma^2 w_t(x-\RR)
\end{equation*}
as soon as $t\in [0,t_0)$ by \eqref{eq:2variance-chi-vol}, where we take $t_0>0$ smaller than in Proposition \ref{prop:2convolution-w-vol} if necessary. On the other hand, by \eqref{eq:bound-Et-by-2E} and dominated convergence
\begin{align*}
\sup_{t\in [0,t_0)} \E^{(t)}\left[(1-\ind{E(x,\RR,\delta)}) \chi(\R) (\chi(\R)-1) \right] \leq 2\E\left[(1-\ind{E(x,\RR,\delta)}) \chi(\R) (\chi(\R)-1) \right] \ulim {r_0} \infty 0 
\end{align*}
and 
\begin{align*}
\sup_{t\in [0,t_0)} \E^{(t)}\left[ \chi((-\infty, K)) (\chi(\R)-1) \right] \leq 2 \E\left[ \chi((-\infty, K)) (\chi(\R)-1) \right] \ulim K {-\infty} 0 ,
\end{align*}
which together with \eqref{eq:2variance-chi-vol} means that we can find $K\in\R$ that only depends on $t_0$ such that, taking $r_\alpha$ larger if necessary,
\begin{align*}
\inf_{t\in[0,t_0)} \inf_{r\geq 2r_\alpha} \inf_{\delta \in (0,\delta^0)} \E^{(t)}\left[ \ind{E(x,\RR,\delta)} \chi([K,\infty)) (\chi(\R)-1) \right] \geq \frac{\sigma^2}{2} ,
\end{align*}
hence $\E^{(t)}[B] \leq - c w_t(x-K)$ for every $x\in [C_-r, C_+r]$, where $c=\sigma^2/4$. We deduce 
\[ \E^{(t)}\left[ I^{(x,\RR,\delta)} \e^{2(A+B)} \right] \leq 1 - c w_t(x-K) + 4(\sigma^2+1) w_t(x-\RR)^2 . \]
By \eqref{eq:wt2-over-wt-goes-to-zero}, taking $r_\alpha$ larger if necessary, we can thus ensure that for every admissible $(\delta, r_\bullet, r, x)$ we have $t(r,\alpha)<t_0$ and $\E^{(t)}\left[ I^{(x,\RR,\delta)} \e^{2(A+B)} \right] \leq 1$. The uniform integrability and the bound on the conditional expectation follows as in Lemma \ref{lem:discount-UI}. 
\end{proof}

We finally state the Feynman--Kac representation. Its proof is a straightforward adaptation of that of Corollary \ref{cor:Feynman-Kac}, since we only need to control the remainder of Proposition \ref{prop:2convolution-w-vol}, which is (up to a multiplicative constant and up to replacing $w$ by $w_t$) the same as that of Proposition \ref{prop:convolution-w}; the only change is the use of Lemma \ref{lem:2discount-UI-vol} instead of Lemma \ref{lem:discount-UI}, hence the more restrictive conditions in the statement of the Corollary.

\begin{corollary}\label{cor:2Feynman-Kac-vol}
For every $\alpha>0$, $0<C_-<1<C_+<\infty$ and $a\in (0,a_\alpha)$, recalling $a_\alpha$, $r_\alpha$ and $T_a$ from Lemma \ref{lem:2discount-UI-vol}, 
there exist two functions $\underline{\delta}(y)\ulim y \infty 0$ with values in $(0,\delta^0)$ and $g(y) \ulim y \infty 0$ such that 
for every $A>0$, $r\geq 2r_\alpha$, $x \in [C_- r, C_+ r]$ and every $t\in[0,t(r,\alpha)] $, letting $\delta_n = \underline{\delta}(r)$ for every $n$, $R_n = \RR := a/\sqrt{w_t(C_- r)}$ for every $n$, and $S_0 = r$ a.s., 
writing $\tau_x = \inf\{ n\geq 0 : S_n \leq x \}$ and $T' = \inf(\tau_x, A/w(x), T_a)$, 
\begin{equation}\label{eq:2FK-vol}
\left| w_t(r) - \E^{(t)}\left[ W^{(r)}_{T'}\right] \right| \leq A w_t(x) g(x) . 
\end{equation}
\end{corollary}


\begin{proof}
By Corollary \ref{cor:2uniform-limit-r2wt}, up to taking $r_\alpha$ larger, using that $t(r,\alpha) = t(C_+ r, C_+^2 \alpha)$, we have for every $z\in [C_- r/2, C_+ r]$ and every $t\in [0,t(r,\alpha)]$
\[ \frac{c((C_-/2)^2 \alpha)}{2 z^2} \leq w_t(z) \leq \frac{12 \eta^2}{\sigma^2 z^2} . \] 
Let $m = C_- r/2$. 
Recall from \eqref{eq:Ta-is-admissible} that on $\{T' > n\}$, $m \leq S_n-R_n \leq C_+r$. Together with Lemma \ref{lem:refined-Markov}, we get a function $g_1(z)\ulim z 0 0$ (that only depends on the law of $\Lambda$ and $\chi$ under $\P$) such that on $\{T'>n\}$, for $c_1 = \frac{48 \eta^2}{\sigma^2 C_-^2}$, 
\begin{align*}
&C_\mathrm{rem} |\mathrm{Remainder}_t(S_n, R_n, \delta_n)| \\
\leq& \ \P(\Lambda>R_n) + \P\left(\chi(\R)>\frac{\delta_n}{w_t(S_n-R_n)}\right) + (\sigma^2+1) w_t(S_n-R_n)^2(w(S_n-R_n)+\delta_n) \\
\leq& \ \frac{w_t(m)^2}{a^2} g_1\left(\frac{w_t(m)}{a^2}\right) + \frac{w_t(m)^2}{\delta_n^2} g_1\left(\frac{w_t(m)}{\delta_n}\right) + (\sigma^2+1) w_t(m)^2 (w_t(m)+\delta_n) \\
\leq& \ \frac{c_1^2}{r^4} g_1\left(\frac{c_1}{a^2 r^2}\right) + \frac{c_1^2}{r^4 \underline{\delta}(r)^2} g_1\left(\frac{c_1}{\underline{\delta}(r) r^2}\right) + \frac{c_1 (\sigma^2+1)}{r^4}\left( \frac{c_1}{r^2} + \underline{\delta}(r)\right)
\end{align*}
and we see that we can find $\underline{\delta}(r)$ that goes to zero sufficiently slowly as $r\to\infty$ that this expression is $o(r^{-4})$ uniformly in $t$. 
Use the analogue of Proposition \ref{prop:almost-martingale} and the stopping theorem as in the proof of Corollary \ref{cor:Feynman-Kac}: 
\begin{align*}
w_t(r) = \E^{(t)}\left[ W_{T'}^{(r)} \right] - \E^{(t)}\left[ \sum_{k=0}^{T'-1} \kW^{(r)}_k \mathrm{Remainder}_t(S_k, R_k, \delta_k) \right] .
\end{align*}
By Lemma \ref{lem:2discount-UI-vol}, on $\{T_a>n\} \supset \{T'>n\}$ we have $\E^{(t)}\left[\kW^{(r)}_{n+1} \, | \, \FF_n \right] \leq 1$. With this, together with our upper bound on the remainder and our lower bound on $w_t(x)$, and following the same ideas as for Corollary \ref{cor:Feynman-Kac}, we find $g(z) \to 0$ as $z\to \infty$ that makes \eqref{eq:2FK-vol} hold. 
\end{proof}

\subsection{The functional equation of the limit}

We have seen in Lemma \ref{lem:positivity-c-alpha} that for every $\alpha\geq 0$ 
\[ 0 < c(\alpha) = \liminf_{r\to\infty} r^2 w_{t(r,\alpha)}(r) \leq \limsup_{r\to\infty} r^2 w_{t(r,\alpha)}(r) \leq \frac{6\eta^2}{\sigma^2} . \]
Along every sequence $(r_k)_{k\geq 1}$, letting $t_k = t(r_k,\alpha)$, we can find some $\LL(\alpha)$ such that, up to extracting a subsequence, we have $r^2 w_t(r) \to \LL(\alpha)$. Clearly $\LL(0) = \frac{6\eta^2}{\sigma^2}$. 
Like in Section \ref{sec:ratio-limit}, by diagonal extraction we can find $\phi_\alpha$ and a subsequence $(r_{k_n})_n$ such that in addition, for every rational $y\geq 0$
\begin{equation}\label{eq:2cv-phialpha}
\frac{w_{t_{k_n}}\left(r_{k_n} + \frac{y}{\sqrt{w_{t_{k_n}}(r_{k_n})}}\right)}{w_{t_{k_n}}(r_{k_n})} \ulim k \infty \phi_\alpha(y) .
\end{equation}
In what follows, for conciseness, and until we specify otherwise, whenever we take the convergence $r\to\infty$ we do so along this subsequence. 
By the same reasoning as in Proposition \ref{prop:continuity-phi} with Lemma \ref{lem:continuity-w-broad} replaced by Lemma \ref{lem:2continuity-w-broad-volume}, $\phi_\alpha$ is continuous, non-increasing, positive, $\phi_\alpha(0)=1$, and the convergence in \eqref{eq:2cv-phialpha} holds uniformly (in $y$) over every compact of $[0,\infty)$. We now establish the analogue of Proposition \ref{prop:func-equation-phi}:

\begin{proposition}\label{prop:2func-equation-phi-vol}
Any limit $\phi_\alpha$ in \eqref{eq:2cv-phialpha} satisfies
\begin{equation}
\phi_\alpha(y) = \E^{y/\eta}\left[\exp\left( -\frac{\sigma^2}{2} \int_0^{\tau_0} \phi_\alpha(\eta B_s) \d s - \frac{\alpha \LL(0)}{\LL(\alpha)} \tau_0 \right)\right] . 
\end{equation}
\end{proposition}

\begin{proof}
We follow closely the proof of Proposition \ref{prop:func-equation-phi}.  
Fix $\nu>0$, $A>0$ and $y>0$. Let $r = x + \frac{y}{\sqrt{w_t(x)}}$ and $C>1$ large enough that $r \leq Cx$ for every $x$ large enough. Fix $C_- = 1/C$ and $C_+ = C$. In what follows we always have $t = t(x,\alpha)$. Let $\alpha' = C^2 \alpha$, so that $t \leq t(r,\alpha')$. 
By Lemma \ref{lem:2continuity-w-broad-volume-strong}, for every $a>0$ small enough and every large enough $x$
\[ \sup_{u\geq x} \frac{w_t\left(u - \frac{a}{\sqrt{w_t(u)}}\right) }{w_s(u)} \leq \sup_{u\geq x} \sup_{s\in [0,t(r,\alpha')]} \frac{w_s\left(u - \frac{a}{\sqrt{w_s(u)}}\right) }{w_s(u)} \leq 1+\nu . \]
Apply Corollary \ref{cor:2Feynman-Kac-vol} with $\alpha'$, taking $a$ smaller if necessary so that $a < a_{\alpha'}$, letting $R_n = \RR := a / \sqrt{w_t(C_- r)}$, $\delta_n = \delta := \underline{\delta}(r)$, and $S_0 = r$ a.s.:
\[ \left| w_t(r) - \E^{(t)}\left[ W^{(r)}_{T'}\right] \right| \leq A w_t(x) g(x) . \]
We now aim to control 
\[ \E^{(t)}\left[ W^{(r)}_{T'}\right] = \E^{(t)}\left[ w_t(S_{T'}) \kW^{(r)}_{T'} \right] . \] 
Recall the expression from $\kW^{(r)}_n$ in \eqref{eq:2def-Z-vol}:
\begin{multline*}
\kW^{(r)}_n = \left(\prod_{k=0}^{n-1} I_{k+1} \right) \exp\Bigg( - \underbrace{\sum_{k=0}^{n-1} \left( I_{k+1} w_t(S_{k+1}) - \E^{(t)}\left[I^{(S_k, R_k, \delta)} w_t(S_k - Z) \, | \, \FF_k\right] \right)}_{M_n} \\
+ n \ln\E^{(t)}\left[\chi(\R)\right]  - \frac{1}{2} \underbrace{ \sum_{k=0}^{n-1} I_{k+1} \int w_t(S_{k}-y) \d M_{S_k - S_{k+1}}^{(S_k, R_k, \delta)}(y)}_{Y_n} \Bigg) . 
\end{multline*}
We control $M_n$ exactly as in Proposition \ref{prop:func-equation-phi} since (taking $a>0$ and $\nu>0$ smaller if necessary) $w_t(u) \leq (1+\nu) w_t(x) \leq 2 w(x)$ for every $u\geq x-\RR$ and $x$ large enough by \eqref{eq:2upper-bound-wt} and Theorem \ref{th:main-tail}, giving $c_x\to 0$ as $x\to\infty$ such that 
\begin{equation}\label{eq:2controle-Z-1-vol}
\P\left( \sup_{0\leq k \leq A/w_t(x)} |M_{k\wedge T'}| \geq c_x \right) \leq c_x .
\end{equation}
On the other hand, writing 
\[ X_k := I_{k+1} \int w_t(S_{k}-z) \d M_{S_{k}-S_{k+1}}^{(S_k, R_k, \delta)}(z) \quad , \quad \tilde Y_n = \sum_{k=0}^{n-1} \E^{(t)}\left[ X_k \ | \ \FF_k \right] \]
then $\tilde Y_n$ is a predictable process such that $\tilde M_n := (Y_{n\wedge T'} - \tilde Y_{n\wedge T'})_{n\geq 0}$ is a martingale. 
Let us compute $\E^{(t)}[X_k \, | \, \FF_{k}]$. 
For every $u\geq x$, using \eqref{eq:2def-Palm-mean-vol} instead of \eqref{eq:def-Palm-mean} 
\begin{align}
\nonumber & \frac{1}{w_t(u)}\E^{(t)}\left[ I^{(u,\RR,\delta)} \int w_t(u-z) \d M_{Z}^{(u, \RR, \delta)}(z) \right] \\
\nonumber =& \ \E^{(t)}[\chi(\R)]^{-1} \cdot \E^{(t)}\left[ \ind{E(u,\RR,\delta)} (\chi(\R)-1)\int \frac{w_t(u-z)}{w_t(u)} \d\chi(z) \right] \\
=& \ \E^{(t)}[\chi(\R)]^{-1} \cdot \E^{(t)}\left[ \ind{E(u,\RR,\delta)} (\chi(\R)-1)\chi(\R) \right] \label{eq:2term1} \\
&+ \E^{(t)}[\chi(\R)]^{-1} \cdot \E^{(t)}\left[ \ind{E(u,\RR,\delta)} (\chi(\R)-1)\int \left(\frac{w_t(u-z)}{w_t(u)}-1\right) \d\chi(z) \right] . \label{eq:2term2}
\end{align}
By \eqref{eq:bound-Et-by-2E}
\[ \sup_{u\geq x} \E^{(t)}\left[ (1-\ind{E(u,\RR,\delta)}) (\chi(\R)-1)\chi(\R) \right] \leq 2 \sup_{u\geq x} \E\left[ (1-\ind{E(u,\RR,\delta)}) (\chi(\R)-1)\chi(\R) \right] \ulim x \infty 0 \]
which together with \eqref{eq:2expectation-chi-vol} shows that 
$\sup_{u\geq x} |\eqref{eq:2term1}-\sigma^2| \ulim x \infty 0$.
On the other hand, using again \eqref{eq:bound-Et-by-2E} to bound \eqref{eq:2term2} by twice the same expression with $\E$ instead of $\E^{(t)}$, then recalling that $w_t(u-\RR) \leq (1+\nu) w_t(u)$ from our choice of $a$, using \eqref{eq:2continuity-ratio-vol-constant-offset} to get that $w_t(u-z) / w_t(u) \to 1$ as $x\to\infty$ for every fixed $z$ and every $C_- r \leq u \leq C_+ r$, we get by dominated convergence that \eqref{eq:2term2}$\to 0$ as $x\to\infty$. 
Hence 
\begin{equation}
d_x := \sup_{u\geq x} \left|\frac{1}{\sigma^2 w_t(u)} \E^{(t)}\left[ I^{(u,\RR,\delta)} \int w_t(u-z) \d M_{Z}^{(u, \RR, \delta)}(z) \right] - 1\right| \ulim x \infty 0 . 
\end{equation}
The same argument as in Proposition \ref{prop:func-equation-phi} 
gives $c'_x\to 0$ as $x\to\infty$ such that 
\begin{equation}\label{eq:2controle-Z-2-vol}
\P^{(t)}\left( \sup_{0\leq k \leq A/w_t(x)} |\tilde M_{k\wedge T'}| \geq c'_x \right) \leq c'_x 
\end{equation}
and as $x\to\infty$, for every $n$
\begin{equation}\label{eq:2controle-Z-3-vol}
\left| \frac{\tilde Y_{n\wedge T'}}{\sigma^2 \sum_{k=0}^{n\wedge T'-1} w_t(S_k)} - 1 \right| \leq d_x \to 0 .
\end{equation}
As in Proposition \ref{prop:func-equation-phi}, 
\begin{align*}
\limsup_{x\to\infty} \left| \frac{w_t(r)}{w_t(x)} - \E^{(t)}\left[ \frac{w_t(S_{T'})}{w_t(x)} \exp\left( n \ln\E^{(t)}[\chi(\R)] - \frac{\sigma^2}{2} \sum_{j=1}^{T'} w_t(S_{j-1}) \right) \right] \right| \leq \nu + \P(\tau_x>A/w(x)) .
\end{align*}
Take $x$ in the sequence used to define $\phi_\alpha$ and $\LL(\alpha)$. 

Recall the coupling from Section \ref{sec:2coupling-RW}, in particular Lemma \ref{lem:coupling-explicit} and the definition of $(U^t_n)_{n\geq 0}$ afterward: $(S_n-r)_{n\geq 0}$ under $\P^{(t)}$ has the same distribution as $(U^t_n)_{n\geq 0}$ under $\P$. By Doob's maximal inequality applied to the martingale $(U^t_n - U^0_n - \E[U^t_n])_{n\geq 0}$, given that $|\E[U^t_n]| = n|\E[Z_t-Z]| \leq n \sqrt{\P(Z_t \neq Z) \E[(Z_t-Z_0)^2]} = n t^{1/4} o_t(1)$ by Lemma \ref{lem:coupling-explicit} and that 
\[ \E[(U_1^t - U_1^0 - \E[U_1^t])^2] = n\mathrm{Var}(Z_t-Z) \leq n \E[(Z_t-Z)^2] = n o_t(1) \]
by the same Lemma, recalling Lemma \ref{lem:positivity-c-alpha} we find that $\E[ (\sup_{0\leq n \leq M/w_t(x)} |U^t_n - U^0_n|)^2] = o(w_t(x)^{-1})$. Since 
\[ \left(\left(\frac{\eta}{w_t(x)}\right)^{-1/2} \left(U^0_{\lfloor s/w_t(x)\rfloor} + \frac{y}{\sqrt{w_t(x)}} \right)\right)_{s\in [0,M]} \ulimd x \infty (\eta B_s)_{s\in [0,M]} \]
in distribution in the space of càdlàg functions equipped with the $L^\infty$ norm over $[0,M]$, where $(B_s)_{s\geq 0}$ is a standard Brownian motion with $B_0 = y/\eta$ a.s., the same convergence holds with $U^0$ replaced by $U^t$, hence with $((\eta / w_t(x))^{-1/2}(S_{\lfloor s/w_t(x) \rfloor}^t-x))_{s\in [0, M]}$ under $\P^{(t)}$. 
%
%
Since in addition $w(x) T' \to \tau_{0,\tilde C}^B$ where 
\[ \tilde C = \lim_{x\to\infty} \frac{C_+ r - x}{\eta}\sqrt{w_t(x)} = \frac{C-1}{\eta} \sqrt{\LL(\alpha)} + \frac{y C}{\eta} \]
and $\tau_{0,\tilde C}^B = \inf\{t>0: B_t \notin [0,\tilde C]\}$, we have 
\[ \sum_{k=0}^{T'-1} w_t(S_k) = w_t(x) \sum_{k=0}^{T'-1} \frac{w_t(S_k)}{w_t(x)} \ulim x \infty \int_0^{\tau_{0,\tilde C}^B \wedge A} \phi_\alpha(\eta B_u)) \d u \]
in distribution.  
In addition, 
\[ - n \ln\E^{(t)}[\chi(\R)] \ulim x \infty (A\wedge \tau_{0,\tilde C}^B) \lim_{x\to\infty} \frac{1}{w_t(x)} \sqrt{2t\sigma^2} = (A\wedge \tau_{0,\tilde C}^B) \frac{1}{\LL(\alpha)} \lim_{x\to\infty} \sqrt{2\sigma^2 t x^4} = (A\wedge \tau_{0,\tilde C}^B) \frac{\alpha \LL(0)}{\LL(\alpha)} , \]
hence, letting $\P^z$ be a probability such that $(B_u)_{u\geq 0}$ is a standard Brownian motion with $B_0 = z$ a.s. under $\P^z$, 
\begin{multline*}
\left| \phi_\alpha(y) - \E^{y/\eta}\left[ \phi_\alpha\left( \eta B_{\tau^B_{0,\tilde C} \wedge A} \right) \exp\left( - \frac{\sigma^2}{2} \int_0^{\tau_{0,\tilde C}^B \wedge A} \phi_\alpha(\eta B_u) \d u - (A\wedge \tau_{0,\tilde C}^B) \frac{\alpha \LL(0)}{\LL(\alpha)} \right) \right] \right| \\
\leq \nu + \P^{y/\eta}(\tau^B_{0,\tilde C}>A)  .
\end{multline*}
This holds for every $\nu, C$ and $A$; by continuity of $\phi_\alpha$, a.s. 
\[ \phi_\alpha\left( \eta B_{\tau^B_{0,\tilde C} \wedge A} \right) \ulim A \infty \phi_\alpha\left( \eta B_{\tau^B_{0,\tilde C}} \right) \ulim C \infty \phi_\alpha(\eta B_{\tau^B_0}) = 1 . \]
We finally get the statement of the Proposition by taking $A\to\infty$, $C\to\infty$ and $\nu\to 0$. 
\end{proof}

\subsection{Proof of Theorem \ref{th:main-volume-smaller}}


We now deduce Theorem \ref{th:main-volume-smaller} from Proposition \ref{prop:2func-equation-phi-vol}. 
A pair $(\alpha, \br = (r_k)_{k\geq 0}$) with $\alpha\geq 0$ and $r_k\to\infty$ as $k\to\infty$ is called said to be ``convergent'' if $r_k^2 w_{t(r_k,\alpha)}(r_k)$ converges; in that case we write $\LL(\br, \alpha)$ for this limit. Up to extracting a subsequence from $\br$ we can define $\phi_\alpha$ from \eqref{eq:2cv-phialpha}. 
By the reasoning of \cite[Corollary 12]{lalley2015maximal} with $V(x) = \phi_\alpha + \frac{2\alpha\LL(0)}{\sigma^2 \LL(\br,\alpha)}$, using Kac's theorem, $\phi_\alpha$ is the unique solution that is bounded on $[0,\infty)$ with $\phi_\alpha(0)=1$ of
\begin{equation}\label{eq:2equadiff-phi-alpha}
\phi_\alpha'' = \frac{\sigma^2}{\eta^2} \phi_\alpha^2 + \frac{12\alpha}{\LL(\br,\alpha) \sigma^2} \phi_\alpha .
\end{equation}
This guarantees the uniqueness of the limit $\phi_\alpha$ in \eqref{eq:2cv-phialpha} for any subsequence of $\br$, from which we deduce that the convergence in \eqref{eq:2cv-phialpha} holds along $\br$. 

Let us study the differential equation \eqref{eq:2equadiff-phi-alpha}. 
Let $\psi_\alpha(y) = \frac{\LL(\br,\alpha) \sigma^4}{12 \alpha \eta^2} \phi_\alpha\left( \sqrt{\frac{\sigma^2\LL(\br,\alpha)}{12\alpha}} y \right) $. Then 
\[ \psi_\alpha''(y) = \frac{\sigma^2 \LL(\br,\alpha)}{12 \alpha} \frac{\LL(\br,\alpha) \sigma^4}{12 \alpha \eta^2} \phi''_\alpha\left(\sqrt{\frac{\sigma^2\LL(\br,\alpha)}{12\alpha}} y \right) = \psi_\alpha^2(y) + \psi_\alpha(y) , \]
and $\psi_\alpha$ is positive and decreases to $0$ towards $+\infty$ with $\psi_\alpha(0) = \frac{\LL(\br,\alpha) \sigma^4}{12 \alpha \eta^2}$. 
The differential equation $\psi'' = \psi^2 + \psi$ admits a unique positive solution that is bounded towards $+\infty$ (in fact $\psi(x)\to 0$ as $x\to\infty$), up to translation in the argument (see Section \ref{sec:appendix-EDO}), meaning that if $\psi$ is the unique such solution with $\psi(x)\to\infty$ as $x\to 0$ and $x_\alpha = \psi^{-1}\left( \frac{\LL(\br,\alpha) \sigma^4}{12 \alpha \eta^2} \right)$ then $\psi_\alpha(x) = \psi(x + x_\alpha)$ hence
\begin{equation}\label{eq:2phi-psi-vol}
\frac{1}{\alpha} \phi_\alpha(x) = \frac{12 \eta^2}{\LL(\br,\alpha) \sigma^4} \psi\left( x_\alpha + x \sqrt{\frac{12\alpha}{\sigma^2\LL(\br,\alpha)}} \right) . 
\end{equation}

For $\beta \geq \alpha$, let
\[ y = \sqrt{\LL(\alpha)}\left(\sqrt{\frac{\beta}{\alpha}} -1\right) \]
and define the sequence $\br^\beta$ such that $r^\beta_k = r_k + y / \sqrt{w_{t(r_k, \alpha)}(r_k)}$. Since $w_{t(r_k,\alpha)}(r^\beta_k) / w_{t(r_k,\alpha)}(r_k) \to \phi_\alpha(y)$ and $r_k^2 w_{t(r_k,\alpha)}(r_k) \to \LL(\br,\alpha)$, 
\begin{align*}
r_k^2 w_{t(r_k,\alpha)}(r^\beta_k) \ulim k \infty \LL(\br,\alpha) \phi_\alpha(y) . 
\end{align*}
Noting that $t(r_k, \alpha) = t(r_k\sqrt{\beta/\alpha}, \beta)$, 
we deduce
\begin{equation*}
r^\beta_k w_{t(r_k\sqrt{\beta/\alpha},\beta)}(r^\beta_k) \ulim k \infty \LL(\br,\alpha)\phi_\alpha(y) \left( 1 + \frac{y}{\sqrt{\LL(\br,\alpha)}} \right)^2 .
\end{equation*}
By our choice of $y$ we have $r^\beta_k \sim r_k \sqrt{\frac{\beta}{\alpha}}$ as $k\to\infty$, which by Lemma \ref{lem:2continuity-w-broad-volume-strong}
implies $w_{t(r_k \sqrt{\beta/\alpha}, \beta)}(r_k\sqrt{\beta/\alpha}) \sim w_{t(r_k\sqrt{\beta/\alpha}, \beta)}(r^\beta_k)$ as $k\to\infty$. 
It follows that $(\beta, \br \sqrt{\beta/\alpha})$ is convergent and 
\begin{align*}
\LL\left(\br\sqrt{\frac{\beta}{\alpha}}, \beta\right) 
= \frac{\beta}{\alpha} \LL(\br,\alpha)\phi_\alpha(y) 
= \beta \frac{12 \eta^2}{\sigma^4} \psi\left( \sqrt{\frac{2}{\sigma^2}} (\sqrt{\beta} - \sqrt{\alpha}) + \psi^{-1}\left( \frac{\LL(\br, \alpha) \sigma^4}{12 \alpha \eta^2} \right) \right) ,
\end{align*} 
i.e. writing $\br' = \br\sqrt{\beta/\alpha}$ both $(\beta,\br')$ and $(\alpha, \br'\sqrt{\alpha/\beta})$ are convergent with 
\begin{equation}\label{eq:LL-when-beta-larger-alpha}
\psi^{-1}\left( \frac{\LL\left(\br', \beta\right) \sigma^4}{12\beta\eta^2} \right) - \sqrt{\frac{2\beta}{\sigma^2}} = \psi^{-1}\left( \frac{\LL\left(\br' \sqrt{\frac{\alpha}{\beta}}, \alpha\right) \sigma^4}{12\alpha\eta^2} \right) - \sqrt{\frac{2\alpha}{\sigma^2}} .
\end{equation}

Consider now $\gamma\leq \alpha$. 
Along any subsequence of $\br \sqrt{\gamma/\alpha}$ such that $r^2 w_{t(r,\gamma)}(r^2)$ converges, by \eqref{eq:LL-when-beta-larger-alpha} (replacing $\alpha$ by $\gamma$ and $\beta$ by $\alpha$) the limit $\ell$ must satisfy 
\begin{equation}\label{eq:LL-when-gamma-smaller-alpha}
\psi^{-1}\left( \frac{\LL\left(\br, \alpha\right) \sigma^4}{12\alpha\eta^2} \right) - \sqrt{\frac{2\alpha}{\sigma^2}} = \psi^{-1}\left( \frac{\ell \sigma^4}{12\gamma\eta^2} \right) - \sqrt{\frac{2\gamma}{\sigma^2}} 
\end{equation}
which uniquely determines it, hence $(\gamma, \br\sqrt{\gamma/\alpha})$ is convergent and  this unique $\ell$ is such that $\ell = \LL(\br\sqrt{\gamma/\alpha}, \gamma)$. By \eqref{eq:2bound-w-liminf-1-alpha} we know that $(1-\gamma)\LL(0) \leq \LL(\br\sqrt{\gamma/\alpha}, \gamma) \leq \LL(0)$, hence $\LL(\br\sqrt{\gamma/\alpha},\gamma)\sigma^4 / (12\gamma \eta^2) \to \infty$ as $\gamma\to 0$. Since $\psi^{-1}(x) \to 0$ as $x\to\infty$ we deduce by taking $\gamma \to 0$ in \eqref{eq:LL-when-gamma-smaller-alpha} that
\begin{equation*}
\LL(\br, \alpha) = \LL(\alpha) := \frac{12 \alpha \eta^2}{\sigma^4} \psi\left( \sqrt{\frac{2\alpha}{\sigma^2}} \right) .
\end{equation*}
Since the limit is uniquely characterized it means that $(\alpha, \br)$ is convergent for every sequence $\br$ with $r_k\to\infty$ as $k\to\infty$, hence 
\[ r^2 w_{t(r,\alpha)}(r) \ulim r \infty \LL(\alpha) \]
in the usual sense.

By \eqref{eq:psi-expansion} $\LL$ is analytic near $0$ with first terms given by 
\[ \frac{\LL(\alpha)}{\LL(0)} = \frac{\sigma^2}{6\eta^2}\LL(\alpha) = \frac{2\alpha}{\sigma^2}\psi\left(\sqrt{\frac{2\alpha}{\sigma^2}}\right) 
= 1 - \frac{\alpha}{\sigma^2} + \frac{3}{5}\left(\frac{\alpha}{\sigma^2}\right)^2 - \frac{2}{7}\left(\frac{\alpha}{\sigma^2}\right)^3 + \frac{3}{25}\left(\frac{\alpha}{\sigma^2}\right)^4 + \dots 
\]
Thus, as $r\to\infty$ with $t=t(r,\alpha)$, 
\[ \E\left[ \e^{-t\sum_{v\in T} D_v} \, | \, \sup_{v\in T} \Lambda_v \leq r \right] = \frac{h_t(r)}{h(r)} = \frac{h_t(\infty) \tilde h_t(r)}{h(r)} \]
is such that
\begin{multline*}
r^2\left( 1 - \E\left[ \e^{-t\sum_{v\in T} D_v} \, | \, \sup_{v\in T} \Lambda_v \leq r \right]\right) \ulim r \infty \LL(\alpha) - \LL(0) + \alpha \frac{6\eta^2}{\sigma^4} \\
= \frac{6\eta^2}{\sigma^2} \left(\frac{3}{5}\left(\frac{\alpha}{\sigma^2}\right)^2 - \frac{2}{7}\left(\frac{\alpha}{\sigma^2}\right)^3 + \frac{3}{25}\left(\frac{\alpha}{\sigma^2}\right)^4 + o\left( \left(\frac{\alpha}{\sigma^2}\right)^4\right) \right) . 
\end{multline*} 
On the other hand, 
\[ \E\left[ \e^{-t\sum_{v\in T} D_v} \, | \, \sup_{v\in T} \Lambda_v > r \right] = \frac{h_t(\infty)(1-\tilde h_t(r))}{1-h(r)} \sim \frac{w_t(r)}{w(r)} \sim \frac{\LL(\alpha)}{\LL(0)} = \frac{2\alpha}{\sigma^2}\psi\left( \sqrt{\frac{12\alpha}{\sigma^2}} \right) . \]
This constitutes Theorem \ref{th:main-volume-smaller}.

\subsection{Conditionning on the depth.}


We follow closely Section \ref{sec:LLT} to prove Theorem \ref{th:main-condition-vol}.  Assume that $\chi$ is supported on $\Z$ and that $M$ has maximum span $1$, and that $\Lambda$ is supported in $\N$. 
In the rest of this section we write $t = t(r,\alpha)$.
Write $g_t(r) = \E\left[ \e^{-t\sum_{v\in T} D_v} \ind{\sup_{v\in T} \Lambda_v = r} \right]$. We obtain, similarly to Proposition \ref{prop:convolution-equation},
\begin{multline}\label{eq:2convolution-LLT-vol}
g_t(r) = \E\Bigg[\e^{-tD_\emptyset} \ind{\Lambda_\emptyset = r} \prod_{u=1}^{\chi_\emptyset(\R)} h_t(r-X_u) \\
+ \ind{\Lambda_\emptyset < r} \sum_{j=1}^{\chi_\emptyset(\R)} \left(\prod_{u=1}^{j-1} h_t(r-1-X_u)\right) g_t(r-X_j) \left(\prod_{v=j+1}^{\chi_\emptyset(\R)} h_t(r-X_v) \right) \Bigg] \ind{r\geq 0} \\
\geq \E^{(t)}\left[ \ind{\Lambda < r} \int g_t(r-x) \exp\left( - \int w_t(r-1-X_u) \d(\chi-\delta_{x})(y)\right) \d\chi(x) \right] \ind{r\geq 0}
\end{multline}
by bounding the first term inside the expectation by $0$, and $h_t(r-x)$ by $h_t(r-1-x)$ from below as well as $\ln h_t(\infty)-\ln h_t(x) \leq w_t(x)$. Using $\e^{-x}\geq 1-x$, recalling $E(r,R,\delta)$ from \eqref{eq:event-E-rRd} as well as \eqref{eq:2def-ZI-vol}, \eqref{eq:2def-Palm-vol} and \eqref{eq:2def-Palm-mean-vol}, 
\begin{align}
g_t(r) 
&\geq \E^{(t)}[\chi(\R)] \cdot \E^{(t)}\left[ I^{(r,R,\delta)} g(r-Z) \left( 1 - \int w_t(r-1-y) \d M_Z^{(r,R,\delta)}(y) \right) \right] .
\end{align}
As before $\int w_t(r-1-y) \d M_Z^{(r,R,\delta)}(y) \leq \delta \frac{w_t(r-R-1)}{w_t(r-R)}$ on $\{I^{(r,R,\delta)}\neq 0\}$. 
By e.g. Corollary \ref{cor:2uniform-limit-r2wt} or Lemma \ref{lem:2continuity-w-broad-volume-strong}, we can find a constant $C$ (that may depends on $\alpha$) such that 
for every $r$ large enough we have $\sup_{x\geq r} w_t(x-1)/w_t(x) \leq C$. Taking $\delta$ small enough we get as before
\[ g(r) \geq \E^{(t)}[\chi(\R)] \cdot \E^{(t)}\left[ I^{(r,R,\delta)} g(r-Z) \exp\left( - (4\ln 2) \int w_t(r-y) \d M_Z^{(r,R,\delta)}(y) \right) \right] . \]

Take $R_n = R = r/8$ and $\delta_n = \delta$ for every $n$. 
By the Markov property, for every $n\geq 0$ and using the same notations as in Proposition \ref{prop:almost-martingale}, 
\begin{align}\label{eq:2g-supermg-vol}
g_t(S_n) \left(\prod_{j=1}^{n} I_j\right) \exp\left( n\ln\E^{(t)}[\chi(\R)] - (4\ln 2) \sum_{k=0}^{n-1} \int w_t(S_k-y)) \d M_{S_k-S_{k+1}}^{(S_k, R_k, \delta_k)}(y) \right)
\end{align}
is a positive $(\FF_n)_{n\geq 0}$-supermartingale. For every $\nu>0$ and every $y$ with $|y-r|\leq \nu r$, define $T_y := \inf\{n \geq 0 : S_n = y\}$. 
Using the coupling of Section \ref{sec:2coupling-RW} (in particular Lemma \ref{lem:coupling-explicit}) and \eqref{eq:proba-visit-RW}, there exists $c<\infty$ and $n_\varepsilon$ 
such that for every $\varepsilon>0$, every $r$ large enough and every $n\geq n_\varepsilon$, 
\[ \P^{(t)}(T_y > n) \leq 2n\sqrt{t\sigma^2} + \varepsilon + c\frac{1+|y-r|}{\sqrt{n}} . \]
Assume henceforth that 
$n = \frac{c_\alpha^2 \nu^2 r^2}{\varepsilon^2}$, with $c_\alpha>0$ small enough that $\P^{(t)}(T_y > n) \leq 2\varepsilon$ uniformly over every $|y-r|\leq \nu r$. 
Clearly $g_t(S_{T_y}) = g_t(y)$. We aim to bound the other factors in \eqref{eq:g-supermg}. Let us work on $0\leq k \leq n$. By \eqref{eq:2upper-bound-wt} we can take $r$ large enough that $x^2 w_t(x) \in [m, M]$ for every $x\geq r/2$, for some $0<m\leq M<\infty$. 
Define the event $\AA := \{ \inf_{0\leq j \leq n} S_j \geq 3r/4 \}$. By the same reasoning as in Section \ref{sec:LLT}, we can find $\delta = \delta(r) \to 0$ as $r\to\infty$ such that if $\delta_k = \delta(r)$ a.s. for every $k$, 
\[ \P^{(t)}\left(\AA \cap \left\{\prod_{k=0}^{n-1} I_k = 0 \right\}\right) \ulim r \infty 0 \quad \text{and} \quad \P^{(t)}(\AA^c) \ulim \nu 0 0 . \]
On the other hand, recalling $d_r$ from the proof of Proposition \ref{prop:2func-equation-phi-vol} and checking that we also have $d_r \to 0$ as $r\to\infty$, by the same argument, we can take $\nu>0$ small enough that $\P^{(t)}(\AA^c)\leq \varepsilon$ and that for every $r$ large enough and every $0\leq n' \leq n$ 
\begin{align*}
&\ind{\AA} \left(\prod_{k=0}^{n'-1} I_{k+1}\right) \exp\left(n\ln\E^{(t)}[\chi(\R)] - (4\ln 2) \sum_{k=0}^{n'-1} \int w_t(S_k-y) \d M_{S_k-S_{k+1}}^{(S_k, R, \delta)}(y) \right) \\
&\geq \ind{\AA}\left(\prod_{k=0}^{n'-1} I_{k+1} \right) (1-\varepsilon) . 
\end{align*}
By the same reasoning, for every $r$ large enough and
for every $y$ with $|y-r|\leq \frac{\nu r}{2}$, 
\[ (1-6\varepsilon) g_t(r) \leq g_t(y) \leq (1+7\varepsilon) g_t(r) . \]
Since $r^2 \sum_{s>r} g_t(s) \sim r^2 w_t(r) \to \LL(\alpha)$ as $r\to\infty$ and $\LL$ is continuous, we deduce
\[ r^3 g_t(r) \ulim r \infty 2 \LL(\alpha) . \]
From there,
\[ \E\left[ \e^{-t\sum_{v\in T} D_v} \, | \, \sup_{v\in T} \Lambda_v = r\right] = \frac{h_t(\infty) g_t(r)}{g(r)} \ulim r \infty \frac{\LL(\alpha)}{\LL(0)} , \]
which concludes the proof of Theorem \ref{th:main-condition-vol}.

\section{Appendix}

\subsection{Improvement on Markov's inequality}

In order to get optimal moment assumptions, we need to refine Markov's inequality:
\[ \P(X\geq t) \leq \frac{\E[X]}{t} \]
when $X\geq 0$ a.s., $\E[X]<\infty$ and $t>0$. A motivation for that comes from the layer-cake formula: 
\[ \int_0^\infty \P(X\geq t) \, \d t = \E[X] <\infty \]
so that Markov's inequality, which gives a non-integrable tail estimate, fails to catch the actual tail behavior. We make extensive use of the following refinement of Markov's inequality. 
\begin{lemma}\label{lem:refined-Markov}
Let $X\geq 0$ be a random variable in $L^p$ with $p\in [1,\infty)$. Then we can find $\psi: [0,\infty) \to [1,\infty)$, increasing, with $\psi(x) \to +\infty$ as $x\to\infty$, such that $\E[X^p \psi(X)]\leq 2 \E[X^p] < \infty$. It follows that $\P(|X|\geq t) \leq \frac{2\E[X^p]}{t^p\psi(t)} = o(t^{-p})$ as $t\to\infty$. 
\end{lemma}

The function $\psi$ obviously depends on the distribution of $X$. 

\begin{proof}
By the layer-cake formula, 
\[ \E[X^p] = \int_0^\infty t^{p-1} \P(Y\geq t) \d t . \]
Write $t_k := \inf \{t : \int_t^\infty t^{p-1} \P(X\geq t) \d t \leq 2^{-2k} \E[X^p] \}$, and let $\psi(t) = 1 + 2^{k-1}$ for every $t_k \leq t < t_{k+1}$. Then
\[ \E[X^p \psi(X)] = \frac{3}{2}\E[X^p] + \sum_{k\geq 1} 2^{k-1} \int_{t_k}^\infty t^{p-1} \P(X\geq t) \d t \leq \left( \frac{3}{2} + \sum_{k\geq 1} 2^{k-1} 2^{-2k}\right) \E[X^p] \leq 2 \E[X^p] . \]
\end{proof}

\subsection{Study of the differential equation}  
\label{sec:appendix-EDO}

We are interested in solutions $y$ that are positive and are bounded towards $+\infty$ of
\begin{equation}\label{eq:EDO-1}
y'' = y^2 + y . 
\end{equation}
By multiplying with $y'$ and integrating, we get 
\begin{equation*}
(y')^2 = \frac{2}{3}y^3 + y^2 + c
\end{equation*}
for some $c$. Assuming $y\geq 0$, we get $y''\geq 0$, hence $y$ is convex; if $y$ is bounded towards $+\infty$ then $y'(x)\to 0$ as $x\to\infty$, which implies $c=0$ and $y(x)\to 0$ as $x\to\infty$. Since $y\geq 0$ we have $y'\leq 0$ hence 
\begin{equation}\label{eq:EDO-2}
y' = - \sqrt{\frac{2}{3}y^3 + y^2} .
\end{equation}
Writing $F$ for a primitive of $f : x \mapsto -(\frac 2 3 x^3+x^2)^{-1/2}$, the equation becomes
\[ (F\circ y)' = 1 . \]
Since $f$ is integrable towards $+\infty$ but not towards $0$, we set the integration constant by deciding that $F(x)\to 0$ as $x\to\infty$, in which case every solution that is defined in a neighborhood of $a\in\R$ 
must satisfy 
\begin{equation}\label{eq:EDO-sol}
F(y(t)) = F(y(a)) + t-a . 
\end{equation}
Since $F>0$ we must have $t > a-F(y(a))$. Then the solution can be expressed as
\[ y(t) = F^{-1}\left( t - (a-F(y(a)) \right) , \]
i.e. they are the translation by some constant in the argument of $t \mapsto F^{-1}(t)$, and we can check that any such function is also a solution of \eqref{eq:EDO-1} and \eqref{eq:EDO-2} with the desired assumptions. From now on, we consider only $y(t) = F^{-1}(t)$.

By series expansion, we get the asymptotics as $x\to 0$
\[ F(x) = F(1) + \int_x^1 \left(\frac{1}{t\sqrt{1+\frac{2t}3}} - \frac{1}{t}\right)\d t + \int_x^1 \frac{\d t}{t} = -\ln x + C + o(1) \]
for some constant $C$, i.e. $y(x) \sim C \e^{-t}$ as $x\to\infty$ for some $C>0$, and as $x\to\infty$, 
\begin{multline*}
F(x) = -\int_x^\infty \frac{\sqrt{3}\d t}{t^{3/2}\sqrt{2}} \left(1+\frac{3}{2t}\right)^{-1/2} = -\int_x^\infty \d t \sum_{n\geq 0} \frac{(-1)^n}{4^n} {{2n}\choose n} \left(\frac{2t}{3}\right)^{-n} t^{-3/2} \sqrt{\frac{3}{2}} \\
= \sqrt{6} \sum_{n\geq 0} {{2n}\choose n} \left(\frac{-3}{8}\right)^n \frac{ t^{-(n+1/2)}}{2n+1} = \sqrt{6} \left( x^{-1/2} - \frac{1}{4} x^{-3/2} + \frac{27}{160} x^{-5/2} - \frac{135}{896} x^{-7/2} + \dots \right) .
\end{multline*}
Let
\[ G(z) = \sum_{n\geq 0} {{2n}\choose n} \left(\frac{-3}{8}\right)^n \frac{z^n}{2n+1} = 1 - \frac{1}{4} z + \frac{27}{160} z^2 - \frac{135}{896} z^3 + \frac{315}{2048} z^4 - \frac{15309}{90112} z^5 + \dots \]
which is analytic on a neighborhood of $0$ and is such that $F(x) = \sqrt{\frac{6}{x}} G(x^{-1})$. Then $y(t) = \frac{6}{t^2} G(y(t)^{-1})^2$, which gives us that $y(t) - \frac{6}{t^2}$ is analytic near $0$ with
\begin{equation}\label{eq:psi-expansion}
y(t) - \frac{6}{t^2} = - \frac{1}{2} + \frac{t^2}{40} - \frac{t^4}{1008} + \frac{t^6}{28800} + \dots
\end{equation}

\subsection{Proof of the results on multitype branching random walks}
\label{sec:proof-multitype}

\subsubsection{Proof of Proposition \ref{prop:equiv-assum}}

We start with two very useful observation. First, for every $i\in\N_*$ 
the subtree $\{ u\in\TT_\emptyset, u\geq i\}$ is exactly $\widetilde\TT_i$. Second, $\widetilde\TT_\emptyset$ under $\E_y$ induces a \brw with branching process $\widetilde\BB = (\widetilde\BB^{(z)})_{z\in\kX}$ such that $\widetilde \BB^{(z)} = \BB^{(z)}$ (where $\BB = (\BB^{(z)})_{z\in\kX}$) if $z\neq x$ and $\BB^{(x)} = (\delta_0, \Lambda^{(x)}, D^{(x)}, \emptyset)$. This branching process is such that
\[ \E\left[\widetilde\chi^{(y)}(\R\times\{z\}) \right] = \wM_{y,z} .
\]
We claim that there exists $n_0>0$ such that $\wM^{n_0}$ has matrix norm strictly smaller than $1$, hence $I-\wM$ is invertible (where $I$ is the identity matrix), and if $v$ has non-negative coefficients then so has $(I-\wM)^{-1}v$ since $(I-wM)^{-1} v = (\sum_{n\geq 0} \wM^n) v$. Assume indeed that $v_z\neq 0$ for some $z\in\kX$, then by irreducibility of $M$ we can find $n_z$ such that $M^{n_z-1}_{z,x}>0$, and then $\| \wM^n v \|_1 < \| M^n v\|_1 \leq \| v \|_1$. Taking $n = \max_z n_z$ gives the claim.

In what follows, we successively prove that each item of Assumptions \ref{assum-multitype} holds for the \brw if and only if the item with the same number in Assumptions \ref{assum} holds for the reduced \brw, provided that all the previous items hold for the \brw.

We start by proving that 1. is equivalent to Assumption \ref{assum}.1:
\begin{align*}
\E_x\left[ \LL_\emptyset(\R) \right] = \E_x\left[ \sum_{i=1}^{\chi^{(x)}_\emptyset(\R\times\kX)} \E_{\ke_i}\left[\widetilde\LL_\emptyset(\R) \right] \right] = \sum_{y\in\kX} M_{x,y} \E_y\left[\widetilde\LL_\emptyset(\R)\right] ,
\end{align*}
$\E_x[\widetilde\LL_\emptyset(\R)] = 1$, 
and for every $y\neq x$
\begin{align*}
\E_y\left[ \widetilde\LL_\emptyset(\R) \right] = \E_y\left[ \sum_{i=1}^{\widetilde\chi^{(y)}_\emptyset(\R\times\kX)} \E_{\ke_i}\left[\widetilde\LL_\emptyset(\R) \right] \right] = \sum_{z\in\kX} \wM_{y,z} \E_y\left[\widetilde\LL_\emptyset(\R)\right] .
\end{align*}
In other words, $\be = (\E_z[\widetilde\LL_\emptyset(\R)])_{z\in\kX}$ satisfies the equation $\be = \wM \be + 1_x$, where $1_x(y) = 1$ if $x=y$ and $0$ otherwise. On the other hand, $(b_y/b_x)_{y\in\kX}$ also satisfies this equation, where we recall that $\bb$ is the right eigenvector of $M$. The solution is unique by invertibility of $I-\wM$, hence for every $y\in\kX$
\begin{equation}
\E_y\left[\widetilde\LL_\emptyset(\R)\right] = \frac{b_y}{b_x} .
\end{equation}
This immediately gives us that $\E_x[\LL_\emptyset(\R)] = (M\be)_x$ equals the largest eigenvalue of $M$. The equivalence between 1. and Assumption \ref{assum}.1 follows.

Let us now consider 2. Define 
\[ \be_y^{(1)} = \E_y\left[\int t \d\widetilde\LL_\emptyset(t)\right]  \]
and recall $N_{y,z}$. 
Then
\begin{align*}
\E_x\left[ \int t \d\LL_\emptyset(t) \right] &= \E_x\left[ \sum_{1\leq i \leq \chi^{(x)}_\emptyset(\R\times\kX)} \int (X^{(x)}_i + t)\d\widetilde\LL_i(t) \right] 
= \sum_{z\in\kX} \be_z N_{x,z} + \sum_{z\in\kX} M_{x,z} \be^{(1)}_z ,
\end{align*}
$\be^{(1)}_x = 0$ and for every $y\neq x$ 
\begin{align*}
\be^{(1)}_y &= \sum_{z\in\kX} \be_z N_{y,z} + \sum_{z\in\kX} M_{y,z} \be^{(1)}_z 
\end{align*}
i.e.
$ ((I-\wM)\be^{(1)})_y = (N\be)_y $. The mean displacement $(N\be + M\be^{(1)})_x$ of the reduced \brw is thus zero if and only if the condition in 2. holds.

Consider now 3; for every $y\in\kX$ define $\bf_y = \E_y\left[ \widetilde\LL_\emptyset(\R)^2 \right]$. 
Then 
\begin{align*}
\E_x\left[\LL_\emptyset(\R)^2\right] &= \E_x\left[\left(\sum_{i=1}^{\chi^{(x)}_\emptyset(\R\times\kX)} \widetilde\LL_i(\R)\right)^2\right] 
= \E_x\left[\sum_{i=1}^{\chi^{(x)}_\emptyset(\R\times\kX)} \bf_{\kx_i} + \sum_{1\leq i,j\leq\chi^{(x)}_\emptyset(\R\times\kX)} \be_{\kx_i} \be_{\kx_j} \right] ,
\end{align*}
$\bf_x = 1$, and for every $y\neq x$
\begin{multline*}
\bf_y = \sum_{z\in\kX} M_{y,z} \bf_z + \sum_{z\in\kX} \E\left[\chi^{(y)}_\emptyset(\R\times\{z\}) \left(\chi^{(y)}_\emptyset(\R\times\{z\})-1\right)\right] \left(\be_z\right)^2 \\
+ \sum_{z\neq w\in\kX} \E_y\left[\chi^{(y)}_\emptyset(\R\times\{z\}) \chi^{(y)}_\emptyset(\R\times\{w\})\right] \be_z \be_w .
\end{multline*}
Then $\E_x[\LL_\emptyset(\R)^2]$ is finite if and only if $\E[\chi^{(y)}_\emptyset(\R,\kX)^2]<\infty$ for every $y\in\kX$. Consider also that, writing $\tilde\bf_z = \bf_z - (\be_z)^2$,
\begin{multline*}
\bf_y = \sum_{z\in \bX} M_{y,z} \left(\bf_z - (\be_z)^2\right) + \sum_{z,w\in\kX} \mathrm{Cov}(\chi^{(y)}(\R\times\{z\}), \chi^{(y)}(\R\times\{w\})) \be_z \be_w \\
+ \sum_{z,w\in\kX} \left(\E[\chi^{(y)}(\R\times\{z\})]\be_z\right)\left(\E[\chi^{(y)}(\R\times\{w\})]\be_w\right) 
\end{multline*}
i.e.
\[ \tilde\bf_y = (M\tilde\bf)_y + (\be^t C^{(y)} \be) \]
where $C^{(y)}_{z,w} = \mathrm{Cov}(\chi^{(y)}(\R\times\{z\}), \chi^{(y)}(\R\times\{w\}))$. Rewriting $\E_x[\LL_\emptyset(\R)^2]$ in the same manner we can then see that $\tilde\bf_y \neq 0$ for some $y$, and that the variance of $\LL_\emptyset(\R)$ is non-zero, if and only if at least one of $C^{(y)}$ has a non-zero entry, which is the case if $\Var(\chi^{(y)}(\R\times\{z\}))$ is non-zero for at least one pair $y,z\in\kX$. This is already guaranteed by the fact that the process is non-degenerate: let us see why by contradiction. If the $\chi^{(y)}(\R\times\{z\})$ are constant r.v. then $M$ must have integer entries. The sum of the coefficient of each line must be non-zero (or $M$ is not irreducible), so they must be at least $1$; but then it is easy to see that they must be exactly one for $M$ to have largest eigenvalue $1$, meaning that for every $y$ we have $\chi^{(y)}(\R\times\kX) = 1$ a.s., and the process is degenerate.

For 4, define
\[ \be^{(2)}_y = \E_y\left[\int t^2 \d\widetilde\LL_\emptyset(t)\right] \]
and recall $O_{y,z}$. 
As for point 2., 
\[ \E_x\left[ \int t^2 \d\LL_\emptyset(t) \right] = \sum_{z\in\kX} O_{x,z} \be_z + \sum_{z\in\kX} N_{x,z} \be^{(1)}_z + \sum_{z\in\kX} M_{x,z} \be^{(2)}_z , \]
$\be^{(2)}_x = 1$ and for every $y\neq x$
\[ ((I-\wM) \be^{(2)})_y = \sum_{z\in\kX} O_{y,z} \be_z + \sum_{z\in\kX} N_{y,z} \be^{(1)}_z . \]
Then $\E_x\left[\int t^2\d\LL_\emptyset(t)\right]<\infty$ if and only if $\E\left[\int t^2 \d\chi^{(y)}(t,\kx)\right]<\infty$ for every $y\in\kX$.

For 5, let $\tilde b$ and $\tilde a$ be the right, resp. left eigenvector of $\wM$ with the maximal eigenvalue $\rho<1$. 
Let $v_x(r) = \tilde b_x^{-1} \P_x(\sup_{u\in\TT^\circ_\emptyset} \Lambda_u > r)$. We have
\begin{align*}
1-\tilde b_y v_y(r) &= \E_y\left[\ind{\Lambda^{(y)}_{\emptyset}\leq r} \prod_{1\leq j \leq\chi^{(y)}_\emptyset(\R\times\kX)} (1-\tilde b_{\kx^{(y)}_{\emptyset,j}} v_{\kx_{\emptyset,j}}(r-X^{(y)}_{\emptyset,j})) \right] \\
&\geq 1 - \P(\Lambda^{(y)}>r) + \E\left[ \sum_j \ln(1-\tilde b_{\kx^{(y)}_{\emptyset,j}} v_{\kx_{\emptyset,j}}(r-X^{(y)}_{\emptyset,j})) \right] \\
&\geq 1 - \P(\Lambda^{(y)}>r) - \E\left[ \sum_j \tilde b_{\kx^{(y)}_{\emptyset,j}} v_{\kx_{\emptyset,j}}(r-X^{(y)}_{\emptyset,j}) \right] .
\end{align*}
This means that, writing $v(r) = \sum_{y\in\kX} \tilde a_y \tilde b_y v_y(r)$, 
\begin{align*}
v(r) &\leq \sum_{y\in\kX} \tilde a_y \P_y(\Lambda>r) + \sum_{y,z\in\kX} \tilde a_y m_{y,z} \tilde b_z \E\left[v_z(r-Y_{y,z})\right] .
\end{align*}
where for every $y,z\in\kX$ and every positive and measurable $f$ we have $\E[f(Y_{y,z})] = m_{y,z}^{-1} \E[\int f(t) \ind{\kx=z} \d\chi^{(y)}(t,\kx)]$. We can find $\bY$ that bounds all the $Y_{y,z}$ stochastically, i.e. for every $y,z\in\kX$, $\P(Y_{y,z}\geq t) \leq \P(\bY\geq t)$. Then since $v_z$ is decreasing for every $z$, 
\begin{align*}
v(r) &\leq \sum_{y\in\kX} \tilde a_y \P_y(\Lambda>r) + \sum_{y,z\in\kX} \tilde a_y m_{y,z} \tilde b_z \E\left[v_z(r-\bY)\right] \\
&\leq \sum_{y\in\kX} \tilde a_y \P_y(\Lambda>r) + \rho \E\left[v(r-\bY)\right] .
\end{align*}
Since $v(r) \leq \sum_y \tilde a_y \tilde b_y$, we can bound
\[ \E[v(r-Y)] \leq v(r(1-\varepsilon))] + \P(\bY\geq \varepsilon r) \sum_y \tilde a_y \tilde b_y = v(r(1-\varepsilon)) + o(r^{-4}) . \]
Let $M(r) = r^4 v(r)$: then assuming that for every $y$, $\P(\Lambda^{(y)}>r) = o(r^{-4})$,
\[ M(r) \leq o(1) + \rho M(r(1-\varepsilon)) . \]
We easily deduce from this that $M(r)\to 0$ as $r\to\infty$, hence $v(r) = o(r^{-4})$ and $\P_x(\Lambda^\T_\emptyset > r) = o(r^{-4})$. 
On the other hand, by the irreducibility of $M$, for every $z$ we have $\P(\exists u\in \TT_\emptyset, t(u) = z)>0$, from which we may get that 
\[ \liminf_{r\to\infty} \frac{\P(\sup_{u\in\TT_\emptyset^\circ} \Lambda_u>r)}{\sup_{z\in\kX} \P(\Lambda^{(z)}>r)} > 0 , \]
hence $\P(\Lambda^\T_\emptyset >r)=o(r^{-4})$ if and only if $\P(\Lambda^{(y)}>r)=o(r^{-4})$ for every $y\in\kX$.

Finally, let us check 6.: letting $\bd_y = \E_y\left[\sum_{v\in\widetilde\TT^\circ_\emptyset} D_v\right]$, then $\bd_x = 0$ and 
\[ \E_x\left[\sum_{v\in\TT^\circ_\emptyset} D_v\right] = \E[D^{(x)}] + (M\bd)_x 
\quad , \quad 
\bd_y = \E[D^{(y)}] + (\wM \bd)_y .
\]
We easily conclude that $\E_x\left[\sum_{v\in\TT^\circ_\emptyset} D_v\right]<\infty$ if and only if $\E[D^{(y)}]<\infty$ for every $y\in\kX$, and that it is positive if and only if one of the $\E[D^{(y)}]$ is positive.

\subsubsection{Proof of Lemma \ref{lem:simpler-assum}}

The first statement is immediate from Assumption \ref{assum-multitype}.2. 
To prove the second, recall
\begin{align*}
\E_x\left[\int t^2 \d\LL_\emptyset(t) \right] &= \E_x\left[\sum_{j} \int (X_j+t)^2 \d\widetilde\LL_j(t) \right] \\
&= \sum_{z\in\kX} M_{x,z} \E_z\left[ \int (X_\emptyset+t)^2 \d\widetilde\LL_\emptyset(t) \right] .
\end{align*}
Letting $X_{y,z}$ be a random variable such that $M_{y,z} \E[f(X_{y,z})] = \E_y\left[\sum_{i: \kx_i = z} f(X_i)\right]$, 
\begin{align*}
\E_x\left[\int t^2 \d\LL_\emptyset(t) \right] &= \sum_{z\in\kX} M_{x,z} \left( \E[X_{x,z}^2] \E_z[\widetilde\LL_\emptyset(\R)] + 2 \E[X_{x,z}] \E_z\left[\int t \d\widetilde\LL_\emptyset(t)\right] + \E_z\left[\int t^2 \d\widetilde\LL_\emptyset(t)\right] \right) . 
\end{align*}
Since 
\begin{equation}\label{eq:constant-Xxz}
\E[X_{x,z}^2] \geq \E[X_{x,z}]^2
\end{equation}
and
\begin{equation}\label{eq:constant-LLempty}
\E_z\left[\int t^2 \d\widetilde\LL_\emptyset(t)\right]\E_z\left[\widetilde\LL_\emptyset(\R)\right] \geq \E_z\left[\int t \d\widetilde\LL_\emptyset(t)\right]^2 
\end{equation}
we have 
\begin{align*}
\E_x\left[\int t^2 \d\LL_\emptyset(t) \right] 
&\geq \sum_{z\in\kX} \frac{M_{x,z}}{\E_z\left[\widetilde\LL_\emptyset(\R)\right]} \left( \E_z[X_{x,z}] \E_z[\widetilde\LL_\emptyset(\R)] + \E_z\left[\int t \d\widetilde\LL_\emptyset(t)\right] \right)^2
\end{align*}
with equality if and only if the equality hold in \eqref{eq:constant-Xxz} and \eqref{eq:constant-LLempty}. For the variance to be zero, we see that we need $X_{x,z}$ to be a.s. constant, and for the variance of the mean measure of $\LL_\emptyset$ to be zero; a similar reasoning as above for $\E_y[\int t^2\d\widetilde\LL_\emptyset(t)]$ implies that we must have $X_{y,z}$ a.s. constant for every $y,z\in\kX$. Note however that this is not a sufficient condition: indeed we could have $X_{y,z}$ a.s. constant for every $y,z$ and still a non-zero variance for $\int t^2\d\LL_\emptyset(t)$ under $\E_x$.

\subsection{Proof of the result of Boltzmann planar maps}
\label{sec:prop-mobile}

Note first that the \repscheme is non-degenerate: indeed $\mu_V$ is a geometric distribution with non-zero expectation hence is not a.s. constant. 

Compute
\begin{align*}
\E[N_F+1] = \sum_{k\geq 0} (k+1) \mu_F(k) &= \frac{1}{1-Z_\bq^{-1}} \sum_{j\geq 1} j {{2j-1}\choose j} \bq_j Z_\bq^{j-1} 
= \frac{1}{1-Z_\bq^{-1}} f_\bq'(Z_\bq) = \frac{Z_\bq}{Z_\bq-1} ,
\end{align*} 
\begin{align*}
\E[N_F^2 + N_F] = \sum_{k\geq 0} k(k+1) \mu_F(k) &= \frac{1}{1-Z_\bq^{-1}} \sum_{j\geq 1} j(j-1) {{2j-1}\choose j} \bq_j Z_\bq^{j-1} 
= \frac{Z_\bq}{1-Z_\bq^{-1}} f_\bq''(Z_\bq) ,
\end{align*} 
hence 
\[ \E[N_F] = \frac{Z_\bq^{-1}}{1-Z_\bq^{-1}} \quad , \quad \Var(N_F) \leq \frac{Z_\bq}{1-Z_\bq^{-1}} f_\bq''(Z_q). \]
On the other hand, 
\[ \E[N_V] = Z_\bq(1-Z_\bq^{-1}) \quad , \quad \Var(N_V) = Z_\bq^2(1-Z_\bq^{-1}) . \]
The mean reproduction matrix is then
\[ \left(\begin{matrix} 0 & \E[N_F] \\ \E[N_V] & 0 \end{matrix}\right) \]
which is irreducible, has finite coefficients, and has maximal eigenvalue $1$. This proves Assumption \ref{assum-multitype}.1 as well as the preliminary conditions before Assumptions \ref{assum-multitype}.

We use Lemma \ref{lem:simpler-assum} to show 2.: the only non-trivial task is to check that $\E[\sum_{k\in\DD} b_k] = 0$, where $(b_k)_{0\leq k < 2N_F+2}$ are distributed as in the definition of $\BB_F$. We show that this holds conditionally on $N_F$. Fix $n>0$ such that $\P(N_F=n)>0$, and work conditionally on the event $\{N_F = n\}$. Fix $0<k<2n+2$. Then if $(c_j)_{0\leq j \leq 2n}$ is a uniform bridge with $c_0 = -1, c_{2n} = +1$ and steps in $\pm 1$, we have 
\[ \P(b_k = b \ | \ k\in\DD, N_F=n) = \P(c_{k-1}=b) . \]
Indeed, conditioned on $b_0 = 0$, $b_1 = -1$, $b_k = b$ and $b_{2n+2}=0$, the process $(c_j)_{0\leq j \leq 2n}$ with $c_j = b_{j+1}$ for every $0\leq j \leq k-1$ and $c_j = b_{j+2}+1$ for $k\leq j \leq 2n$ has the desired distribution. On the other hand, since the $n$ downsteps of $(b_k)_{0\leq k < 2n+2}$ (excluding 0) are uniformly distributed among the $2n+1$ indices $j$ with $1\leq j <2n+2$, we have $\P(k\in\DD \ | \ N_F=n) = \frac{n}{2n+1}$. Hence, for every $1\leq k \leq 2n+1$, 
\[ \E[b_k \ind{k\in\DD} \ | \ N_F=n] = \frac{n}{2n+1} \E[c_{k-1}] , \]
hence 
\[ \E\left[\sum_{k\in\DD} b_k \ | \ N_F=n\right] = \frac{n}{2n+1} \E\left[ \sum_{0\leq j \leq 2n} c_j \right] . \]
Since $(c_{2n-j})_{0\leq j\leq 2n}$ has the same distribution as $(-c_j)_{0\leq j \leq 2n}$, this latter expectation is zero, proving point 2.

Condition 3 is immediate for $\mu_V$, and follows from the generic criticality of $\bq$ for $\mu_F$.


The finiteness of the variance in point 4. follows from a control over $\E[\sum_{k\in\DD} b_k^2]$. We follow the proof of point 2. Conditionally on $N_F = n$, where $\P(N_F=n)>0$, for every $1\leq k \leq 2n+1$, 
\[ \E[b_k^2 \ind{k\in\DD} \ | \ N_F=n] = \frac{n}{2n+1} \E[c_{k-1}^2] . \]
The distribution of $\frac 12 (c_j+j+1)$ is hypergeometric with population size $2n$, $n+1$ successes, and $j$ draws, with variance
\[ \frac{n^2-1}{4n^2}\frac{j(2n-j)}{2n-1} . \]
On the other hand, $\E[c_j] = \frac{j}{n}-1$. Thus
\begin{align*}
\E\left[\sum_{k\in\DD} b_k^2 \ | \ N_F=n\right] &= \frac{n^2-1}{n(2n-1)(2n+1)} \sum_{j=0}^{2n} j(2n-j) + \frac{1}{n(2n+1)} \sum_{j=0}^{2n} (j-n)^2  \\
&= \frac{n^2-1}{n(2n-1)(2n+1)} \frac{2n(2n+1)(2n-1)}{6} + \frac{2}{n(2n+1)} \frac{n(n+1)(2n+1)}{6}  \\
&= \frac{n^2-1}{3} + \frac{n+1}{3} = \frac{n^2 + n}{3} .
\end{align*}
Hence
$ \E[ \sum_{k\in\DD} b_k^2 ] < \infty$. 
To show that the variance is non-zero, we use Lemma \ref{lem:simpler-assum}: conditionally on $N_F\geq 1$ (which has non-zero probability) we have at least two downsteps in $\DD$, and with positive probability at least one of them takes a non-zero value. The second assertion of Lemma \ref{lem:simpler-assum} follows easily. We can explicitely compute

Point 6. is immediate, so the only point left to check is point 5. By the remark after Assumptions \ref{assum}, it is enough to check that $\E[(\sup_k b_k)^4]<\infty$. 
Conditioning on $N_F=n$ with $\P(N_F=n)>0$, by the reflection principle, the probability that the maximum of $(b_k)_{0\leq k \leq 2n+2}$ is larger or equal to some $m\geq 0$ is
\[  \frac{ {{2n+1} \choose {n+1+m}} }{ {{2n+1} \choose {n+1}} } = \prod_{k=1}^m \frac{n+k-m}{n+k} \leq \left( 1 - \frac{m}{2n} \right)^m \leq \e^{-\frac{m^2}{2n}} . \]
Hence 
\begin{align*}
\E\left[ \sup_{0\leq k \leq 2n+2} b_k^4 \ | \ N_F = n \right] &= \sum_{m>0} m^3 \P\left( \sup_{0\leq k \leq 2n+2} b_k \geq m \right) \\
&\leq 3^{\frac{3}{2}}\e^{-\frac{3}{2}} n^{\frac{3}{2}} + \int_0^\infty x^3 \e^{-\frac{x^2}{2n}} \d x  \\
&= 3^{\frac{3}{2}}\e^{-\frac{3}{2}} n^{\frac{3}{2}} + 2n \int_0^\infty x \e^{-\frac{x^2}{2n}} \d x  \\
&= 3^{\frac{3}{2}}\e^{-\frac{3}{2}} n^{\frac{3}{2}} + 2n^2 \leq 4 n^2 ,
\end{align*}
hence $\E[(\sup_k b_k)^4]<\infty$, finishing the proof. 

We can compute the explicit values of Proposition \ref{prop:mobile}: first
\begin{align*}
\eta^2 &= \E[N_V] \E\left[\frac{N_F^2 + N_F}{3}\right] 
= \frac{Z_\bq^2}{3} f_\bq''(Z_q) .
\end{align*}
Then letting $(N_F^{(i)})_{i\geq 1}$ be i.i.d. copies of $N_F$, 
\begin{align*}
1+\sigma^2 &= \E\left[\left(\sum_{1\leq i\leq N_V} N_F^{(i)} \right)^2\right] \\
&= \E[N_V] \E[N_F^2] + \E[N_V(N_V-1)] \E[N_F]^2 \\
&= Z_\bq (1-Z_\bq^{-1}) \left( \frac{Z_\bq}{1-Z_\bq^{-1}} f''_\bq(Z_\bq) - \frac{Z_\bq^{-1}}{1-Z_\bq^{-1}} \right) + Z_\bq (1-Z_\bq^{-1}) \left( Z_\bq + Z_\bq (1-Z_\bq^{-1}) - 1 \right) \frac{Z_\bq^{-2}}{(1-Z_\bq^{-1})^2} \\
&= Z_\bq^2 f''_\bq(Z_\bq) - 1 + 2 Z_\bq^2 (1-Z_\bq^{-1})^2 \frac{Z_\bq^{-2}}{(1-Z_\bq^{-1})^2} 
\end{align*}
so that $\sigma^2 = 3\eta^2$. 
Finally, $\E[D]$ is given by $\E[N_V] = Z_\bq-1$ if counting faces, $1$ if counting vertices, and $1+\E[N_V] = Z_\bq$ if counting vertices and faces (i.e. edges by Euler's formula).

\subsection*{Acknowledgements}

The author would like to thank Alejandro Rosalez Ortiz, Armand Riera, Grégory Miermont, Luc Lehéricy, Pascal Maillard and Michel Pain for helpful discussions. 

\bibliographystyle{apalike} 
\bibliography{BibTotale.bib}

\end{document}